\documentclass[11pt,twoside,a4paper]{article}
\usepackage{color,amscd,amsmath,amssymb,amsthm,latexsym,stmaryrd,authblk,mathabx,shuffle,hyperref,tikz,tkz-tab,wasysym} 
\usepackage[latin1]{inputenc}
\usepackage[T1]{fontenc} 
\usepackage[english]{babel}
\providecommand{\keywords}[1]{\textbf{\textit{Keywords.}} #1}
\providecommand{\AMSclass}[1]{\textbf{\textit{AMS classification.}} #1}

\newcommand{\rond}[1]{*++[o][F-]{#1}}

\newcommand{\calT}{\mathcal{T}}
\newcommand{\calD}{\mathcal{D}}
\newcommand{\g}{\mathfrak{g}}
\newcommand{\K}{\mathbb{K}}

\newcommand{\adm}{\mathrm{Adm}}

\newcommand{\prelie}{\mathbf{PreLie}}
\newcommand{\free}{\mathbf{F}}
\newcommand{\perm}{\mathbf{Perm}}
\newcommand{\NAP}{\mathbf{NAP}}
\newcommand{\im}{\mathrm{Im}}
\newcommand{\vect}{\mathrm{Vect}}
\renewcommand{\ker}{\mathrm{Ker}}

\newcommand{\eas}{\mathbf{EAS}}
\newcommand{\leas}{\ell\mathbf{EAS}}
\newcommand{\id}{\mathrm{Id}}
\newcommand{\indic}{\mathbf{1}}
\newcommand{\bfT}{\mathcal{T}}
\newcommand{\calP}{\mathcal{P}}

\input{xy}
\xyoption{all}

\newcommand{\bfdeux}[2]{
\begin{tikzpicture}[line cap=round,line join=round,>=triangle 45,x=0.2cm,y=0.2cm]
\clip(-1.2,0.) rectangle (1.2,3.2);
\draw [line width=.5pt] (0.,0.)-- (0.,1.);
\draw [line width=.5pt] (0.,1.)-- (-1.,2.);
\draw [line width=.5pt] (0.,1.)-- (1.,2.);
\draw(-1.,2.5) node {\tiny #1};
\draw(1.,2.5) node {\tiny #2};
\end{tikzpicture}}

\newcommand{\bftroisun}[3]{
\begin{tikzpicture}[line cap=round,line join=round,>=triangle 45,x=0.2cm,y=0.2cm]
\clip(-2.2,0.) rectangle (2.2,4.2);
\draw [line width=.5pt] (0.,0.)-- (0.,1.);
\draw [line width=.5pt] (0.,1.)-- (-1.,2.);
\draw [line width=.5pt] (0.,1.)-- (2.,3.);
\draw [line width=.5pt] (-1.,2.)-- (0.,3.);
\draw [line width=.5pt] (-1.,2.)-- (-2.,3.);
\draw(-2.,3.5) node {\tiny #1};
\draw(0.,3.5) node {\tiny #2};
\draw(2.,3.5) node {\tiny #3};
\end{tikzpicture}}
\newcommand{\bftroisdeux}[3]{
\begin{tikzpicture}[line cap=round,line join=round,>=triangle 45,x=0.2cm,y=0.2cm]
\clip(-2.2,0.) rectangle (2.2,4.2);
\draw [line width=.5pt] (0.,0.)-- (0.,1.);
\draw [line width=.5pt] (0.,1.)-- (-2.,3.);
\draw [line width=.5pt] (0.,1.)-- (1.,2.);
\draw [line width=.5pt] (1.,2.)-- (2.,3.);
\draw [line width=.5pt] (1.,2.)-- (0.,3.);
\draw(-2.,3.5) node {\tiny #1};
\draw(0.,3.5) node {\tiny #2};
\draw(2.,3.5) node {\tiny #3};
\end{tikzpicture}}

\newcommand{\tun}{\begin{tikzpicture}[line cap=round,line join=round,>=triangle 45,x=0.5cm,y=0.5cm]
\clip(-0.2,-0.1) rectangle (0.2,0.2);
\begin{scriptsize}
\draw [fill=black] (0.,0.) circle (1pt);
\end{scriptsize}
\end{tikzpicture}}

\newcommand{\tdeux}{\begin{tikzpicture}[line cap=round,line join=round,>=triangle 45,x=0.5cm,y=0.5cm]
\clip(-.2,-.1) rectangle (0.2,0.7);
\draw [line width=.5pt] (0.,0.5)-- (0.,0.);
\begin{scriptsize}
\draw [fill=black] (0.,0.) circle (1pt);
\draw [fill=black] (0.,0.5) circle (1pt);
\end{scriptsize}
\end{tikzpicture}}

\newcommand{\ttroisun}{\begin{tikzpicture}[line cap=round,line join=round,>=triangle 45,x=0.5cm,y=0.5cm]
\clip(-0.5,-0.1) rectangle (0.5,0.7);
\draw [line width=0.5pt] (0.,0.)-- (-0.3,0.5);
\draw [line width=0.5pt] (0.,0.)-- (0.3,0.5);
\begin{scriptsize}
\draw [fill=black] (-0.3,0.5) circle (1pt);
\draw [fill=black] (0.,0.) circle (1pt);
\draw [fill=black] (0.3,0.5) circle (1pt);
\end{scriptsize}
\end{tikzpicture}}
\newcommand{\ttroisdeux}{\begin{tikzpicture}[line cap=round,line join=round,>=triangle 45,x=0.5cm,y=0.5cm]
\clip(-.2,-.1) rectangle (0.2,1.2);
\draw [line width=0.5pt] (0.,0.5)-- (0.,0.);
\draw [line width=0.5pt] (0.,0.5)-- (0.,1.);
\begin{scriptsize}
\draw [fill=black] (0.,0.) circle (1pt);
\draw [fill=black] (0.,0.5) circle (1pt);
\draw [fill=black] (0.,1.) circle (1pt);
\end{scriptsize}
\end{tikzpicture}}

\newcommand{\tquatreun}{\begin{tikzpicture}[line cap=round,line join=round,>=triangle 45,x=0.5cm,y=0.5cm]
\clip(-0.5,-0.1) rectangle (0.5,0.7);
\draw [line width=0.5pt] (0.,0.)-- (-0.3,0.5);
\draw [line width=0.5pt] (0.,0.)-- (0.3,0.5);
\draw [line width=0.5pt] (0.,0.)-- (0.,0.5);
\begin{scriptsize}
\draw [fill=black] (-0.3,0.5) circle (1.0pt);
\draw [fill=black] (0.,0.) circle (1.0pt);
\draw [fill=black] (0.3,0.5) circle (1.0pt);
\draw [fill=black] (0.,0.5) circle (1.0pt);
\end{scriptsize}
\end{tikzpicture}}
\newcommand{\tquatredeux}{\begin{tikzpicture}[line cap=round,line join=round,>=triangle 45,x=0.5cm,y=0.5cm]
\clip(-0.5,-0.1) rectangle (0.5,1.2);
\draw [line width=0.5pt] (0.,0.)-- (-0.3,0.5);
\draw [line width=0.5pt] (0.,0.)-- (0.3,0.5);
\draw [line width=0.5pt] (-0.3,0.5)-- (-0.3,1.);
\begin{scriptsize}
\draw [fill=black] (-0.3,0.5) circle (1.0pt);
\draw [fill=black] (0.,0.) circle (1.0pt);
\draw [fill=black] (0.3,0.5) circle (1.0pt);
\draw [fill=black] (-0.3,1.) circle (1.0pt);
\end{scriptsize}
\end{tikzpicture}}

\newcommand{\tquatrequatre}{\begin{tikzpicture}[line cap=round,line join=round,>=triangle 45,x=0.5cm,y=0.5cm]
\clip(-0.5,-0.1) rectangle (0.5,1.7);
\draw [line width=0.5pt] (0.,0.)-- (0.,0.5);
\draw [line width=0.5pt] (0.,0.5)-- (0.3,1.);
\draw [line width=0.5pt] (0.,0.5)-- (-0.3,1.);
\begin{scriptsize}
\draw [fill=black] (0.,0.) circle (1.0pt);
\draw [fill=black] (0.,0.5) circle (1.0pt);
\draw [fill=black] (-0.3,1.) circle (1.0pt);
\draw [fill=black] (0.3,1.) circle (1.0pt);
\end{scriptsize}
\end{tikzpicture}}
\newcommand{\tquatrecinq}{\begin{tikzpicture}[line cap=round,line join=round,>=triangle 45,x=0.5cm,y=0.5cm]
\clip(-.2,-.1) rectangle (0.5,1.7);
\draw [line width=0.5pt] (0.,0.)-- (0.,0.5);
\draw [line width=0.5pt] (0.,0.5)-- (0.,1.);
\draw [line width=0.5pt] (0.,1.)-- (0.,1.5);
\begin{scriptsize}
\draw [fill=black] (0.,0.) circle (1.0pt);
\draw [fill=black] (0.,0.5) circle (1.0pt);
\draw [fill=black] (0.,1.) circle (1.0pt);
\draw [fill=black] (0.,1.5) circle (1.0pt);
\end{scriptsize}
\end{tikzpicture}}

\newcommand{\tcinqun}{\begin{tikzpicture}[line cap=round,line join=round,>=triangle 45,x=0.5cm,y=0.5cm]
\clip(-0.7,-0.1) rectangle (0.8,0.7);
\draw [line width=0.5pt] (0.,0.)-- (-0.5,0.5);
\draw [line width=0.5pt] (0.,0.)-- (-0.2,0.5);
\draw [line width=0.5pt] (0.,0.)-- (0.2,0.5);
\draw [line width=0.5pt] (0.,0.)-- (0.5,0.5);
\begin{scriptsize}
\draw [fill=black] (-0.5,0.5) circle (1.0pt);
\draw [fill=black] (-0.2,0.5) circle (1.0pt);
\draw [fill=black] (0.2,0.5) circle (1.0pt);
\draw [fill=black] (0.5,0.5) circle (1.0pt);
\draw [fill=black] (0.,0.) circle (1.0pt);
\end{scriptsize}
\end{tikzpicture}}
\newcommand{\tcinqdeux}{\begin{tikzpicture}[line cap=round,line join=round,>=triangle 45,x=0.5cm,y=0.5cm]
\clip(-0.5,-0.1) rectangle (0.5,1.2);
\draw [line width=0.5pt] (0.,0.)-- (-0.3,0.5);
\draw [line width=0.5pt] (0.,0.)-- (0.3,0.5);
\draw [line width=0.5pt] (0.,0.)-- (0.,0.5);
\draw [line width=0.5pt] (-0.3,0.5)-- (-0.3,1.);
\begin{scriptsize}
\draw [fill=black] (-0.3,0.5) circle (1.0pt);
\draw [fill=black] (0.,0.) circle (1.0pt);
\draw [fill=black] (0.,0.5) circle (1.0pt);
\draw [fill=black] (0.3,0.5) circle (1.0pt);
\draw [fill=black] (-0.3,1.) circle (1.0pt);
\end{scriptsize}
\end{tikzpicture}}

\newcommand{\tcinqcinq}{\begin{tikzpicture}[line cap=round,line join=round,>=triangle 45,x=0.5cm,y=0.5cm]
\clip(-0.5,-0.1) rectangle (0.5,1.2);
\draw [line width=0.5pt] (0.,0.)-- (-0.3,0.5);
\draw [line width=0.5pt] (-0.3,0.5)-- (-0.3,1.);
\draw [line width=0.5pt] (0.,0.)-- (0.3,0.5);
\draw [line width=0.5pt] (0.3,0.5)-- (0.3,1.);
\begin{scriptsize}
\draw [fill=black] (-0.3,0.5) circle (1.0pt);
\draw [fill=black] (0.,0.) circle (1.0pt);
\draw [fill=black] (-0.3,1.) circle (1.0pt);
\draw [fill=black] (0.3,0.5) circle (1.0pt);
\draw [fill=black] (0.3,1.) circle (1.0pt);
\end{scriptsize}
\end{tikzpicture}}
\newcommand{\tcinqsix}{\begin{tikzpicture}[line cap=round,line join=round,>=triangle 45,x=0.5cm,y=0.5cm]
\clip(-0.7,-0.1) rectangle (0.5,1.2);
\draw [line width=0.5pt] (0.,0.)-- (-0.3,0.5);
\draw [line width=0.5pt] (0.,0.)-- (0.3,0.5);
\draw [line width=0.5pt] (-0.3,0.5)-- (-0.6,1.);
\draw [line width=0.5pt] (-0.3,0.5)-- (0.,1.);
\begin{scriptsize}
\draw [fill=black] (-0.3,0.5) circle (1.0pt);
\draw [fill=black] (0.,0.) circle (1.0pt);
\draw [fill=black] (0.3,0.5) circle (1.0pt);
\draw [fill=black] (-0.6,1.) circle (1.0pt);
\draw [fill=black] (0.,1.) circle (1.0pt);
\end{scriptsize}
\end{tikzpicture}}

\newcommand{\tcinqhuit}{\begin{tikzpicture}[line cap=round,line join=round,>=triangle 45,x=0.5cm,y=0.5cm]
\clip(-0.5,-0.1) rectangle (0.5,1.7);
\draw [line width=0.5pt] (0.,0.)-- (-0.3,0.5);
\draw [line width=0.5pt] (-0.3,0.5)-- (-0.3,1.);
\draw [line width=0.5pt] (0.,0.)-- (0.3,0.5);
\draw [line width=0.5pt] (-0.3,1.)-- (-0.3,1.5);
\begin{scriptsize}
\draw [fill=black] (-0.3,0.5) circle (1.0pt);
\draw [fill=black] (0.,0.) circle (1.0pt);
\draw [fill=black] (-0.3,1.) circle (1.0pt);
\draw [fill=black] (0.3,0.5) circle (1.0pt);
\draw [fill=black] (-0.3,1.5) circle (1.0pt);
\end{scriptsize}
\end{tikzpicture}}

\newcommand{\tcinqdix}{\begin{tikzpicture}[line cap=round,line join=round,>=triangle 45,x=0.5cm,y=0.5cm]
\clip(-0.5,-0.1) rectangle (0.5,1.7);
\draw [line width=0.5pt] (0.,0.)-- (0.,0.5);
\draw [line width=0.5pt] (0.,0.5)-- (0.3,1.);
\draw [line width=0.5pt] (0.,0.5)-- (-0.3,1.);
\draw [line width=0.5pt] (0.,0.5)-- (0.,1.);
\begin{scriptsize}
\draw [fill=black] (0.,0.) circle (1.0pt);
\draw [fill=black] (0.,0.5) circle (1.0pt);
\draw [fill=black] (-0.3,1.) circle (1.0pt);
\draw [fill=black] (0.3,1.) circle (1.0pt);
\draw [fill=black] (0.,1.) circle (1.0pt);
\end{scriptsize}
\end{tikzpicture}}
\newcommand{\tcinqonze}{\begin{tikzpicture}[line cap=round,line join=round,>=triangle 45,x=0.5cm,y=0.5cm]
\clip(-0.5,-0.1) rectangle (0.5,1.7);
\draw [line width=0.5pt] (0.,0.)-- (0.,0.5);
\draw [line width=0.5pt] (0.,0.5)-- (0.3,1.);
\draw [line width=0.5pt] (0.,0.5)-- (-0.3,1.);
\draw [line width=0.5pt] (-0.3,1.)-- (-0.3,1.5);
\begin{scriptsize}
\draw [fill=black] (0.,0.) circle (1.0pt);
\draw [fill=black] (0.,0.5) circle (1.0pt);
\draw [fill=black] (-0.3,1.) circle (1.0pt);
\draw [fill=black] (0.3,1.) circle (1.0pt);
\draw [fill=black] (-0.3,1.5) circle (1.0pt);
\end{scriptsize}
\end{tikzpicture}}
\newcommand{\tcinqdouze}{\begin{tikzpicture}[line cap=round,line join=round,>=triangle 45,x=0.5cm,y=0.5cm]
\clip(-0.5,-0.1) rectangle (0.5,1.7);
\draw [line width=0.5pt] (0.,0.)-- (0.,0.5);
\draw [line width=0.5pt] (0.,0.5)-- (0.3,1.);
\draw [line width=0.5pt] (0.,0.5)-- (-0.3,1.);
\draw [line width=0.5pt] (0.3,1.)-- (0.3,1.5);
\begin{scriptsize}
\draw [fill=black] (0.,0.) circle (1.0pt);
\draw [fill=black] (0.,0.5) circle (1.0pt);
\draw [fill=black] (-0.3,1.) circle (1.0pt);
\draw [fill=black] (0.3,1.) circle (1.0pt);
\draw [fill=black] (0.3,1.5) circle (1.0pt);
\end{scriptsize}
\end{tikzpicture}}

\newcommand{\tcinqquatorze}{\begin{tikzpicture}[line cap=round,line join=round,>=triangle 45,x=0.5cm,y=0.5cm]
\clip(-.2,-.1) rectangle (0.5,2.2);
\draw [line width=0.5pt] (0.,0.)-- (0.,0.5);
\draw [line width=0.5pt] (0.,0.5)-- (0.,1.);
\draw [line width=0.5pt] (0.,1.)-- (0.,1.5);
\draw [line width=0.5pt] (0.,1.5)-- (0.,2.);
\begin{scriptsize}
\draw [fill=black] (0.,0.) circle (1.0pt);
\draw [fill=black] (0.,0.5) circle (1.0pt);
\draw [fill=black] (0.,1.) circle (1.0pt);
\draw [fill=black] (0.,1.5) circle (1.0pt);
\draw [fill=black] (0.,2.) circle (1.0pt);
\end{scriptsize}
\end{tikzpicture}}

\title{Generalized prelie and permutative algebras}
\date{}
\author{Lo\"\i c Foissy}
\affil{\small{Univ. Littoral Côte d'Opale, UR 2597
LMPA, Laboratoire de Mathématiques Pures et Appliquées Joseph Liouville
F-62100 Calais, France}.\\ Email:\texttt{foissy@univ-littoral.fr}}

\theoremstyle{plain}
\newtheorem{theo}{Theorem}[section]
\newtheorem{lemma}[theo]{Lemma}
\newtheorem{cor}[theo]{Corollary}
\newtheorem{prop}[theo]{Proposition}
\newtheorem{defi}[theo]{Definition}

\theoremstyle{remark}
\newtheorem{remark}{Remark}[section]
\newtheorem{notation}{Notations}[section]
\newtheorem{example}{Example}[section]

\begin{document}

\maketitle

\begin{abstract}
We study generalizations of pre-Lie algebras, where the free objects are based on rooted trees which edges are typed, instead of usual rooted trees, and with generalized pre-Lie products formed by graftings. 
Working with a discrete set of types, we show how to obtain such objects when this set is given an associative commutative product and a second product
making it a commutative extended semigroup. Working with a vector space of types, these two products are replaced by a bilinear map $\Phi$ which satisfies a braid equation and a commutation relation. 
 Examples of such structures are defined on sets, semigroups, or groups.
 
 These constructions define a family of operads $\mathbf{PreLie}_\Phi$ which generalize the operad of pre-Lie algebras $\mathbf{PreLie}$.
 For any embedding from $\mathbf{PreLie}$ into $\mathbf{PreLie}_\phi$, we construct a family of pairs of cointeracting bialgebras, based on typed and decorated trees: 
the first coproduct is given by an extraction and contraction process, the types being modified by the action of $\Phi$; the second coproduct is given by admissible cuts, in the Connes and Kreimer's way, with again types modified by the action of $\Phi$. 
 
 We also study the Koszul dual of $\mathbf{PreLie}_\Phi$, which gives generalizations of permutative algebras.
\end{abstract}

\\

\AMSclass{18M60, 16T30, 05C05, 20M75, 16T05}

\tableofcontents

\section*{Introduction}

\:\:\:\: Recently, numerous parameterization of well-known operads were introduced. Choosing a set $\Omega$ of parameters, any product defining the considered operad is replaced by a bunch of products indexed by $\Omega$,
and various relations are defined on them, mimicking the relations defining the initial operads.
One can first require that any linear span of the parameterized products also satisfy the relations of the initial operads this is the\emph{matching} parameterization. For example, matching Rota-Baxter algebras,
associative, dendriform, pre-Lie algebras are introduced in\cite{Zhang2020,Foissy47}.
Another way is the use of one or more semigroup structures on $\Omega$: this it the\emph{family} parameterization. For example, family Rota-Baxter algebras, dendriform, pre-Lie algebras are introduced and studied in\cite{Zhang2019,Zhang2020-3,Manchon2020}. 
A way to obtain both these parametrizations for dendriform algebras is introduced in\cite{Foissy58}, with the help of a generalization of diassociative semigroups, namely extended diassociative semigroups (EDS),
and a two-parameters version for dendriform algebras and pre-Lie algebras is described in\cite{Foissy50}. The same is done for associative algebras in\cite{Foissy56} and for tridendriform algebras in\cite{Foissy57}.\\

In this paper, we extend the parametrizations of pre-Lie algebras to a more general setting, and study the pairs of cointeracting bialgebras on decorated and typed rooted trees which result of these constructions. 
We start with a discrete version of these parametrizations. 
The set of parameters $\Omega$ is here given an associative product $\rightarrow$ and another (maybe non-associative) product $\triangleright$, satisfying the four axioms of commutative extended diassociative semigroup
(briefly, CEDS), see Definition\ref{defi1.1}. An $\Omega$-pre-Lie algebra is a pair $(V,(\circ_\alpha)_{\alpha\in\Omega})$, where $V$ is a vector space and for any $\alpha\in\Omega$, $\circ_\alpha:V\otimes V\longrightarrow V$
such that, for any $x,y,z\in V$, for any $\alpha,\beta\in\Omega$,
\begin{align*}
x\circ_\alpha (y\circ_\beta z)- (x\circ_{\alpha\triangleright\beta} y)\circ_{\alpha\rightarrow\beta}z&=y\circ_\beta (x\circ_\alpha z)- (y\circ_{\beta\triangleright\alpha} x)\circ_{\beta\rightarrow\alpha}z.
\end{align*}
In the particular case where $\Omega$ is reduced to a singleton, these are classical (left) pre-Lie algebras. There are other interesting examples:
\begin{itemize}
\item For any set $\Omega$, define the two products $\rightarrow,\triangleright$ by
\[\alpha\rightarrow\beta=\beta\triangleright\alpha=\beta.\]
Then $(\Omega,\rightarrow,\triangleright)$ is a CEDS, and $\Omega$-pre-Lie algebras are matching pre-Lie algebras of\cite{Zhang2020}.
\item Let $(\Omega,\rightarrow)$ be a commutative semigroup. Define the product $\triangleright$ by
\[\alpha\triangleright\beta=\alpha.\]
Then $(\Omega,\rightarrow,\triangleright)$ is a CEDS, and $\Omega$-pre-Lie algebras are family pre-Lie algebras of\cite{Manchon2020}.
\item Let $(\Omega,\star)$ be a group. Define the two products $\rightarrow,\triangleright$ by
\begin{align*}
\alpha\rightarrow\beta&=\beta,&\alpha\triangleright\beta&=\alpha\star\beta^{\star-1}.
\end{align*}
Then $(\Omega,\rightarrow,\triangleright)$ is a CEDS. The associated $\Omega$- pre-Lie algebras do not seem to appear in the literature. 
\end{itemize}
In order to be more general, we turn to a linearized version of CEDS, which is based on the following observation (Lemma\ref{lem1.4}): if $\Omega$ is a set with two operations $\rightarrow$ and $\triangleright$, then we can consider the maps
\begin{align*}
\phi&:\left\{\begin{array}{rcl}
\Omega^2&\longrightarrow&\Omega^2\\
(\alpha,\beta)&\longrightarrow&(\alpha\rightarrow\beta,\alpha\triangleright\beta),
\end{array}\right.&
\tau&:\left\{\begin{array}{rcl}
\Omega^2&\longrightarrow&\Omega^2\\
(\alpha,\beta)&\longrightarrow&(\beta,\alpha).
\end{array}\right.
\end{align*}
Then $(\Omega,\rightarrow,\triangleright)$ is a CEDS if, and only if
\begin{align*}
(\id\times\phi)\circ (\phi\times\id)\circ (\id\times\phi)&=(\phi\times\id)\circ (\id\times\tau)\circ (\phi\times\id),\\
(\id\times\phi)\circ (\id\times\tau)\circ (\tau\times\id)\circ (\phi\times\id)&=(\tau\times\id)\circ (\phi\times\id)\circ (\id\times\phi)\circ (\id\times\tau).
\end{align*}
The first equation is called the braid equation, the second one the commutation relation. 
This observation leads to the definition of linear CEDS, which are pairs $(A,\Phi)$, where $A$ is a vector space and $\Phi:A\otimes A\longrightarrow A\otimes A$, such that
\begin{align*}
(\id\otimes\Phi)\circ (\Phi\otimes\id)\circ (\id\otimes\Phi)&=(\Phi\otimes\id)\circ (\id\otimes\tau)\circ (\Phi\otimes\id),\\
(\id\otimes\Phi)\circ (\id\otimes\tau)\circ (\tau\otimes\id)\circ (\Phi\otimes\id)&=(\tau\otimes\id)\circ (\Phi\otimes\id)\circ (\id\otimes\Phi)\circ (\id\otimes\tau),
\end{align*}
where $\tau:A\otimes A\longrightarrow A\otimes A$ is the usual flip (Definition\ref{defi1.8}). 
In particular, if $\Omega$ is a CEDS, then the vector space $\K\Omega$ generated by $\Omega$ is a linear CEDS; there are far more examples of linear CEDS than the linearizations of CEDS.

To any linear CEDS $(A,\Phi)$ is associated a category of $\Phi$-pre-Lie algebras (Definition\ref{defi2.1}).
We prove in the second section of this paper that $\Phi$-pre-Lie algebras can be described in terms of trees with graftings if, and only if, $(A,\Phi)$ is linear CEDS (Theorem\ref{theo2.2}). 
The trees used here are rooted, decorated (that is to say to any vertex is attached a decoration, taken in a set $\calD$ of generators), and typed (that is to say to any edge is attached a type, taken in the underlying space $A$ of the CEDS). 
In each case, the products of two trees is a sum over all the graftings of the first tree to a vertex of the second one, where the CEDS structure is used to modify the types of the edges in the result of the grafting,
see Proposition\ref{prop2.6} for a more precise description of this process. For example, in the three examples of CEDS described earlier, we obtain
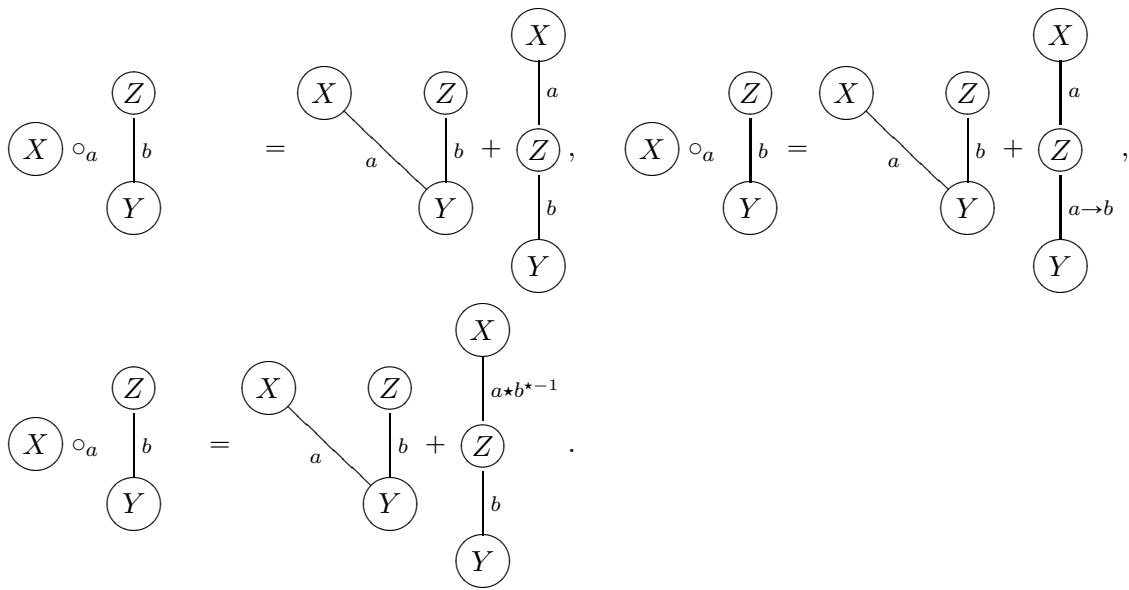
\begin{align*}
&\xymatrix{\rond{X}}\circ_a\substack{\hspace{3mm}\\ \xymatrix{\rond{Z}\ar@{-}[d]^b\\
\rond{Y}}}&=\substack{\hspace{3mm}\\ \xymatrix{\rond{X}\ar@{-}[rd]_a&\rond{Z}\ar@{-}[d]^b\\
&\rond{Y}}}+\substack{\hspace{5mm}\\ \xymatrix{\rond{X}\ar@{-}[d]^a\\
\rond{Z}\ar@{-}[d]^b\\
\rond{Y}}},&
&\xymatrix{\rond{X}}\circ_a\substack{\hspace{3mm}\\ \xymatrix{\rond{Z}\ar@{-}[d]^b\\
\rond{Y}}}&=\substack{\hspace{3mm}\\ \xymatrix{\rond{X}\ar@{-}[rd]_a&\rond{Z}\ar@{-}[d]^b\\
&\rond{Y}}}+\substack{\hspace{5mm}\\ \xymatrix{\rond{X}\ar@{-}[d]^a\\
\rond{Z}\ar@{-}[d]^{a\rightarrow b}\\
\rond{Y}}},\\
&\xymatrix{\rond{X}}\circ_a\substack{\hspace{3mm}\\ \xymatrix{\rond{Z}\ar@{-}[d]^b\\
\rond{Y}}}&=\substack{\hspace{3mm}\\ \xymatrix{\rond{X}\ar@{-}[rd]_a&\rond{Z}\ar@{-}[d]^b\\
&\rond{Y}}}+\substack{\hspace{5mm}\\ \xymatrix{\rond{X}\ar@{-}[d]^{a\star b^{\star -1}}\\
\rond{Z}\ar@{-}[d]^b\\
\rond{Y}}}.
\end{align*}
This combinatorial description of free $\Phi$-pre-Lie algebras induces a description of the operad $\prelie_\Phi$ of $\Phi$-pre-Lie algebras in terms of rooted trees with insertion into vertices, 
which generalizes the description of the operad $\prelie$ of pre-Lie algebras of\cite{Chapoton2001}.
The second section ends with the study of operadic morphisms from $\prelie$ to $\prelie_\Phi$, or equivalently to the study of pre-Lie products in any $\Phi$-pre-Lie algebras. We prove that these products
are in one-to-one correspondence with vectors $a\in A$, such that $\Phi(a)=a\otimes a$ (Proposition\ref{prop2.9}).
These vectors will be called special vectors of eigenvalue $1$. There are also special vectors of eigenvalue $0$, which give rise to morphism from the operad of non-associative permutative algebras\cite{Livernet2006}.
We prove that for any such nonzero pre-Lie product, the free $\Phi$-pre-Lie algebras are also free pre-Lie algebras and we give an explicit set of generators in Proposition\ref{prop2.10}, with the help of a convenient order on trees, in the spirit of\cite{Li2021}.\\

The third section is devoted to the study of the Koszul dual of the operad $\prelie_\Phi$. In the non-parameterized case, the Koszul dual of the operad $\prelie$ is the operad $\perm$ of permutative algebras (proposition\ref{prop3.2}): 
we naturally obtain a parameterization of this operad, by objects which are dual linear CEDS (Definition\ref{defi3.1}). 
Free $\Phi$-permutative algebras are combinatorially described in terms of monomials (Proposition\ref{prop3.2}).
Looking for morphisms from the operad $\perm$ to the operad $\perm_\Phi$ leads to special vectors of eigenvalue 1 and to weak special vectors (Proposition\ref{prop3.6}).
Note that we did not find any weak special vector, and that the only results we obtained are theorems of nonexistence in some particular cases (Proposition\ref{prop3.7}).\\

We construct some pairs of cointeracting bialgebras in the last section. These objects are pairs of bialgebras $(A,m,\delta)$ and $(A,m,\Delta)$ sharing the same algebraic background, such that $(A,m,\Delta)$
is a bialgebra in the category of right comodules over $(A,m,\delta)$. 
One of the first examples of such an object is based on trees, the first coproduct being given by an extraction-contraction process of edges, the second one being the Connes-Kreimer's one, given by admissible cuts\cite{Calaque2011}. 
Other examples are based on graphs\cite{Manchon2012,Foissy36}, posets and finite topologies\cite{Foissy37}$\ldots$ An important example of cointeracting bialgebra based on decorated and typed trees is used in\cite{Bruned2019,Bruned2015} 
in order to study stochastic PDEs, in a more general context (the tensor products need a completion there). We construct a way to obtain such a pair in an operadic context in\cite{Foissy55}: 
if $\calP$ is an operad, the the symmetric algebra generated by the invariants of the dual of $\calP$ is a bialgebra $(A,m,\delta)$, where $\delta$ is obtained by dualizing the composition of the operad.
For any operadic morphism $\phi:\prelie\longrightarrow\calP$, one can define a second coproduct $\Delta$ on $A$, making it a bialgebra $(A,m,\Delta)$, cointeracting with $(A,m,\delta)$: the coproduct $\Delta$ is obtained by dualizing the pre-Lie product induced by $\phi$. 
In our context, if $(A,\Phi)$ is a finite-dimensional CEDS and $a\in A$ is a special vector of $\Phi$ of eigenvalue 1, we obtain a pair of cointeracting bialgebras on $A$-typed trees; 
this can be generalized to $\calD$-decorated and $A$-typed trees, under the condition that $\calD$ is given a structure of commutative semigroup. We describe these to coproducts in terms of extraction-contraction and admissible cuts, 
where the types are modified according to the maps $\Phi$.
We end this paper by a more explicit description of this structure when the considered map $\Phi$ comes from one of the three examples described earlier.\\

\textbf{Acknowledgments}. 
The author acknowledges support from the grant ANR-20-CE40-0007
\emph{Combinatoire Algébrique, Résurgence, Probabilités Libres et Opérades}.

\textbf{Thanks}. The author thanks the anonymous referee for his useful comments and remarks.
\\

\begin{notation}
$\K$ is a commutative field of characteristic zero. All the vector spaces in this text will be taken over $\K$.
\end{notation}

\section{Extended (di)associative semigroups}

\subsection{Commutative extended diassociative semigroups}

Extended diassociative semigroups (briefly, EDS) are introduced in\cite{Foissy58}, where they are used to define generalizations of dendriform algebras. We here consider commutative extended semigroups:

\begin{defi}\label{defi1.1}
A commutative extended diassociative semigroup (briefly, CEDS) is a triple $(\Omega,\rightarrow,\triangleright)$, where $\Omega$ is a set
and $\rightarrow,\triangleright:\Omega^2\longrightarrow\Omega$ are maps such that, for any $\alpha,\beta,\gamma\in\Omega$,
\begin{align}
\label{EQ1}\alpha\rightarrow (\beta\rightarrow\gamma)&=(\alpha\rightarrow\beta)\rightarrow\gamma=(\beta\rightarrow\alpha)\rightarrow\gamma,\\
\label{EQ2}\alpha\triangleright (\beta\rightarrow\gamma)&=\alpha\triangleright\gamma,\\
\label{EQ3} (\alpha\triangleright\gamma)\rightarrow (\beta\triangleright\gamma)&=(\alpha\rightarrow\beta)\triangleright\gamma,\\
\label{EQ4} (\alpha\triangleright\gamma)\triangleright (\beta\triangleright\gamma)&=\alpha\triangleright\beta.
\end{align}\end{defi}

CEDS are particular examples of extended associative semigroups, introduced in\cite{Foissy58}.

\begin{defi}\cite{Foissy58}
An extended associative semigroup (briefly, EAS) is a triple $(\Omega,\rightarrow,\triangleright)$, where $\Omega$ is a set and $\rightarrow,\triangleright:\Omega^2\longrightarrow\Omega$ are maps such that, for any $\alpha,\beta,\gamma\in\Omega$,
\begin{align}
\label{EQ5}\alpha\rightarrow (\beta\rightarrow\gamma)&=(\alpha\rightarrow\beta)\rightarrow\gamma,\\
\label{EQ6}(\alpha\triangleright (\beta\rightarrow\gamma))\rightarrow (\beta\triangleright\gamma)&=(\alpha\rightarrow\beta)\triangleright\gamma,\\
\label{EQ7}(\alpha\triangleright (\beta\rightarrow\gamma))\triangleright (\beta\triangleright\gamma)&=\alpha\triangleright\beta.
\end{align}\end{defi}

More details and examples on these objects can be found in\cite{Foissy46}.

\subsection{Examples}

\begin{example} 
1. Let $\Omega$ be a set. We put
\begin{align*}
&\forall (\alpha,\beta)\in\Omega^2,&\begin{cases}
\alpha\rightarrow\beta=\beta,\\
\alpha\triangleright\beta=\alpha.
\end{cases}\end{align*}
Then $(\Omega,\rightarrow,\triangleright)$ is an EAS, denoted by $\eas(\Omega)$. It is a CEDS.\\

2. Let $(\Omega,\star)$ be an associative semigroup. We put
\begin{align*}
&\forall\alpha,\beta\in\Omega,&\alpha\triangleright\beta&=\alpha.
\end{align*}
It is an EAS, which we denote by $\eas(\Omega,\star)$. It is a CEDS if, and only if, for any $\alpha,\beta,\gamma\in\Omega$,
\[(\alpha\star\beta)\star\gamma=(\beta\star\alpha)\star\gamma.\]

3. Let $\Omega$ be a set with a binary operation $\triangleright$ such that, for any $\alpha,\beta,\gamma\in\Omega$,
\[(\alpha\triangleright\gamma)\triangleright (\beta\triangleright\gamma)=\alpha\triangleright\beta.\]
We then put
\begin{align*}
&\forall (\alpha,\beta)\in\Omega^2,&\alpha\rightarrow\beta&=\beta.
\end{align*}
Then $(\Omega,\rightarrow,\triangleright)$ is a CEDS (so is an EAS). This holds for example if $(\Omega,\star)$ is a group, with
\[\alpha\triangleright\beta=\alpha\star\beta^{\star-1}.\]
This EAS is denoted by $\eas'(\Omega,\star)$. 
\end{example}

\begin{defi}\label{defi1.3}
Let $(\Omega,\rightarrow,\triangleright)$ be an EAS. We shall say that it is non-degenerate if the following map is bijective:
\[\phi:\left\{\begin{array}{rcl}
\Omega^2&\longrightarrow&\Omega^2\\
(\alpha,\beta)&\longrightarrow&(\alpha\rightarrow\beta,\alpha\triangleright\beta).
\end{array}\right.\]
\end{defi}

\begin{example}\label{ex1.2}
1. Let $\Omega$ be a set. In $\eas(\Omega)$, for any $\alpha,\beta\in\Omega$, $\phi(\alpha,\beta)=(\beta,\alpha)$, so $\eas(\Omega)$ is non-degenerate and $\phi^{-1}=\phi$.\\

2. Let $(\Omega,\star)$ be a group. Then $\eas(\Omega,\star)$ is non-degenerate. Indeed, in this case, $\phi(\alpha,\beta)=(\alpha\star\beta,\alpha)$, so $\phi$ is a bijection, of inverse given by $\phi^{-1}(\alpha,\beta)=(\beta,\beta^{\star-1}\star\alpha)$.\\

3. Let $(\Omega,\star)$ be an associative semigroup with the right inverse condition. Then $\eas'(\Omega,\star)$ is non-degenerate.
Indeed, in this case, $\phi(\alpha,\beta)=(\beta,\alpha\star\beta^{\star-1})$, so $\phi$ is a bijection, of inverse given by $\phi^{-1}(\alpha,\beta)=(\beta\star\alpha,\alpha)$.
\end{example}

\subsection{Reformulations with the map $\phi$}

Let us first reformulate the axioms of EAS and CEDS in terms of the map $\phi$ of Definition\ref{defi1.3}.

\begin{lemma}\label{lem1.4}
Let $(\Omega,\rightarrow,\triangleright)$ be a set with two binary operations. We consider the maps
\begin{align*}
\phi&:\left\{\begin{array}{rcl}
\Omega^2&\longrightarrow&\Omega^2\\
(\alpha,\beta)&\longrightarrow&(\alpha\rightarrow\beta,\alpha\triangleright\beta),
\end{array}\right.&
\tau&:\left\{\begin{array}{rcl}
\Omega^2&\longrightarrow&\Omega^2\\
(\alpha,\beta)&\longrightarrow&(\beta,\alpha).
\end{array}\right.
\end{align*}
Then:
\begin{enumerate}
\item $(\Omega,\rightarrow,\triangleright)$ is an EAS if, and only if
\begin{align}
\label{EQ8}(\id\times\phi)\circ (\phi\times\id)\circ (\id\times\phi)&=(\phi\times\id)\circ (\id\times\tau)\circ (\phi\times\id).
\end{align}
\item $(\Omega,\rightarrow,\triangleright)$ is a CEDS if, and only if
\begin{align}
\tag{\ref{EQ8}}(\id\times\phi)\circ (\phi\times\id)\circ (\id\times\phi)&=(\phi\times\id)\circ (\id\times\tau)\circ (\phi\times\id),\\
\label{EQ9}(\id\times\phi)\circ (\id\times\tau)\circ (\tau\times\id)\circ (\phi\times\id)&=(\tau\times\id)\circ (\phi\times\id)\circ (\id\times\phi)\circ (\id\times\tau).
\end{align}\end{enumerate}\end{lemma}

\begin{proof}
Let $\alpha,\beta,\gamma\in\Omega$. Then
\begin{align*}
&(\id\times\phi)\circ (\phi\times\id)\circ (\id\times\phi)(\alpha,\beta,\gamma)\\
&=(\alpha\rightarrow (\beta\rightarrow\gamma),(\alpha\triangleright (\beta\rightarrow\gamma))\rightarrow (\beta\triangleright\gamma),(\alpha\triangleright (\beta\rightarrow\gamma))\triangleright (\beta\triangleright\gamma)),\\ 
&(\phi\times\id)\circ (\id\times\tau)\circ (\phi\times\id)(\alpha,\beta,\gamma)\\
&=((\alpha\rightarrow\beta)\rightarrow\gamma, (\alpha\rightarrow\beta)\triangleright\gamma,\alpha\triangleright\beta),\\ \\
&(\id\times\phi)\circ (\id\times\tau)\circ (\tau\times\id)\circ (\phi\times\id)(\alpha,\beta,\gamma)\\
&=(\alpha\triangleright\beta,\gamma\rightarrow (\alpha\rightarrow\beta),\gamma\triangleright (\alpha\rightarrow\beta)),\\ 
&(\tau\times\id)\circ (\phi\times\id)\circ (\id\times\phi)\circ (\id\times\tau)(\alpha,\beta,\gamma)\\
&=(\alpha\triangleright (\gamma\rightarrow\beta),\alpha\rightarrow (\gamma\rightarrow\beta),\gamma\triangleright\beta).
\end{align*}
The result immediately follows.\end{proof}

Reversing (\ref{EQ8}) gives (\ref{EQ8}) again; reversing (\ref{EQ9}), we obtain the notion of dual CEDS:

\begin{defi}
A dual CEDS is a pair $(\Omega,\phi)$ where $\Omega$ is a set and $\phi:\Omega^2\longrightarrow\Omega^2$
such that
\begin{align}
\tag{\ref{EQ8}}(\id\times\phi)\circ (\phi\times\id)\circ (\id\times\phi)&=(\phi\times\id)\circ (\id\times\tau)\circ (\phi\times\id),\\
\label{EQ10}(\phi\times\id)\circ (\tau\times\id)\circ (\id\times\tau)\circ (\id\times\phi)&=(\id\times\tau)\circ (\id\times\phi)\circ (\phi\times\id)\circ (\tau\times\id).
\end{align}\end{defi}

By direct computation, we can reformulate the axioms of dual CEDS:

\begin{prop}
Let $(\Omega,\rightarrow,\triangleright)$ be a map with two binary operations. It is a dual CEDS if, for any $\alpha,\beta,\gamma\in\Omega$,
\begin{align}
\label{EQ11} (\alpha\rightarrow\beta)\rightarrow\gamma&=\alpha\rightarrow (\beta\rightarrow\gamma),\\
\tag{\ref{EQ6}}(\alpha\triangleright (\beta\rightarrow\gamma))\rightarrow (\beta\triangleright\gamma)&=(\alpha\rightarrow\beta)\triangleright\gamma,\\
\tag{\ref{EQ7}} (\alpha\triangleright (\beta\rightarrow\gamma))\triangleright (\beta\triangleright\gamma)&=\alpha\triangleright\beta,\\
\label{EQ12} (\alpha\triangleright\beta)\rightarrow\gamma&=\alpha\rightarrow\gamma,\\
\label{EQ13} (\alpha\triangleright\beta)\triangleright\gamma&=(\alpha\triangleright\gamma)\triangleright\beta.
\end{align}\end{prop}

\begin{remark}
By definition, dual CEDS are EAS.
\end{remark}

We immediately obtain:

\begin{prop}\label{prop1.7}
Let $(\Omega,\rightarrow,\triangleright)$ be a set with two binary operations, such that the map $\phi$ is a bijection. We put
\[\phi^{-1}:\left\{\begin{array}{rcl}
\Omega^2&\longrightarrow&\Omega^2\\
(\alpha,\beta)&\longrightarrow&(\alpha\curvearrowright\beta,\alpha\blacktriangleright\gamma).
\end{array}\right.\]
Then $(\Omega,\rightarrow,\triangleright)$ is an EAS (respectively a CEDS, a dual CEDS) if, and only if, $(\Omega,\curvearrowright,\blacktriangleright)$ is an EAS (respectively a dual CEDS, a CEDS).
\end{prop}

\begin{example}
From Example\ref{ex1.2}:
\begin{enumerate}
\item If $(A,\Phi)=\eas(\Omega)$, then $(A,\Phi^{-1})=\eas(\Omega)$.
\item If $(\Omega,\star)$ is a group and if $(A,\Phi)=\eas(\Omega,\star)$, then $(A,\Phi^{-1})=\eas'(\Omega,\star^{op})$.
\item If $(\Omega,\star)$ is a group and if $(A,\Phi)=\eas'(\Omega,\star)$, then $(A,\Phi^{-1})=\eas(\Omega,\star^{op})$.
\end{enumerate}\end{example}

\subsection{Linear extended semigroups}

Following the discrete version of Lemma\ref{lem1.4}, we now introduce the notion of $\ell$EAS, $\ell$CEDS and dual $\ell$CEDS:

\begin{defi}\label{defi1.8}
Let $A$ be a vector space and let $\Phi:A\otimes A\longrightarrow A\otimes A$ be a linear map.
\begin{enumerate}
\item We shall say that $(A,\Phi)$ is a linear extended associative semigroup (briefly, $\ell$EAS) if
\begin{align}
\label{EQ14}(\id\otimes\Phi)\circ (\Phi\otimes\id)\circ (\id\otimes\Phi)&=(\Phi\otimes\id)\circ (\id\otimes\tau)\circ (\Phi\otimes\id).
\end{align}
\item We shall say that $(A,\Phi)$ is a linear commutative extended diassociative semigroup (briefly, $\ell$CEDS) if
\begin{align}
\tag{\ref{EQ14}}(\id\otimes\Phi)\circ (\Phi\otimes\id)\circ (\id\otimes\Phi)&=(\Phi\otimes\id)\circ (\id\otimes\tau)\circ (\Phi\otimes\id),\\
\label{EQ15}(\id\otimes\Phi)\circ (\id\otimes\tau)\circ (\tau\otimes\id)\circ (\Phi\otimes\id)&=(\tau\otimes\id)\circ (\Phi\otimes\id)\circ (\id\otimes\Phi)\circ (\id\otimes\tau).
\end{align}
\item We shall say that $(A,\Phi)$ is a linear dual commutative extended diassociative semigroup (briefly, dual $\ell$CEDS) if
\begin{align}
\tag{\ref{EQ14}}(\id\otimes\Phi)\circ (\Phi\otimes\id)\circ (\id\otimes\Phi)&=(\Phi\otimes\id)\circ (\id\otimes\tau)\circ (\Phi\otimes\id),\\
\label{EQ16}(\Phi\otimes\id)\circ (\tau\otimes\id)\circ (\id\otimes\tau)\circ (\id\otimes\Phi)&=(\id\otimes\tau)\circ (\id\otimes\Phi)\circ (\Phi\otimes\id)\circ (\tau\otimes\id).
\end{align}
\end{enumerate}
If $(A,\Phi)$ is an $\ell$EAS (respectively an $\ell$CEDS or a dual $\ell$CEDS), we shall say that it is non-degenerate if $\Phi$ is bijective.
\end{defi}

Note that, by definition, $\ell$CEDS and dual $\ell$CEDS are $\ell$EAS.

\begin{example}\label{ex1.4} 
Let $(\Omega,\rightarrow,\triangleright)$ be an EAS (respectively, a CEDS, a dual CEDS), and let $A=\K\Omega$ be the vector space generated by $\Omega$. We define
\[\Phi:\left\{\begin{array}{rcl}
A\otimes A&\longrightarrow&A\otimes A\\
a\otimes b&\longrightarrow&(a\rightarrow b)\otimes (a\triangleright b).
\end{array}\right.\]
Then $(A,\Phi)$ is an $\ell$EAS (respectively, an $\ell$CEDS, a dual $\ell$CEDS), called the linearization of $(\Omega,\rightarrow,\triangleright)$. It is a non-degenerate $\ell$EAS if, and only if, $(\Omega,\rightarrow,\triangleright)$ is a non-degenerate EAS.
\end{example}

Other examples can be found in\cite{Foissy46}.

\begin{notation}
Let $(A,\Phi)$ be an $\ell$EAS. We use the Sweedler notation
\begin{align}
\label{EQ17}\Phi(a\otimes b)&=\sum a'\rightarrow b'\otimes a''\triangleright b''.
\end{align}
Note that the operations $\rightarrow$ and $\triangleright$ may not necessarily exist, nor the coproducts $a'\otimes a''$ or $b'\otimes b''$. With this notation, (\ref{EQ14}) can be rewritten as
\begin{align}
\tag{\ref{EQ14}'}&\sum\sum\sum a'\rightarrow (b'\rightarrow c')'\otimes(a''\triangleright (b'\rightarrow c')'')'\rightarrow (b''\triangleright c'')'\otimes (a''\triangleright (b'\rightarrow c'))''\triangleright (b''\triangleright c'')''\\
\nonumber &=\sum\sum (a'\rightarrow b')'\rightarrow c'\otimes (a'\rightarrow b')''\triangleright c''\otimes a''\triangleright b''.
\end{align}
Similarly, (\ref{EQ15}) and (\ref{EQ16}) are rewritten as
\begin{align}
\tag{\ref{EQ15}'}&\sum\sum a''\triangleright (c'\rightarrow b')''\otimes a'\rightarrow (c''\rightarrow b'')'\otimes c''\triangleright b''\\
\nonumber&=\sum\sum a''\triangleright b''\otimes c'\rightarrow (a'\rightarrow b')'\otimes c''\triangleright (a'\rightarrow b')'',\\ 
\nonumber\\
\tag{\ref{EQ16}'}&\sum\sum (b''\triangleright c'')'\rightarrow a'\otimes (b''\triangleright c'')''\triangleright a''\otimes b'\rightarrow c'\\
\nonumber&=\sum\sum b'\rightarrow a'\otimes (b''\triangleright a'')''\triangleright c''\otimes (b''\triangleright a'')'\rightarrow c'.
\end{align}\end{notation}

By transposition of (\ref{EQ14}), (\ref{EQ15}), (\ref{EQ16}):

\begin{prop}
Let $V$ be a finite-dimensional space and $\Phi:V\otimes V\longrightarrow V\otimes V$ be a linear map.
We consider $\Phi^*:V^*\otimes V^*=(V\otimes V)^*\longrightarrow(V\otimes V)^*=V^*\otimes V^*$. Then $(V,\Phi)$ is an $\ell$EAS [respectively an $\ell$CEDS, a dual $\ell$CEDS] if, and only if, $(V^*,\Phi^*)$ is an $\ell$EAS [respectively a dual $\ell$CEDS, an $\ell$CEDS].
\end{prop}

\begin{example}
Let $\Omega$ be a finite EAS and $A=(\K\Omega,\Phi)$ be its linearization. The dual $A^*$ is identified with the space $\K^\Omega$ of maps from $\Omega$ to $\K$, with the dual basis $(\partial_\alpha)_{\alpha\in\Omega}$ of the basis $\Omega$ of $\K\Omega$. 
Then, for any $\alpha,\beta\in\Omega$,
\[\Phi^*(\partial_\alpha\otimes\partial_\beta)=\sum_{(\gamma,\partial)\in\phi^{-1}(\alpha,\beta)}\partial_\gamma\otimes\partial_\delta.\]
If $(\Omega,\Phi)$ is degenerate, this is not the linearization of an EAS.
\end{example}

\section{Generalized prelie algebras}

\subsection{Definition}

\begin{defi}\label{defi2.1}
Let $A$ be a vector space, $\Phi:A\otimes A\longrightarrow A\otimes A$ be a linear map, and let $(V,\circ)$ be a family such that $V$ is a vector space and $\circ$ is a linear map:
\[\circ:\left\{\begin{array}{rcl}
A&\longrightarrow&\mathcal{L}(V\otimes V,V)\\
a&\longrightarrow&\circ _a:\left\{\begin{array}{rcl}
V\otimes V&\longrightarrow&V\\
v\otimes w&\longrightarrow&v\circ_a w.
\end{array}\right.
\end{array}\right.\]
We shall say that $(V,\circ)$ is a $\Phi$-prelie algebra if, for any $x,y,z\in V$, for any $a,b\in A$, using Sweedler's notation (\ref{EQ17}) for $\Phi$,
\begin{align}
\label{EQ18}x\circ_a (y\circ_b z)-\sum (x\circ_{a''\triangleright b''} y)\circ_{a'\rightarrow b'}z&=y\circ_b (x\circ_a z)-\sum (y\circ_{b''\triangleright a''} x)\circ_{b'\rightarrow a'}z.
\end{align}
We denote by $\prelie_\Phi$ the operad of $\Phi$-prelie algebras.
\end{defi}

\begin{example}
1. Let $(\Omega,\rightarrow,\triangleright)$ be a CEDS and $(A,\Phi)$ its linearization. We obtained a discretized version of generalized prelie algebras:
for any $\alpha\in\Omega$, $V$ is given a product $\circ_\alpha:V\otimes V\longrightarrow V$, such that, for any $\alpha,\beta\in\Omega$, for any $x,y,z\in V$,
\begin{align*}
x\circ_\alpha (y\circ_\beta z)- (x\circ_{\alpha\triangleright\beta} y)\circ_{\alpha\rightarrow\beta}z&=y\circ_\beta (x\circ_\alpha z)- (y\circ_{\beta\triangleright\alpha} x)\circ_{\beta\rightarrow\alpha}z.
\end{align*}
For example:
\begin{itemize}
\item If $\Omega$ is a set, taking $\eas(\Omega)$ (which is a CEDS), we obtain $\Omega$-matching prelie algebras of\cite{Foissy47,Zhang2020}.
\item If $(\Omega,\star)$ is a commutative semigroup, taking $\eas(\Omega,\star)$ (which is a CEDS), we obtain $\Omega$-family prelie algebras of\cite{Aguiar2020,Manchon2020}.
\end{itemize}

2. Let us take $A=\K$ and let us put $\bullet=\circ_1$. If $\Phi(1\otimes 1)=1\otimes 1$, then (\ref{EQ18}) becomes
\[x\bullet (y\bullet z)-(x\bullet y)\bullet z=y\bullet (x\bullet z)-(y\bullet x)\bullet z.\]
In other words, $\Phi$-prelie algebras are (left) prelie algebras. f $\Phi(1\otimes 1)=0$, then (\ref{EQ18}) becomes
\[x\bullet (y\bullet z)=y\bullet (x\bullet z).\]
In other words, $\Phi$-prelie algebras are non-associative permutative (briefly, NAP) algebras, as considered by\cite{Livernet2006}.
\end{example}

\subsection{Structure on rooted trees}

Let $A$ be a vector space and $\calD$ be a set. 

\begin{notation}
1. Let us denote by $\calT$ the set of rooted trees:
\[\calT=\left\{
\substack{\hspace{1cm}\\ \tun,\tdeux,\ttroisun,\ttroisdeux,\tquatreun,\tquatredeux,\tquatrequatre,\tquatrecinq,\tcinqun,\tcinqdeux,\tcinqcinq,\tcinqsix,\tcinqhuit,\tcinqdix,\tcinqonze,\tcinqdouze,\tcinqquatorze\ldots}\right\}.\]
We denote by $\calT_{A,\calD}$ the set of $\calD$-decorated and $A$-typed trees, that is to say rooted trees $T$ with a map $d:V(T)\longrightarrow\calD$ from the set of vertices of $T$ to $\calD$
and a map $t:E(T)\longrightarrow A$ from the set of edges of $T$ to $A$, the type being linear in each edge:
for example, if $\g_{A,\calD}$, if $d_1,d_2,d_3\in\calD$, $a,a',b,b'\in A$ and $\lambda,\mu\in\K$,
\begin{align*}
\xymatrix{\rond{d_2}\ar@{-}[rd]_{a+\lambda a'}&\rond{d_3}\ar@{-}[d]^{b+\mu b'}\\&\rond{d_1}}&=\xymatrix{\rond{d_2}\ar@{-}[rd]_a&\rond{d_3}\ar@{-}[d]^b\\&\rond{d_1}}+\lambda\xymatrix{\rond{d_2}\ar@{-}[rd]_{a'}&\rond{d_3}\ar@{-}[d]^b\\&\rond{d_1}}\\
&+\mu\xymatrix{\rond{d_2}\ar@{-}[rd]_a&\rond{d_3}\ar@{-}[d]^{b'}\\&\rond{d_1}}+\lambda\mu\xymatrix{\rond{d_2}\ar@{-}[rd]_{a'}&\rond{d_3}\ar@{-}[d]^{b'}\\&\rond{d_1}}.
\end{align*}
We denote by $\g_{A,\calD}$ the vector space generated by $\calT_{A,\calD}$, the trees being linear in the type of each edge.
More formally, denoting by $\calT_\calD$ the set of isoclasses of rooted trees decorated by $\calD$,
\[g_{A,\calD}=\bigoplus_{T\in\calT_\calD}\K T\otimes_{\mathrm{Aut}(T)}\bigotimes_{e\in E(T)} A,\]
where $\mathrm{Aut}(T)$ acts trivially on $\K T$ and by permutation of the tensors on $\displaystyle\bigotimes_{e\in E(T)} A$.\\

2. If $T_1,\ldots,T_k\in\calT_{A,\calD}$, $d\in\calD$ and $a_1,\ldots,a_k\in A$, we denote by $B_d(a_1T_1,\ldots,a_kT_k)$ the tree obtained by grafting $T_1,\ldots,T_k$ on a common root decorated by $d$; the edge relating the root to the root of $T_i$
in the process is of type $a_i$ for any $i$. This defines for any $d\in\calD$ a map
$B_d:S(A\otimes\g_{A,\calD})\longrightarrow\g_{A,\calD}$. For example, if $d_1,d_2,d_3,d\in\calD$
and $a,b,c\in A$,
\[B_d\left(a\xymatrix{\rond{d_1}},\:b\hspace{-5mm}\substack{\hspace{2cm}\\ \xymatrix{\rond{d_3}\ar@{-}[d]^c\\ \rond{d_2}}}
\hspace{-5mm}\right)=\substack{\hspace{5mm}\\ \xymatrix{&\rond{d_3}\ar@{-}[d]^c\\
\rond{d_1}\ar@{-}[rd]_a&\rond{d_2}\ar@{-}[d]^b\\
&\rond{d}}}.\]
\end{notation}

Generalizing the construction of free prelie algebras of\cite{Chapoton2001}, we obtain:

\begin{theo}\label{theo2.2}
Let $\Phi:A\otimes A\longrightarrow A\otimes A$ be a map. We define products $\circ_a:\g_{A,\calD}\otimes\g_{A,\calD}\longrightarrow\g_{A,\calD}$ for any $a\in A$ in an inductive way: for any $T,T_1,\ldots,T_k\in\calT_{A,\calD}$, for any $a,a_1,\ldots,a_k\in A$,
\begin{align*}
T\circ_a B_d(a_1T_1,\ldots,a_kT_k)&=B_d(a T,a_1T_1,\ldots, a_kT_k)\\
&+\sum_{i=1}^k\sum B_d(a_1T_1,\ldots,a_{i-1}T_{i-1}, (a'\rightarrow a_i') T\circ_{a ''\triangleright a_i''}T_i,a_{i+1}T_{i+1},\ldots, a_kT_k).
\end{align*}
Then, for any set $\calD$, $(\g_{A,\calD},\circ)$ is a $\Phi$-prelie algebra if, and only if, $(A,\Phi)$ is an $\ell$CEDS.
Moreover, if this holds, then $\g_{A,\calD}$ is the free $\Phi$-prelie algebra generated by trees $\xymatrix{\rond{d}}$ with one vertex, decorated by an element $d$ of $\calD$. 
\end{theo}

\begin{example}
Let $a,b\in A$ and $X,Y,Z\in\calD$. Then
\begin{align*}
\xymatrix{\rond{X}}\circ_a\xymatrix{\rond{Y}}
&=\substack{\hspace{3mm}\\ \xymatrix{\rond{X}\ar@{-}[d]^a\\
\rond{Y}}},&
\xymatrix{\rond{X}}\circ_a\substack{\hspace{3mm}\\ \xymatrix{\rond{Z}\ar@{-}[d]^b\\
\rond{Y}}}&=\substack{\hspace{3mm}\\ \xymatrix{\rond{X}\ar@{-}[rd]_a&\rond{Z}\ar@{-}[d]^b\\
&\rond{Y}}}+\sum\substack{\hspace{5mm}\\ \xymatrix{\rond{X}\ar@{-}[d]^{a''\triangleright b''}\\
\rond{Z}\ar@{-}[d]^{a'\rightarrow b'}\\
\rond{Y}}}.
\end{align*}\end{example}

\begin{proof} $\Longrightarrow$. Let $\calD=\{X,Y,Z\}$ be a set of cardinality 3. In $\g_{A,\calD}$,
\begin{align*}
&\xymatrix{\rond{X}}\circ_a\left(\xymatrix{\rond{Y}}\circ_b
\substack{\hspace{3mm}\\ \xymatrix{\rond{Z}\ar@{-}[d]^c\\
\rond{X}}}\right)=\substack{\hspace{3mm}\\ \xymatrix{\rond{X}\ar@{-}[rd]_a
&\rond{Y}\ar@{-}[d]^b&\rond{Z}\ar@{-}[ld]^c\\&\rond{X}}}
+\sum\substack{\hspace{5mm}\\ \xymatrix{\rond{X}\ar@{-}[d]_{a''\triangleright b''}&\\
\rond{Y}\ar@{-}[rd]_{a'\rightarrow b'}&\rond{Z}\ar@{-}[d]^c\\&\rond{X}&}
\hspace{-1cm}}+\sum\substack{\hspace{5mm}\\ \xymatrix{&\rond{X}\ar@{-}[d]^{a''\triangleright c''}\\
\rond{Y}\ar@{-}[rd]_b&\rond{Z}\ar@{-}[d]^{a'\rightarrow c'}\\
&\rond{X}&}}\\
&+\sum\substack{\hspace{5mm}\\ \xymatrix{&\rond{Y}\ar@{-}[d]^{b''\triangleright c''}\\
\rond{X}\ar@{-}[rd]_a&\rond{Z}\ar@{-}[d]^{b'\rightarrow c'}\\&\rond{X}&}}
\hspace{-.5cm}+\sum\sum 
\substack{\hspace{5mm}\\ \xymatrix{\rond{X}\ar@{-}[rd]_{a''\triangleright (b'\rightarrow c')''}&\rond{Y}\ar@{-}[d]^{b''\triangleright c''}\\
&\rond{Z}\ar@{-}[d]^{a'\rightarrow (b'\rightarrow c')}\\&\rond{X}}}
+\sum\sum\sum\substack{\hspace{8mm}\\ \xymatrix{\rond{X}\ar@{-}[d]^{(a''\triangleright (b'\rightarrow c')'')''
\triangleright (b''\triangleright c'')''}\\
\rond{Y}\ar@{-}[d]^{(a''\triangleright (b'\rightarrow c')'')'\rightarrow (b''\triangleright c'')'}\\
\rond{Z}\ar@{-}[d]^{a'\rightarrow(b'\rightarrow c')'}\\
\rond{X}}},
\end{align*}
whereas
\begin{align*}
\left(\xymatrix{\rond{X}}\circ_{a''\triangleright b''}\xymatrix{\rond{Y}}\right)\circ_{a'\rightarrow b'}
\substack{\hspace{3mm}\\ \xymatrix{\rond{Z}\ar@{-}[d]^c\\
\rond{X}}}&=\sum\substack{\hspace{5mm}\\ \xymatrix{\rond{X}\ar@{-}[d]_{a''\triangleright b''}&\\
\rond{Y}\ar@{-}[rd]_{a'\rightarrow b'}&\rond{Z}\ar@{-}[d]^c\\
&\rond{X}&}}\hspace{-1cm}+\sum\sum\substack{\hspace{8mm}\\ \xymatrix{\rond{X}\ar@{-}[d]^{a''\triangleright b''}\\
\rond{Y}\ar@{-}[d]^{(a'\rightarrow b')''\triangleright c''}\\
\rond{Z}\ar@{-}[d]^{(a'\rightarrow b')'\rightarrow c'}\\
\rond{X}}}.
\end{align*}
Hence, relation (\ref{EQ18}) gives, by identification of the types of the types of the edges of the trees of the form
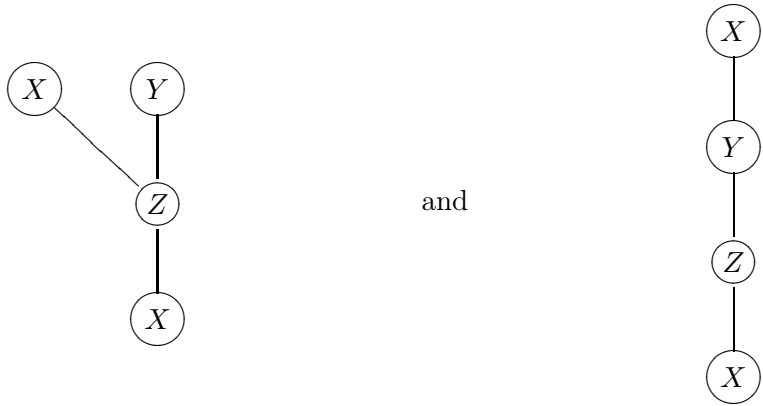
\begin{align*}
&\substack{\hspace{5mm}\\ \xymatrix{\rond{X}\ar@{-}[rd]&\rond{Y}\ar@{-}[d]\\
&\rond{Z}\ar@{-}[d]\\
&\rond{X}}}&\mbox{ and }&&\substack{\hspace{8mm}\\ \xymatrix{\rond{X}\ar@{-}[d]\\
\rond{Y}\ar@{-}[d]\\
\rond{Z}\ar@{-}[d]\\
\rond{X}}},
\end{align*}
the relation (\ref{EQ15}) for the first one, up to the permutation of $b$ and $c$, and (\ref{EQ14}) for the second one. So $A$ is an $\ell$CEDS.\\

$\Longleftarrow$. Let us prove (\ref{EQ18}) for $x$, $y$ and $z$ trees, by induction on the number of vertices $n$ of $z$. We put $z=B_d(a_1T_1,\ldots,a_kT_k)$. Then
\begin{align*}
x\circ_a (y\circ_b z)&=B_d(a x,b, y, a_1T_1,\ldots,a_kT_k)+\sum B_d((a'\rightarrow b')(x\circ_{a''\triangleright b''} y),a_1T_1,\ldots,a_kT_k)\\
&+\sum_{i=1}^k\sum B_d(b y, a_1T_1,\ldots, (a'\rightarrow a_i')(x\circ_{a''\triangleright a_i''}T_i),\ldots,a_kT_k)\\
&+\sum_{i=1}^k\sum B_d(a x, a_1T_1,\ldots, (b'\rightarrow a_i')(y\circ_{b''\triangleright a_i''}T_i),\ldots,a_kT_k)\\
&+\sum_{i=1}^k\sum\sum B_d(a_1T_1,\ldots,((a'\rightarrow (b'\rightarrow a_i')')x\circ_{a''\triangleright(b'\rightarrow a_i')''}(y\circ_{b''\triangleright a_i''} T_i)),\ldots,a_kT_k)\\
&+\sum_{i\neq j}\sum\sum B_d(a_1T_1,\ldots, (a'\rightarrow a_i')(x\circ_{a''\triangleright a_i''}T_i),\ldots,(b'\rightarrow a_j')(y\circ_{b''\triangleright a_j''}T_j),\ldots,a_kT_k),
\end{align*}
and
\begin{align*}
&\sum (x\circ_{a''\triangleright b''} y)\circ_{a'\rightarrow b'} z\\
&=\sum B_d((a'\rightarrow b')(x\circ_{a''\triangleright b''} y),a_1T_1,\ldots,a_kT_k)\\
&+\sum_{i=1}^k\sum\sum B_d(a_1T_1,\ldots,((a'\rightarrow b')'\rightarrow a_i')(x\circ_{a''\triangleright b''} y)\circ_{(a'\rightarrow b')''\triangleright a_i''} T_i),\ldots,a_kT_k).
\end{align*}
Hence,
\begin{align}
\label{EQ19}x\circ_a (y\circ_b z)-\sum (x\circ_{a'\triangleright b'} y)\circ_{a''\rightarrow b''} z&=\sum_{i=1}^kB_d(a_1T_1,\ldots,\omega(a x,b y,a_iT_i),\ldots,a_kT_k)\\
\nonumber &+\mbox{terms symmetric in }(a x,b y),
\end{align}
where
\begin{align*}
\omega(a x,b y,c z')&=\sum\sum a'\rightarrow (b'\rightarrow c')' x\circ_{a''\triangleright (b'\rightarrow c')''}(y\circ_{b''\triangleright c''} z')\\
&-\sum\sum (a'\rightarrow b')'\rightarrow c '(x\circ_{a''\triangleright b''} y)\circ_{(a'\rightarrow b')''\triangleright c''} z'.
\end{align*}
If $n=0$, then $k=0$, so (\ref{EQ18}) is satisfied from (\ref{EQ19}). Otherwise, we put
\begin{align*}
u&=\sum\sum a'\rightarrow (b'\rightarrow c')'\otimes a''\triangleright(b'\rightarrow c')''\otimes b''\triangleright c''.
\end{align*}
Then
\begin{align*}
\id\otimes\Phi(u)=\sum\sum (a'\rightarrow b')'\rightarrow c'\otimes (a'\rightarrow b')''\triangleright c''\otimes a''\triangleright b''.
\end{align*}
Applying the induction hypothesis on $z'=T_i$, we obtain that
\[\omega(a x,b y,a_i T_i)=0.\]
So by (\ref{EQ19}), (\ref{EQ18}) is satisfied for $(x,yz)$. Hence, $\g_{A,\calD}$ is indeed a $\Phi$-prelie algebra.\\

The proof of its freeness is similar to the proof of\cite{Chapoton2001}, see also\cite{Foissy47}. 
\end{proof}

As the number of indexed rooted trees with $n$ vertices is $n^{n-1}$, see sequence A000169 of the OEIS\cite{Sloane}:

\begin{cor}
If $(A,\Phi)$ is a finite dimensional $\ell$CEDS, then for any $n\geqslant 1$,
\[\dim(\prelie_\Phi(n))=n^{n-1}\dim(A)^{n-1}.\]
\end{cor}

\subsection{Links with prelie algebras}

\begin{prop}\label{prop2.4}
Let $A$ be a vector space, $\Phi:A\otimes A\longrightarrow A\otimes A$ be a linear map and $V$ be a vector space with a map 
\[\circ:\left\{\begin{array}{rcl}
A&\longrightarrow&\hom(V\otimes V,V)\\
a&\longrightarrow&\circ_a:\left\{\begin{array}{rcl}
V\otimes V&\longrightarrow&V\\
x\otimes y&\longrightarrow&x\circ_a y.
\end{array}\right.
\end{array}\right.\]
We define a product on $A\otimes V$ by
\begin{align*}
&\forall x,y\in V,\:\forall a,b\in A,&a x\bullet b y&=\sum a'\rightarrow b' x\circ_{a''\triangleright b''} y.
\end{align*}
Then:
\begin{enumerate}
\item If $(A,\Phi)$ is an $\ell$CEDS and $(A,\circ)$ is a $\Phi$-prelie algebra, then $(A\otimes V,\bullet)$ is a prelie algebra.
\item If $(A,\Phi)$ is a non-degenerate $\ell$CEDS and $(A\otimes V,\bullet)$ is a prelie algebra, then $(A,\circ)$ is a $\Phi$-prelie algebra.
\item If, for the free $\Phi$-prelie algebra $\g_{A,\calD}$, where $\calD$ contains at least three elements, $(A\otimes\g_{A,\calD},\circ)$ is a prelie algebra, then $(A,\Phi)$ is an $\ell$CEDS.
\end{enumerate}\end{prop}

\begin{proof}
Let $a,b,c\in A$ and $x,y,z\in V$. Then
\begin{align*}
(ax\bullet by)\bullet cz&=\sum\sum (a'\rightarrow b')'\rightarrow c' (x\circ_{a''\triangleright b''} y)\circ_{(a'\rightarrow b')''\triangleright c''} z,\\
ax\bullet (by\bullet cz)&=\sum\sum a'\rightarrow (b'\rightarrow c')' x\circ_{a''\triangleright (b'\rightarrow c')''}(y\circ_{b''\triangleright c''} z),\\
(by\bullet ax)\bullet cz&=\sum\sum (b'\rightarrow a')'\rightarrow c' (y\circ_{b''\triangleright a''} x)\circ_{(b'\rightarrow a')''\triangleright c''} z,\\
by\bullet (ax\bullet cz)&=\sum\sum b'\rightarrow (a'\rightarrow c')'y\circ_{b''\triangleright (a'\rightarrow c')''}(x\circ_{a''\triangleright c''} z).
\end{align*}

1. We obtain
\begin{align*}
&(ax\bullet by)\bullet cz-ax\bullet (by\bullet cz)\\
&=\sum\sum (a'\rightarrow b')'\rightarrow c' (x\circ_{a''\triangleright b''} y)\circ_{(a'\rightarrow b')''\triangleright c''} z\\
&-\sum\sum a'\rightarrow (b'\rightarrow c')' x\circ_{a''\triangleright (b'\rightarrow c')''}(y\circ_{b''\triangleright c''} z)\\
&=\sum\sum\sum a'\rightarrow (b'\rightarrow c')'x\circ_{(a''\triangleright(b'\rightarrow c')'')''\triangleright (b''\triangleright c'')''} y)\circ_{(a''\triangleright(b'\rightarrow c')'')'\rightarrow (b''\triangleright c'')'}z&\mbox{by (\ref{EQ14})}\\
&-\sum\sum a'\rightarrow (b'\rightarrow c')' x\circ_{a''\triangleright (b'\rightarrow c')''}(y\circ_{b''\triangleright c''} z)\\
&=\sum\sum\sum b'\rightarrow (a'\rightarrow c')'(x\circ_{(a''\triangleright c'')''\triangleright (b''\triangleright (a'\rightarrow c')'')''} y)\circ_{(a''\triangleright c'')'\rightarrow (b''\triangleright (a'\rightarrow c')'')'}z&\mbox{by (\ref{EQ14})}\\
&-\sum\sum b'\rightarrow (a'\rightarrow c')' x\circ_{a''\triangleright c''} (y\circ_{b''\triangleright (a'\rightarrow c')''}z)\\
&=\sum\sum\sum b'\rightarrow (a'\rightarrow c')' (y\circ_{(b''\triangleright(a'\rightarrow c')'')''\triangleright (a''\triangleright c'')''} x)\circ_{(b''\triangleright(a'\rightarrow c')'')'\rightarrow (a''\triangleright c'')'}z&\mbox{by (\ref{EQ5})}\\
&-\sum\sum b'\rightarrow (a'\rightarrow c') y\circ_{b''\triangleright(a'\rightarrow c')''} (x\circ_{a''\triangleright c''}z)\\
&=(by\bullet ax)\bullet cz-by\bullet (ax\bullet cz).
\end{align*}
So $(A\otimes V,\bullet)$ is prelie.\\

2. Let $a_1,b_1,c_1\in V$. As $\Phi$ is surjective, there exists $a\otimes b\otimes c\in A^{\otimes 3}$, such that
\[(\Phi\otimes\id)\circ (\id\otimes\Phi)(a\otimes b\otimes c)=a_1\otimes b_1\otimes c_1,\]
or, with Sweedler's notation,
\[\sum\sum a'\rightarrow(b'\rightarrow c')'\otimes a''\triangleright (b'\rightarrow c')''\otimes b''\triangleright c''=a_1\otimes b_1\otimes c_1.\]
The prelie relation gives
\begin{align*}
&(ax\bullet by)\bullet cz-ax\bullet (by\bullet cz)\\
&=\sum\sum a_1 (x\circ_{a''\triangleright b''} y)\circ_{(a'\rightarrow b')''\triangleright c''} z-a_1 x\circ_{b_1} (y\circ_{c_1} z)\\
&=\sum a_1(x\circ_{b_1''\triangleright c_1''} y)\circ_{b_1'\rightarrow c_1'} z-a_1 x\circ_{c_1}(y\circ_{b_1} z)&\mbox{by (\ref{EQ14})},\\
&=(by\bullet ax)\bullet cz-by\bullet (ax\bullet cz)\\
&=\sum a_1(y\circ_{c_1''\triangleright b_1''} x)\circ_{c_1'\rightarrow b_1'} z-a_1 y\circ_{b_1}(x\circ_{c_1} z).
\end{align*}
We deduce (\ref{EQ18}).\\

3. Let $\calD=\{x,y,z\}$ be a set of cardinality 3 and let $V=\g_{A,\calD}$ be the free $\Phi$-prelie algebra generated by $x,y,z$. 
Then, for any $a,b,c\in A$,
\begin{align*}
\sum\sum (a'\rightarrow b')'\rightarrow c' (x\circ_{a''\triangleright b''} y)\circ_{(a'\rightarrow b')''\triangleright c''} z\\
-\sum\sum a'\rightarrow (b'\rightarrow c')' x\circ_{a''\triangleright (b'\rightarrow c')''}(y\circ_{b''\triangleright c''} z)\\
-\sum\sum (b'\rightarrow a')'\rightarrow c' (y\circ_{b''\triangleright a''} x)\circ_{(b'\rightarrow a')''\triangleright c''} z\\
+\sum\sum b'\rightarrow (a'\rightarrow c')'y\circ_{b''\triangleright (a'\rightarrow c')''}(x\circ_{a''\triangleright c''} z)&=0.
\end{align*}
Let $R$ be the space of relations defining $\Phi$-prelie algebras, seen as a subspace of the free operad
generated by $A$, concentrated in degree $2$. We deduce that the following element belongs to $A\otimes R$:
\begin{align*}
U&=\sum\sum (a'\rightarrow b')'\rightarrow c '(x\circ_{a''\triangleright b''} y)\circ_{(a'\rightarrow b')''\triangleright c''} z\\
&-\sum\sum a'\rightarrow (b'\rightarrow c')' x\circ_{a''\triangleright (b'\rightarrow c')''}(y\circ_{b''\triangleright c''} z)\\
&-\sum\sum (b'\rightarrow a')'\rightarrow c '(y\circ_{b''\triangleright a''} x)\circ_{(b'\rightarrow a')''\triangleright c''} z\\
&+\sum\sum b'\rightarrow (a'\rightarrow c')'y\circ_{b''\triangleright (a'\rightarrow c')''}(x\circ_{a''\triangleright c''} z).
\end{align*}
Necessarily, by identification of elements of the form
$x\circ (y\circ z)$, $(x\circ y)\circ z$, $y\circ (x\circ z)$ and $(y\circ x)\circ z$, we obtain relations (\ref{EQ14}) and (\ref{EQ15}). So $(A,\Phi)$ is an $\ell$CEDS.\end{proof}

\subsection{Combinatorial description of the generalized prelie products on trees}

\begin{defi}
Let $(A,\Phi)$ be an $\ell$CEDS, and $\calD$ be a set. Let $T,T'\in\calT_{A,\calD}$, $x$ be a vertex of $T'$ and $a\in A$. We denote by $T\diamond_{\Phi,a}^{(x)} T'$ the element of $\g_{A,\calD}$ obtained by the following process:\\

1. Graft $T$ on the vertex $x$ of $T'$. This process add an edge $e$ which is of type $a$.\\

2. Let $e_1,\ldots,e_k$ be the edges of $T'$ on the path between the root of $T'$ and the vertex $x$, and $a_1,\ldots,a_k$ their type. Compute
\[\sum_{i=1}^pa'_{1,i}\otimes\ldots\otimes a'_{k,i}\otimes a'_i=(\id^{\otimes (k-1)}\otimes\Phi)\circ(\id^{\otimes (k-2)}\otimes\Phi\otimes\id)\circ\ldots\circ (\Phi\otimes\id^{\otimes (k-1)})(a\otimes a_1\otimes\ldots\otimes a_k),\]
and take the sum for $1\leqslant i\leqslant p$ of trees obtained by typing $e$ by $a'_i$ and $e_j$ by $a'_{j,i}$ for any $j$.
\end{defi}

\begin{example}
Let us consider the following trees:
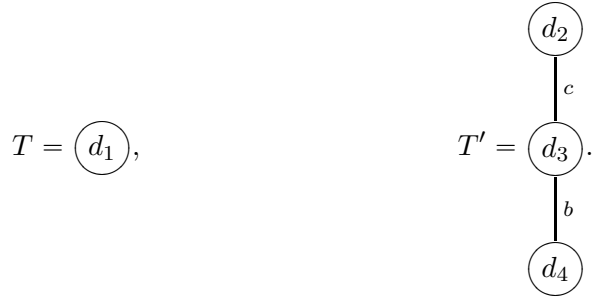
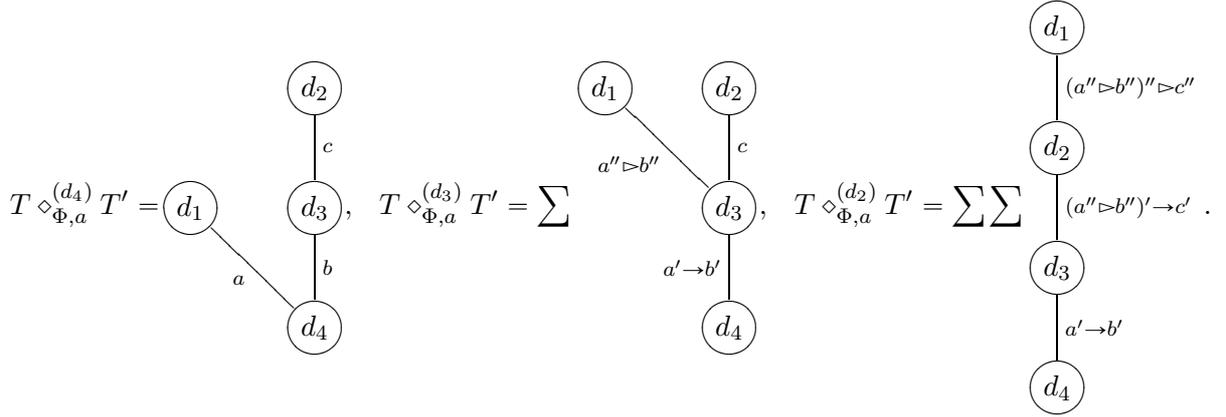
\begin{align*}
T&=\xymatrix{\rond{d_1}},& T'&=\substack{\hspace{5mm}\\ \xymatrix{\rond{d_2}\ar@{-}[d]^c\\ \rond{d_3}\ar@{-}[d]^b\\ \rond{d_4}}}.
\end{align*}
Then
\begin{align*}
T\diamond_{\Phi,a}^{(d_4)} T'&=\hspace{-1mm}\substack{\hspace{5mm}\\
\xymatrix{&\rond{d_2}\ar@{-}[d]^c\\ \rond{d_1}\ar@{-}[rd]_a
&\rond{d_3}\ar@{-}[d]^b\\ &\rond{d_4}}},&
T\diamond_{\Phi,a}^{(d_3)} T'&=\sum\substack{\hspace{5mm}\\ \xymatrix{\rond{d_1}\ar@{-}[rd]_{a''\triangleright b''}
&\rond{d_2}\ar@{-}[d]^c\\ 
&\rond{d_3}\ar@{-}[d]_{a'\rightarrow b'}\\ &\rond{d_4}}},&
T\diamond_{\Phi,a}^{(d_2)} T'&=\sum\sum
\substack{\hspace{8mm}\\ \xymatrix{\rond{d_1}\ar@{-}[d]^{(a''\triangleright b'')''\triangleright c''}\\
\rond{d_2}\ar@{-}[d]^{(a''\triangleright b'')'\rightarrow c'}
\\ \rond{d_3}\ar@{-}[d]^{a'\rightarrow b'}\\ \rond{d_4}}}.
\end{align*}
\end{example}

\begin{prop}\label{prop2.6}
Let $T,T'\in\calT_{A,\calD}$ and $a\in A$. Then, in $\g_{A,\calD}$,
\[ T\circ_a T'=\sum_{x\in V(T)'} T\diamond_{\Phi,a}^{(x)} T'.\]
\end{prop}

\begin{proof}
We proceed by induction on the number $n$ of vertices of $T'$. If $n=1$ let us put $T'=\xymatrix{\rond{d}}$. Then
\[T\circ_a T'=B_d(a T)=T\diamond_{\Phi,a}^{(\mathrm{root}(T'))} T'.\]
Let us assume the result at all ranks $<n$. We put $T'=B_d(a_1 T_1,\ldots, a_k T_k)$. Applying the induction hypothesis to $T_1,\ldots, T_k$, we obtain
\begin{align*}
T\circ_a T'&=B_d(aT,a_1 T_1,\ldots, a_k T_k)+\sum_{i=1}^k\sum_{x\in V(T_i)}\underbrace{\sum B_d(a_1 T_1,\ldots, (a'\rightarrow a_i') T\diamond_{\Phi, a''\triangleright a_i''}^{(x)} T_i,\ldots, a_k T_k)}_{T\diamond_{\Phi,a}^{(x)} T'}\\
&=T\diamond_{\Phi,a}^{(\mathrm{root}(T'))} T'+\sum_{i=1}^k\sum_{x\in V(T_i)} T\diamond_{\Phi,a}^{(x)} T'\\
&=\sum_{x\in V(T')}T\diamond_{\Phi,a}^{(x)} T'.\qedhere
\end{align*}\end{proof}

\begin{cor}
Let $(A,\Phi)$ be an $\ell$CEDS and $\calD$ a nonempty set. We denote by $\g'_{A,\calD}$ the prelie subalgebra of $A\otimes\g_{A,\calD}$ generated by the elements $a\xymatrix{\rond{X}}$, with $a\in A$ and $a\in\calD$.
\begin{enumerate}
\item The following conditions are equivalent:
\begin{itemize}
\item $\g'_{A,\calD}=A\otimes\g_{A,\calD}$.
\item $\Phi$ is surjective.
\end{itemize}
\item The following conditions are equivalent:
\begin{itemize}
\item The prelie algebra $\g'_{A,\calD}$ is free.
\item $\Phi$ is injective.
\end{itemize}
\end{enumerate}
\end{cor}

\begin{proof}
1. $\Longrightarrow$. Let $X\in\calD$ and $a,b\in A$. Then
\[a\xymatrix{\rond{X}}\circ b\xymatrix{\rond{X}}=\sum a'\rightarrow b'\substack{\hspace{3mm}\\ \xymatrix{\rond{X}\ar@{-}[d]^{a''\triangleright b''}\\ \rond{X}}}.\]
Let us assume that $\g'_{A,\calD}=A\otimes\g_{A,\calD}$. Let $a,b\in A$. Because of the graduation by the number of vertices, there exist elements $c_i,d_i\in A$ such that
\[\sum_{i=1}^k c_i\xymatrix{\rond{X}}\circ d_i\xymatrix{\rond{X}}=a\substack{\hspace{3mm}\\ \xymatrix{\rond{X}\ar@{-}[d]^b\\ \rond{X}}}.\]
By the preceding computation,
\[\sum_{i=1}^k\sum c_i'\rightarrow d_i'\substack{\hspace{3mm}\\ \xymatrix{\rond{X}\ar@{-}[d]^{c_i''\triangleright d_i''}\\ \rond{X}}}=a\substack{\hspace{3mm}\\ \xymatrix{\rond{X}\ar@{-}[d]^b\\ \rond{X}}}.\]
Consequently,
\[\Phi\left(\sum_{i=1}^k c_i\otimes d_i\right)=a\otimes b,\]
so $\Phi$ is surjective.\\

1. $\Longleftarrow$. Let us denote by $\g_{A\otimes\K\calD}$ be the prelie algebra of rooted trees decorated by $A\otimes\K\calD$, the decorations being linear in each vertex, and let $\psi:\g_{A\otimes\calD}\longrightarrow A\otimes\g_{A,\calD}$
be the unique prelie algebra morphism sending the tree with one vertex decorated by $a\otimes d$, with $a\in A$ and $d\in\calD$, to $a\xymatrix{\rond{d}}$. We plan to prove that $\psi$ is surjective. 
Let us prove that for any $a\in A$, any tree $T$, $a T\in\im(\psi)$. We proceed by induction on the number $n$ of vertices of $T$. If $n=1$, it is obvious. 
Otherwise, let us put $T=B_d(a_1T_1,\ldots,a_kT_k)$. We proceed by induction on $k$. This is obvious if $k=0$. Otherwise, let us put $x=T_1$ and $Y=B_d(a_2 T_2,\ldots,a_k T_k)$. 
By the induction hypothesis on the number of vertices, $x$ and $y$ belong to $\im(\psi)$. As $\Phi$ is surjective, let us choose $\sum c_i\otimes d_i\in A\otimes A$ such that $\Phi(\sum c_i\otimes d_i)=a\otimes a_1$. Then
\[\sum c_i x\circ d_i y=T+\mbox{a sum of trees with $n$ vertices and with $k-1$ trees born from the root}.\]
By the induction hypothesis on $k$, this sum belongs to $\im(\psi)$, so $T\in\im(\psi)$.\\

2. $\Longrightarrow$. If $\g'_{A,\calD}$ is free, then because of the graduation, it is freely generated by its elements $a\xymatrix{\rond{X}}$. Let us assume that $\Phi(\sum a_i\otimes b_i)=0$. Then, for any $d\in\calD$,
\begin{align*}
\sum a_i\otimes\xymatrix{\rond{d}}\circ b_i\xymatrix{\rond{d}}=\sum\sum a_i'\rightarrow b_i'\substack{\hspace{3mm}\\ \xymatrix{\rond{d}\ar@{-}[d]^{a_i''\triangleright b_i''}\\ \rond{d}}}=0.
\end{align*}
By freeness of $A\otimes\g_{A,\calD}$, $\sum a_i\otimes b_i=0$, so $\phi$ is injective.\\

2. $\Longleftarrow$. We now show that $\psi$ is injective. 
The combinatorial description of the generalized prelie product implies by a direct induction that for any tree $T$ which vertices are decorated by $A\otimes\K\calD$, $\psi(T)$ is a sum of terms $a\otimes T'$, where $T'$ is a tree with the same form as $T$, 
the decorations by $\calD$ being conserved, and $a$ and the types of the edges of $T'$ being obtained from the decoration by elements of $A$ by iterated applications of $\Phi$, depending uniquely of the form of $T$. 
As $\Phi$ is injective, $\psi$ is injective.\end{proof}

\subsection{Combinatorial description of the underlying operad}

Let $(A,\Phi)$ be an $\ell$CEDS. The description of the free $\Phi$-prelie algebras induce a description of this operad. We shall use the formalism of operads in the category of species. 
For any set $\calD$, $\prelie_\Phi[\calD]$ is the space generated by the set of trees $\calT_A[\calD]$ $A$-typed, the types being linear on each edge, and which set of vertices is $\calD$. The operadic composition is given in the following way:
if $\calD$ and $\calD'$ are sets, $d\in\calD$, $T\in\calT_A[\calD]$ and $T'\in\calT_A[\calD']$, let us consider the unique $\Phi$-prelie algebra morphism $\psi:\g_{A,\calD}\longrightarrow\g_{A,\calD\sqcup\calD'\setminus\{d\}}$ 
sending $\xymatrix{\rond{X}}$ on itself if $X\neq d$ and on $T'$ otherwise; then $T\circ_d T'=\psi(T)$. From the combinatorial description of the prelie products with graftings, we deduce the following description of this operadic composition:
\begin{itemize}
\item Let us denote by $T_1,\ldots,T_k$ be the subtrees of $T$ born from the vertex $d$ of $T$, and by $a_i$ the type of the edge relating this vertex to the root of $T_i$ for any $i$.
\item Denote $T_0=T\setminus (T_1\sqcup\ldots\sqcup T_k)$ the subtree of $T$ obtained by deleting the subtrees $T_1,\ldots,T_k$ and by $T_0*_d T'$ the tree obtained by identification of the vertex $d$ of $T$ with the root of $T'$.
\end{itemize}
Then
\[T\circ_d T'=\sum_{d_1,\ldots d_k\in V(T')} T_1\diamond_{\Phi,a_1}^{(d_1)}(\ldots T_k\diamond_{\Phi,a_k}^{(d_k)}(T_0 *_d T')\ldots).\]

The product $\circ_a\in\prelie_\Phi(2)$ is represented by the tree
\[\circ_a=\substack{\hspace{3mm}\\ \xymatrix{\rond{1}\ar@{-}[d]^a\\ \rond{2}}}.\]
The unit is the tree $I=\xymatrix{\rond{1}}\in\prelie_\Phi(1)$.

\begin{example}
If $x,y,z\in\calD$, $u,v,w\in\calD'$, and $a,b,c,d\in A$,
\begin{align*}
\substack{\hspace{3mm}\\ \xymatrix{\rond{y}\ar@{-}_a[rd]&\rond{z}\ar@{-}^b[d]\\ &\rond{x}}}\circ_x\substack{\hspace{5mm}\\ \xymatrix{\rond{w}\ar@{-}[d]^d\\ \rond{v}\ar@{-}[d]^c\\ \rond{u}}}
&=\substack{\hspace{5mm}\\ \xymatrix{&\rond{w}\ar@{-}[d]^d&\\ \rond{y}\ar@{-}[rd]_a&\rond{v}\ar@{-}[d]^c&\rond{z}\ar@{-}[ld]^b\\ &\rond{u}&}}+\sum\substack{\hspace{5mm}\\ \xymatrix{&\rond{w}\ar@{-}[d]^d&\rond{z}\ar@{-}[ld]^{b''\triangleright c''}\\ 
\rond{y}\ar@{-}[rd]_a&\rond{v}\ar@{-}[d]^{b'\rightarrow c'}&\\ &\rond{u}&}}
+\sum\sum\substack{\hspace{8mm}\\ \xymatrix{&\rond{z}\ar@{-}[d]^{(b''\triangleright c'')''\triangleright d''}&\\
&\rond{w}\ar@{-}[d]^{(b''\triangleright c'')'\rightarrow d'}&\\ \rond{y}\ar@{-}[rd]_a&\rond{v}\ar@{-}[d]^{b'\rightarrow c'}
&\\ &\rond{u}&}}\\
&+\sum\substack{\hspace{5mm}\\ \xymatrix{\rond{y}\ar@{-}[rd]_{a''\triangleright c''}&\rond{w}\ar@{-}[d]^d&\\ 
&\rond{v}\ar@{-}[d]_{a'\rightarrow c'}&\rond{z}\ar@{-}[ld]^b\\ &\rond{u}&}}
+\sum\sum\substack{\hspace{5mm}\\ \xymatrix{\rond{y}\ar@{-}[rd]_{a''\triangleright (b'\rightarrow c')''}&\rond{w}\ar@{-}[d]^d
&\rond{z}\ar@{-}[ld]^{b''\triangleright c''}\\ 
&\rond{v}\ar@{-}[d]^{a'\rightarrow (b'\rightarrow c')'}&\\ &\rond{u}&}}\\
&+\sum\sum\substack{\hspace{8mm}\\ \xymatrix{&\rond{z}\ar@{-}[d]^{(b''\triangleright c'')''\triangleright d''}&\\
\rond{y}\ar@{-}[rd]_{a''\triangleright (b'\rightarrow c')''}&\rond{w}\ar@{-}[d]^{(b''\triangleright c'')'\rightarrow d'}&\\ 
&\rond{v}\ar@{-}[d]^{a'\rightarrow (b'\rightarrow c')'}&\\ &\rond{u}&}}
+\sum\sum\hspace{-6mm}\substack{\hspace{5mm}\\ \xymatrix{&\rond{y}\ar@{-}[d]^{(a''\triangleright c'')''\triangleright d''}&\\
&\rond{w}\ar@{-}[d]^{(a''\triangleright c'')'\rightarrow d'}&\\ 
&\rond{v}\ar@{-}[d]_{a'\rightarrow c'}&\rond{z}\ar@{-}[ld]^b\\ &\rond{u}&}}\\
&+\sum\sum\sum\substack{\hspace{11mm}\\ \xymatrix{&\rond{y}\ar@{-}[d]^{(a''\triangleright (b'\rightarrow c')'')''\triangleright d''}&\\
&\rond{w}\ar@{-}[d]_{(a''\triangleright (b'\rightarrow c')'')'\rightarrow d'}&\rond{z}\ar@{-}[ld]^{b''\triangleright c''}\\ 
&\rond{v}\ar@{-}[d]^{a'\rightarrow(b'\rightarrow c')'}&\\ &\rond{u}&}}\\
&+\sum\sum\sum\sum\hspace{-5mm}\substack{\hspace{11mm}\\ \xymatrix{
\rond{y}\ar@{-}[rd]_{(a''\triangleright(b'\rightarrow c')'')''\triangleright ((b''\triangleright c'')'\rightarrow d')''}
&\rond{z}\ar@{-}[d]^{(b''\triangleright c'')''\triangleright d''}&\\
&\rond{w}\ar@{-}[d]_{(a''\triangleright(b'\rightarrow c')'')'\rightarrow ((b''\triangleright c'')'\rightarrow d')'}&\\ 
&\rond{v}\ar@{-}[d]_{a'\rightarrow (b'\rightarrow c')'}&\\ &\rond{u}&}}\hspace{-1cm}.
\end{align*}\end{example}

\begin{remark}
If $A=\K$, we can forget about the types of the edges and we obtain operadic structures on rooted trees. If $\Phi(1\otimes 1)=1\otimes 1$, we recover the description of the operad of prelie algebras as done in\cite{Chapoton2001}.
If $\Phi(1\otimes 1)=0$, we recover the description of the operad of NAP algebras of\cite{Livernet2006}. 
\end{remark}

\subsection{Pre-Lie and non-associative permutative products}

We now look for prelie and non-associative permutative products in $\Phi$-prelie algebras. 
More formally, we look for operad morphisms from the operad $\prelie$ of prelie algebras or from the operad $\NAP$ of non-associative permutative (briefly, NAP) algebras\cite{Livernet2006} to the operad $\prelie_\Phi$. 
Recall that the operad $\prelie$ is generated by $\circ\in\prelie(2)$, with the relation
\[p\circ (I,p)-p\circ(p,I)=p\circ (I,p)^{(12)}-p\circ(p,I)^{(12)},\]
whereas the operad $\NAP$ is generated by $\circ\in\NAP(2)$, with the relation
\[p\circ (I,p)=p\circ (I,p)^{(12)}.\]

Let us first recall this definition of\cite{Foissy46}:

\begin{defi}
Let $(A,\Phi)$ be an $\ell$EAS, $\lambda\in\K$ and $a\in A$. We shall say that $a$ is a special vector of $(A,\Phi)$ of eigenvalue $\lambda$ if $\Phi(a\otimes a)=\lambda a\otimes a$. 
\end{defi}

We prove in\cite[Lemma 4.4]{Foissy46} that if $(A,\Phi)$ has a nonzero special vector of eigenvalue $\lambda$, then $\lambda=0$ or $1$. 

\begin{prop}\label{prop2.9}
Let $(A,\Phi)$ be an $\ell$CEDS.
\begin{enumerate}
\item The prelie products of $\prelie_\Phi$ are the products $\circ_a$, where $a$ is a special vector of $(A,\Phi)$ of eigenvalue $1$.
\item The NAP products of $\prelie_\Phi$ are the products $\circ_a$, where $a$ is a special vector of $(A,\Phi)$ of eigenvalue $0$.
\end{enumerate}
\end{prop}

\begin{proof} We use the description of $\prelie_\Phi$ in terms of typed rooted trees. From the combinatorial description of the operad $\prelie_\Phi$, for any $a,b\in A$,
\begin{align*}
\circ_a\circ (I,\circ_b)&=\substack{\hspace{3mm}\\ \xymatrix{\rond{1}\ar@{-}[dr]_a&\rond{2}\ar@{-}[d]^b\\&\rond{3}}}+\sum\substack{\hspace{5mm}\\ \xymatrix{\rond{1}\ar@{-}[d]^{a''\triangleright b''}\\
\rond{2}\ar@{-}[d]^{a'\rightarrow b'}\\ \rond{3}}},&
\circ_a\circ (\circ_b,I)&=\substack{\hspace{5mm}\\ \xymatrix{\rond{1}\ar@{-}[d]^a\\ \rond{2}\ar@{-}[d]^b\\ \rond{3}}}.
\end{align*}
Let $p=\circ_a+\circ_b^{op}\in\prelie_\Phi(2)$, with $a,b\in A$. Then
\begin{align*}
p\circ (I,p)&=\substack{\hspace{3mm}\\ \xymatrix{\rond{1}\ar@{-}[rd]_{a} &\rond{2}\ar@{-}[d]^{a}\\ &\rond{3}}}
+\sum\substack{\hspace{5mm}\\
\xymatrix{\rond{1}\ar@{-}[d]^{a''\triangleright a''}\\ \rond{2}\ar@{-}[d]^{a'\rightarrow a'}\\ \rond{3}}}
+\substack{\hspace{3mm}\\ \xymatrix{\rond{1}\ar@{-}[rd]_{a} &\rond{3}\ar@{-}[d]^{b}\\ &\rond{2}}}
+\sum\substack{\hspace{5mm}\\ 
\xymatrix{\rond{1}\ar@{-}[d]^{a''\triangleright b''}\\ \rond{3}\ar@{-}[d]^{a'\rightarrow b'}\\ \rond{2}}}
+\substack{\hspace{5mm}\\ \xymatrix{\rond{2}\ar@{-}[d]^{a}\\ \rond{3}\ar@{-}[d]^{b}\\ \rond{1}}}
+\substack{\hspace{5mm}\\ \xymatrix{\rond{3}\ar@{-}[d]^{b}\\ \rond{2}\ar@{-}[d]^{b}\\ \rond{1}}},\\
p\circ (p,I)&=\substack{\hspace{5mm}\\ \xymatrix{\rond{1}\ar@{-}[d]^{a}\\ \rond{2}\ar@{-}[d]^{a}\\ \rond{3}}}
+\substack{\hspace{5mm}\\ \xymatrix{\rond{2}\ar@{-}[d]^{b}\\ \rond{1}\ar@{-}[d]^{a}\\ \rond{3}}}
+\substack{\hspace{3mm}\\ \xymatrix{\rond{1}\ar@{-}[rd]_{a} &\rond{3}\ar@{-}[d]^{b}\\ &\rond{2}}}
+\sum\substack{\hspace{5mm}\\ 
\xymatrix{\rond{3}\ar@{-}[d]^{b''\triangleright a''}\\ \rond{1}\ar@{-}[d]^{b'\rightarrow a'}\\ \rond{2}}}
+\substack{\hspace{3mm}\\ \xymatrix{\rond{2}\ar@{-}[rd]_{b} &\rond{3}\ar@{-}[d]^{b}\\ &\rond{1}}}
+\sum\substack{\hspace{5mm}\\
\xymatrix{\rond{3}\ar@{-}[d]^{b''\triangleright b''}\\ \rond{2}\ar@{-}[d]^{b'\rightarrow b'}\\ \rond{1}}}.
\end{align*}

1. In the prelie relation $p\circ (I,p)-p\circ(p,I)=p\circ (I,p)^{(12)}-p\circ(p,I)^{(12)}$, let us identify the type of the edges of the trees of the form
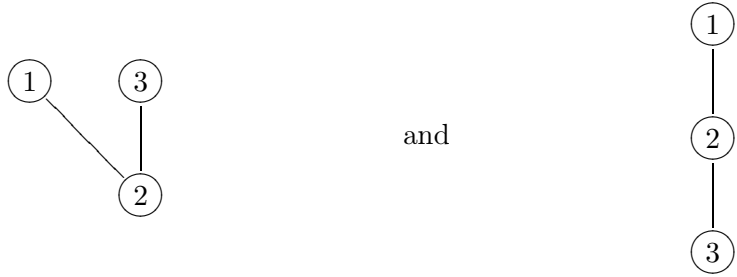
\begin{align*}
&\substack{\hspace{3mm}\\ \xymatrix{\rond{1}\ar@{-}[rd] &\rond{3}\ar@{-}[d]\\ &\rond{2}}}&\mbox{and}&&
\substack{\hspace{5mm}\\ \xymatrix{\rond{1}\ar@{-}[d]\\ \rond{2}\ar@{-}[d]\\ \rond{3}}}.
\end{align*}
The first one gives $a\otimes b-a\otimes b=0-b\otimes b$, so $b\otimes b=0$ and consequently $b=0$. The second one gives $\sum a'\rightarrow a'\otimes a''\triangleright a''-a\otimes a=0-a\otimes b=0$,
so $\Phi(a\otimes a)=a\otimes a$. Conversely, if $a\otimes a=\Phi(a\otimes a)$ and $p=\circ_a$,
\[p\circ (I,p)-p\circ(p,I)=\substack{\hspace{3mm}\\ \xymatrix{\rond{1}\ar@{-}[rd]_{a} &\rond{2}\ar@{-}[d]^{a}\\ &\rond{3}}}
+\sum\substack{\hspace{5mm}\\ \xymatrix{\rond{1}\ar@{-}[d]^{a''\triangleright a''}\\ 
\rond{2}\ar@{-}[d]^{a'\rightarrow a'}\\ \rond{3}}}
-\substack{\hspace{5mm}\\ \xymatrix{\rond{1}\ar@{-}[d]^{a}\\ \rond{2}\ar@{-}[d]^{a}\\ \rond{3}}}
=\substack{\hspace{3mm}\\ \xymatrix{\rond{1}\ar@{-}[rd]_{a} &\rond{2}\ar@{-}[d]^{a}\\ &\rond{3}}},\]
which is invariant under the action of $(12)$. So $\circ_a$ is prelie.\\

2. In the NAP relation $p\circ (I,p)=p\circ (I,p)^{(12)}$, let us identify the trees of the form
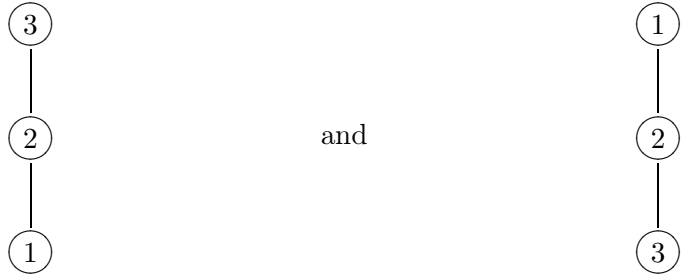
\begin{align*}
&\substack{\hspace{5mm}\\ \xymatrix{\rond{3}\ar@{-}[d]\\ \rond{2}\ar@{-}[d]\\ \rond{1}}}&\mbox{and}&&
\substack{\hspace{5mm}\\ \xymatrix{\rond{1}\ar@{-}[d]\\ \rond{2}\ar@{-}[d]\\ \rond{3}}}.
\end{align*}
The first one gives $b\otimes b=0$, so $b=0$. The second one gives $\sum a'\rightarrow a'\otimes a''\triangleright a''=0$. Conversely, if $\Phi(a\otimes a)=0$ and $p=\circ_a$,
\[p\circ(I,p)=p=\substack{\hspace{3mm}\\ \xymatrix{\rond{1}\ar@{-}[rd]_{a} &\rond{2}\ar@{-}[d]^{a}\\ &\rond{3}}},\]
which is invariant under the action of $(12)$. So $\circ_a$ is non-associative permutative.
\end{proof}

\begin{example}\label{ex2.5}
Let $(\Omega,\rightarrow,\triangleright)$ be an EAS and $(\K\Omega,\Phi)$ be its linearization. Let $a=\displaystyle\sum_{\alpha\in\Omega} a_\alpha\alpha\in\K\Omega$. Then $a$ is a special vector of eigenvalue 1 if, and only if
\begin{align*}
&\forall\alpha',\beta'\in\Omega,&\sum_{\substack{\alpha,\beta\in\Omega,\\
(\alpha\rightarrow\beta,\alpha\triangleright\beta)=(\alpha',\beta')}} a_\alpha a_\beta&=a_{\alpha'}a_{\beta'}.
\end{align*}
It is a special vector of eigenvalue 0 if, and only if
\begin{align*}
&\forall\alpha',\beta'\in\Omega,&\sum_{\substack{\alpha,\beta\in\Omega,\\
(\alpha\rightarrow\beta,\alpha\triangleright\beta)=(\alpha',\beta')}} a_\alpha a_\beta&=0.
\end{align*}
\end{example}

As a consequence, if $a$ is a special vector of $(A,\Phi)$ of eigenvalue $1$, then $(\g_{A,\calD},\circ_a)$ is a pre-Lie algebra. Let us study the structure of this pre-Lie algebra in the particular case where $a$ is a left unit.

\begin{prop}\label{prop2.10}
Let $a$ be a nonzero special vector of $A$ of eigenvalue 1 and let $f\in A^*$ such that $f(a)=1$. The pre-Lie algebra $(\g_{A,\calD},\circ_a)$ is freely generated by the space of trees $T$ such that any edge born from the root of $T$ is typed by an element of $\ker(f)$. 
\end{prop}

\begin{proof} 
We fix a basis $(e_i)_{i\in I'}$ of $\ker(f)$ and fix a total order $\leqslant$ on $I'$.
Putting $I=I\sqcup\{\infty\}$ and $e_\infty=a$, we obtain a basis $(e_i)_{i\in I}$ of $A$, and $I$ is totally ordered, $\infty$ being its greatest element. We also fix a total order $\leqslant$ on $\calD$.\\

A basis of $\g_{A,\calD}$ is given by the set $\bfT$ of rooted trees which vertices are decorated by $\calD$ and the edges are typed by elements of the basis $(e_i)_{i\in I}$. 
For any $n\geqslant 1$, we denote by $\bfT_n$ the set of elements of $\bfT$ with $n$ vertices. We now define a total order on $\bfT_n$ in the following way.
\begin{itemize}
\item If $n=1$, the considered trees are reduced to a single vertex decorated by an element of $\calD$. The total order of $\calD$ induced a total order on $\bfT_1$.
\item Let us assume that the total order is totally defined on $\bfT_k$ for any $k<n$. We then define a total order on $\displaystyle\{e_i,i\in I\}\times\left(\sqcup_{k<n}\bfT_k\right)$ in the following way: $e_iT'<e_j T'$ if one of the following condition holds:
\begin{itemize}
\item $i<j$ in $I$.
\item $i=j$ and $T$ has strictly less vertices than $T'$.
\item $i=j$, $T$ and $T'$ have the same number of vertices $p$ and $T<T'$ in $\bfT_p$.
\end{itemize}
Let us consider $T,T'\in\bfT_n$ which we write $T=B_d(a_1 T_1,\ldots,a_k T_k)$ and $T'=B_{d'}(a'_1 T'_1,\ldots,a'_l T'_l)$, with
\begin{align*}
&a_1T_1\geqslant\ldots\geqslant a_k T_k,&&a'_1T'_1\geqslant\ldots\geqslant a'_k T'_k.
\end{align*}
We shall say that $T<T'$ in $\bfT_n$ if one of the following condition holds:
\begin{itemize}
\item $d<d'$ in $\calD$.
\item $d=d'$ and $k<l$.
\item $d=d'$, $k=l$ and there exists $i\in\{1,\ldots,k\}$ such that
\begin{align*}
a_1T_1&=a'_1T'_1,&&\ldots,&a_{i-1}T_{i-1}&=a'_{i-1}T'_{i-1},&a_iT_i&<a'_iT'_i.
\end{align*}\end{itemize}\end{itemize}

We denote by $\bfT'$ the set of trees $T\in\bfT$, such that no edge born the root is typed by $a=e_\infty$. 
A basis $\bfT_{\bfT'}$ of the free pre-Lie algebra $\g_{\bfT'}$ generated by $\bfT'$ is given by non-typed rooted trees $T$, which vertices are decorated by elements of $\bfT'$.
There exists a unique pre-Lie algebra morphism $\phi:(\g_{\bfT'},\circ)\longrightarrow (\g_{A,\calD},\circ_a)$ sending any tree $T\in\bfT_{\bfT'}$ with a single vertex decorated by $T\in\bfT'$ to the tree $T$.
Let us consider a tree $T=B_{T'}(T_1,\ldots,T_k)\in\g_{\bfT'}$, with $T'=B_d(a'_1T'_1,\ldots,a'_lT'_l)\in\bfT'$.
$\phi(T)$ is a nonzero linear span of trees with the same number of vertices, obtained by generalized graftings of the decorations of $T$. We denote by $\pi(T)$ the greatest of these trees for the the total order defined earlier. 
In all the trees appearing in $\phi(T)$, the decoration of the root is the decoration $d$ of the root of $T'$. If we consider a tree obtained by grafting trees of $\phi(T_1),\ldots,\phi(T_k)$ on the root of $T'$, the fertility of the root is $k+l$.
If one of the trees of $\phi(T_1),\ldots,\phi(T_k)$ is not grafted on the root of $T'$, the fertility of the root is $<k+l$. By definition of the total order, we obtain
\[\phi(T)=\pi(B_d(a\phi(T_1),\ldots,a\phi(T_k),a'_1T'_1,\ldots,a'_lT'_l)).\]
Up to a permutation, we can assume that $\pi(T_1)\geqslant\ldots\geqslant\pi(T_k)$. Then any tree in $\phi(T)$ is smaller than $B_d(a\pi(T_1),\ldots,a\pi(T_k),a'_1T'_1,\ldots,a'_lT'_l)$. Hence,
\[\pi(T)=B_d(a\pi(T_1),\ldots,\pi(T_k),a'_1T'_1,\ldots,a'_lT'_l).\]
In other words, $\pi(T)$ is obtained in the following process: for any vertex $v$ of $T$ which is not the root, graft the decoration of $v$ on the root of the decoration of the father of $v$.
This defines a bijection from the set $\bfT_{\bfT'}$ (basis of $\g_{\bfT'}$) to $\bfT$ (basis of $\g_{A,\calD})$. By triangularity of $\phi$, $\phi$ is bijective.\end{proof}

\begin{remark}
A similar order was used in\cite{Li2021} in free pre-Lie algebras in order to define and study Gröbner-Shirshov bases on these objects.
\end{remark}

\section{Generalized permutative algebras}

\subsection{Definition and Koszul duality}

\begin{defi}\label{defi3.1}
Let $(A,\Phi)$ be a dual $\ell$CEDS. An $(A,\Phi)$-permutative algebra is a pair $(V,\diamond)$ where $V$ is a vector space and
\[\diamond:\left\{\begin{array}{rcl}
A&\longrightarrow&\L(V\otimes V,V)\\
a&\longrightarrow&\diamond_a:\left\{\begin{array}{rcl}
V\otimes V&\longrightarrow&V\\
v\otimes w&\longrightarrow&v\diamond_a w,
\end{array}\right.\end{array}\right.\]
such that for any $x,y,z\in V$, for any $a,b\in A$,
\begin{align}
\label{EQ20}(x\diamond_b y)\diamond_a z&=\sum x\diamond_{a'\curvearrowright b'}(y\diamond_{a''\blacktriangleright b''} z),\\
\label{EQ21}x\diamond_a (y\diamond_b z)&=y\diamond_b (x\diamond_a z),
\end{align}
with Sweedler's notation
\[\Phi(a\otimes b)=\sum a'\curvearrowright b'\otimes a''\blacktriangleright b''.\]
We denote by $\perm_\Phi$ the operad of $(A,\Phi)$-permutative algebras.
\end{defi}

\begin{prop}\label{prop3.2}
Let $(A,\Phi)$ be a finite-dimensional $\ell$CEDS. The Koszul dual of the quadratic operad $\prelie_\Phi$ is $\perm_{\Phi^*}$. 
\end{prop}

\begin{proof}
For any vector space $E$, let us denote by $\free_E$ the free operad generated by $E\subseteq\free_E$. As a vector space, with tree-operadic notations,
\begin{align*}
\free_E(2)&=\left(\bfdeux{$1$}{$2$}\otimes E\right)\bigoplus\left(\bfdeux{$2$}{$1$}\otimes E\right),\\
\free_E(3)&=\left(\bftroisun{$1$}{$2$}{$3$}\otimes E^{\otimes 2}\right)\bigoplus\left(\bftroisun{$1$}{$3$}{$2$}\otimes E^{\otimes 2}\right)\bigoplus\left(\bftroisun{$2$}{$1$}{$3$}\otimes E^{\otimes 2}\right)\\
&\bigoplus\left(\bftroisun{$2$}{$3$}{$1$}\otimes E^{\otimes 2}\right)\bigoplus\left(\bftroisun{$3$}{$1$}{$2$}\otimes E^{\otimes 2}\right)\bigoplus\left(\bftroisun{$3$}{$2$}{$1$}\otimes E^{\otimes 2}\right)\\
&\bigoplus\left(\bftroisdeux{$1$}{$2$}{$3$}\otimes E^{\otimes 2}\right)\bigoplus\left(\bftroisdeux{$1$}{$3$}{$2$}\otimes E^{\otimes 2}\right)\bigoplus\left(\bftroisdeux{$2$}{$1$}{$3$}\otimes E^{\otimes 2}\right)\\
&\bigoplus\left(\bftroisdeux{$2$}{$3$}{$1$}\otimes E^{\otimes 2}\right)\bigoplus\left(\bftroisdeux{$3$}{$1$}{$2$}\otimes E^{\otimes 2}\right)\bigoplus\left(\bftroisdeux{$3$}{$2$}{$1$}\otimes E^{\otimes 2}\right),
\end{align*}
and the action of the symmetric groups is given by permutations on the indices on the leaves. The composition is given by grafting on the leaves: for any $a,b\in E$,
\begin{align*}
\bfdeux{$1$}{$2$}\otimes a\circ_1\bfdeux{$1$}{$2$}\otimes b&=\bftroisun{$1$}{$2$}{$3$}\otimes a\otimes b,&
\bfdeux{$1$}{$2$}\otimes a\circ_2\bfdeux{$1$}{$2$}\otimes b&=\bftroisdeux{$1$}{$2$}{$3$}\otimes a\otimes b,\\
\bfdeux{$1$}{$2$}\otimes a\circ_1\bfdeux{$2$}{$1$}\otimes b&=\bftroisun{$2$}{$1$}{$2$}\otimes a\otimes b,&
\bfdeux{$1$}{$2$}\otimes a\circ_2\bfdeux{$2$}{$1$}\otimes b&=\bftroisdeux{$1$}{$3$}{$2$}\otimes a\otimes b,\\
\bfdeux{$2$}{$1$}\otimes a\circ_1\bfdeux{$1$}{$2$}\otimes b&=\bftroisdeux{$3$}{$1$}{$2$}\otimes a\otimes b,&
\bfdeux{$2$}{$1$}\otimes a\circ_2\bfdeux{$1$}{$2$}\otimes b&=\bftroisun{$2$}{$3$}{$1$}\otimes a\otimes b,\\
\bfdeux{$2$}{$1$}\otimes a\circ_1\bfdeux{$2$}{$1$}\otimes b&=\bftroisdeux{$3$}{$2$}{$1$}\otimes a\otimes b,&
\bfdeux{$2$}{$1$}\otimes a\circ_2\bfdeux{$2$}{$1$}\otimes b&=\bftroisun{$3$}{$2$}{$1$}\otimes a\otimes b.
\end{align*}
Moreover, if $E$ is finite-dimensional, $\free_E^*$ is identified with $\free_{E^*}$ with a pairing compatible with the action of the symmetric groups, such that:
\begin{itemize}
\item For any operadic trees $T$, $T'$ with two leaves, $f\in E^*$ and $a\in E$,
\[\langle T\otimes f,T'\otimes a\rangle=f(a)\delta_{T,T'}\varepsilon_T,\]
with
\[\begin{array}{|c|c|c|}
\hline T&\bfdeux{$1$}{$2$}&\bfdeux{$2$}{$1$}\\
\hline\varepsilon_T&1&-1\\
\hline\end{array}\]
\item For any operadic trees $T$, $T'$ with three leaves, $f,g\in E^*$ and $a,b\in E$,
\[\langle T\otimes f\otimes g,T'\otimes a\otimes b\rangle=f(a)g(b)\delta_{T,T'}\varepsilon_T,\]
with
\[\begin{array}{|c|c|c|c|c|c|c|}
\hline T&\bftroisun{$1$}{$2$}{$3$}&\bftroisun{$1$}{$3$}{$2$}&\bftroisun{$2$}{$1$}{$3$}&\bftroisun{$2$}{$3$}{$1$}&\bftroisun{$3$}{$1$}{$2$}&\bftroisun{$3$}{$2$}{$1$}\\
\hline\varepsilon_T&1&-1&-1&1&1&-1\\
\hline\hline T&\bftroisdeux{$1$}{$2$}{$3$}&\bftroisdeux{$1$}{$3$}{$2$}&\bftroisdeux{$2$}{$1$}{$3$}
&\bftroisdeux{$2$}{$3$}{$1$}&\bftroisdeux{$3$}{$1$}{$2$}&\bftroisdeux{$3$}{$2$}{$1$}\\
\hline\varepsilon_T&-1&1&1&-1&-1&1\\
\hline\end{array}\]
The operad $\prelie_\Phi$ is the quotient of $\free_A$ by the ideal $I$ generated by 
\[X(a\otimes b)=\bftroisdeux{$1$}{$2$}{$3$}\otimes a\otimes b-\bftroisun{$1$}{$2$}{$3$}\otimes\Phi(a\otimes b)-\bftroisdeux{$2$}{$1$}{$3$}\otimes b\otimes a+\bftroisun{$2$}{$1$}{$3$}\otimes\Phi(b\otimes a),\]
with $a,b\in\Omega$. As a vector space, $I(3)$ is of dimension $3\dim(A)^2$, and is generated by the elements
\begin{align*}
X(a\otimes b),\\
Y(a\otimes b)&=X(a\otimes b)^{(23)}\\
&=\bftroisdeux{$1$}{$3$}{$2$}\otimes a\otimes b-\bftroisun{$1$}{$3$}{$2$}\otimes\Phi(a\otimes b)-\bftroisdeux{$3$}{$1$}{$2$}\otimes b\otimes a+\bftroisun{$3$}{$1$}{$2$}\otimes\Phi(b\otimes a),\\
Z(a\otimes b)&=X(a\otimes b)^{(13)}\\
&=\bftroisdeux{$3$}{$2$}{$1$}\otimes a\otimes b-\bftroisun{$3$}{$2$}{$1$}\otimes\Phi(a\otimes b)-\bftroisdeux{$2$}{$3$}{$1$}\otimes b\otimes a+\bftroisun{$2$}{$3$}{$1$}\otimes\Phi(b\otimes a),
\end{align*}
with $a,b\in A$. Recall that the Koszul dual $\prelie^!_\Phi$ is the quotient of $\free_{A^*}$ by the ideal generated by $I(3)^\perp$. For any $f,g\in A^*$, we put
\[X_1(f,g)=\bftroisun{$1$}{$2$}{$3$}\otimes f\otimes g-\bftroisdeux{$1$}{$2$}{$3$}\otimes\Phi^*(f\otimes g).\]
\end{itemize}
For any $f,g\in A^*$, any $a,b\in A$,
\begin{align*}
\langle X_1(f\otimes g),X(a\otimes b)\rangle&=-(f\otimes g)(\Phi(a\otimes b)+\Phi^*(f\otimes g)(a\otimes b)=0,\\
\langle X_1(f\otimes g),Y(a\otimes b)\rangle&=0,\\
\langle X_1(f\otimes g),Z(a\otimes b)\rangle&=0.
\end{align*}
So $X_1(f\otimes g)\in I(3)^\perp$. As a consequence, we obtain a free $\mathfrak{S}_3$-submodule of $I(3)^\perp$, generated by the elements
\begin{align*}
X_1(f\otimes g),&\\
X_2(f\otimes g)&=X_1(f\otimes g)^{(12)}=\bftroisun{$2$}{$1$}{$3$}\otimes f\otimes g-\bftroisdeux{$2$}{$1$}{$3$}\otimes\Phi^*(f\otimes g),\\
X_3(f\otimes g)&=X_1(f\otimes g)^{(23)}=\bftroisun{$1$}{$3$}{$2$}\otimes f\otimes g-\bftroisdeux{$1$}{$3$}{$2$}\otimes\Phi^*(f\otimes g),\\
X_4(f\otimes g)&=X_1(f\otimes g)^{(13)}=\bftroisun{$3$}{$2$}{$1$}\otimes f\otimes g-\bftroisdeux{$3$}{$2$}{$1$}\otimes\Phi^*(f\otimes g),\\
X_5(f\otimes g)&=X_1(f\otimes g)^{(123)}=\bftroisun{$2$}{$3$}{$1$}\otimes f\otimes g-\bftroisdeux{$2$}{$3$}{$1$}\otimes\Phi^*(f\otimes g),\\
X_6(f\otimes g)&=X_1(f\otimes g)^{(132)}=\bftroisun{$3$}{$1$}{$2$}\otimes f\otimes g-\bftroisdeux{$3$}{$1$}{$2$}\otimes\Phi^*(f\otimes g),
\end{align*}
with $f,g\in A^*$. For any $f,g\in A^*$, we put
\[X_7(f\otimes g)=\bftroisdeux{$1$}{$2$}{$3$}\otimes f\otimes g-\bftroisdeux{$2$}{$1$}{$3$}\otimes g\otimes f.\]
For any $f,g\in A^*$, any $a,b\in A$,
\begin{align*}
\langle X_7(f\otimes g),X(a\otimes b)\rangle&=-f(a)g(b)+g(b)f(a)=0,\\
\langle X_7(f\otimes g),Y(a\otimes b)\rangle&=0,\\
\langle X_7(f\otimes g),Z(a\otimes b)\rangle&=0.
\end{align*}
So $X_7(f\otimes g)\in I(3)^\perp$. As a consequence, noticing that $X_7(f\otimes g)^{(12)}=-X_7(g\otimes f)$, we obtain a $\mathfrak{S}_3$-submodule of $I(3)^\perp$, generated by the elements
\begin{align*}
X_7(f\otimes g),&\\
X_8(f\otimes g)&=X_7(f\otimes f)^{(23)}=\bftroisdeux{$1$}{$3$}{$2$}\otimes f\otimes g-\bftroisdeux{$3$}{$1$}{$2$}\otimes g\otimes f,\\
X_9(f\otimes g)&=X_7(f\otimes f)^{(13)}=\bftroisdeux{$3$}{$2$}{$1$}\otimes f\otimes g-\bftroisdeux{$2$}{$3$}{$1$}\otimes g\otimes f.
\end{align*}
Let $f_i\otimes g_i\in A^*\otimes A^*$ such that
\[\sum_{i=1}^9 X_i(f_i\otimes g_i)=0.\]
Considering the terms $\bftroisdeux{$i$}{$j$}{$k$}\otimes f\otimes g$, with $(i,j,k)\in\mathfrak{S}_3$, we obtain that if $1\leqslant i\leqslant 6$, $f_i\otimes g_i=0$. 
Considering the terms $\bftroisun{$i$}{$j$}{$k$}$, we obtain that if $7\leqslant i\leqslant 9$, $f_i\otimes g_i=0$. Hence, we have obtained a submodule of $I(3)^\perp$ of dimension $9\dim(V)^2$. Moreover,
\[\dim(I(3)^\perp)=\dim(\free_A(3))-\dim(I(3))=12\dim(V)^2-3\dim(V)^2=9\dim(V)^2,\]
so the elements $X_1(f\otimes f)$ and $X_7(f\otimes g)$, with $f,g\in A^*$, generate $I(3)^\perp$, which gives the description of $\prelie_\Phi^!$.\end{proof}

\subsection{Free generalized permutative algebras}

\begin{prop}\label{prop3.3}
Let $(A,\Phi)$ be a dual $\ell$CEDS and $V$ be a vector space. We define
\[C_A(V)=S(A\otimes V)\otimes V.\]
Its elements will be denoted as linear spans of terms $a_1x_1\ldots a_nx_n\mid x$, with $n\geqslant 0$, $a_1,\ldots,a_n\in A$ and $x_1,\ldots,x_n,x\in V$. Note that for any $\sigma\in\mathfrak{S}_n$,
\[a_1x_1\ldots a_nx_n\mid x=a_{\sigma(1)}x_{\sigma(1)}\ldots a_{\sigma(n)}x_{\sigma(n)}\mid x.\]
For any $a\in A$, we define a product a product $\diamond_a$ on $C_A(V)$ by
\begin{align}
\label{EQ22} a_1x_1\ldots a_kx_k\mid x\diamond_a b_1y_1\ldots b_ly_k\mid y&=\sum a'_1x_1\ldots a'_kx_ka'_{k+1}x y_1\ldots b_ly_k\mid y,
\end{align}
with
\[\sum a'_1\otimes\ldots\otimes a'_{k+1}=\left(\id^{\otimes (k-1)}\otimes\Phi\right)\circ\ldots\circ\left(\Phi\otimes\id^{\otimes (k-1)}\right)(a\otimes a_1\otimes\ldots\otimes a_k).\]
Then $(C_A(V),\diamond)$ is the free $(A,\Phi)$-permutative algebra generated by $V$. 
\end{prop}

\begin{proof}
Firstly, note that $\diamond_a$ is well-defined, that is to say the result in (\ref{EQ22}) does not depend on the order chosen on the $a_ix_i$. For example, for $k=2$, taking $b_1\otimes b_2=a_2\otimes a_1$, by (\ref{EQ16}),
\begin{align*}
\sum b'_1\otimes b'_2\otimes b'_3&=(\id\otimes\Phi)\circ (\Phi\otimes\id)(a\otimes a_2\otimes a_1)\\
&=\sum\sum a'\curvearrowright a_2'\otimes (a''\blacktriangleright a''_2)'\curvearrowright a'_1\otimes (a''\blacktriangleright a''_2)''\blacktriangleright a''_1\\
&=\sum\sum(a''\blacktriangleright a''_1)'\curvearrowright a'_2\otimes a'\curvearrowright a'_1\otimes (a''\blacktriangleright a''_1)''\blacktriangleright a''_2,
\end{align*}
so, in $S(A\otimes V)$,
\[\sum a'_1x_1a'_2x_2a'_3x=\sum b'_1x_2b'_2x_1b'_3x.\]
Let us prove that $C_A(V)$ is a $(A,\Phi)$-permutative algebra. Let us consider
\begin{align*}
X&=a_1 x_1\ldots a_k x_k\mid x,&Y&=b_1 y_1\ldots b_k y_k\mid y,&Z&=c_1 z_1\ldots c_m z_m\mid z.
\end{align*}
Let $a,b\in A$. Then
\begin{align*}
(X\diamond_b Y)\diamond_a Z&=\sum a'_1 x_1\ldots a'_k x_k a'_{k+1}x b'_1 y_1\ldots b'_k l_l b'_{l+1} y c_1 z_1\ldots c_m z_m\mid z,\\
\sum X\diamond_{a'\curvearrowright b'} (Y\diamond_{a''\blacktriangleright b''} Z)&=\sum a''_1 x_1\ldots a''_k x_k a''_{k+1}x b''_1 y_1\ldots b''_k l_l b''_{l+1} y c_1 z_1\ldots c_m z_m\mid z,
\end{align*}
with
\begin{align*}
\sum a'_1\otimes\ldots a'_{k+1}\otimes b'_1\otimes\ldots\otimes b'_{l+1}&=\left(\id^{\otimes (k+l)}\otimes\Phi\right)\circ\ldots\circ\left(\Phi\otimes\id^{\otimes (k+l)}\right)\\
&\circ\left (\id^{\otimes k}\otimes\Phi\otimes\id^{\otimes l}\right)\circ\ldots\circ\left(\id\otimes\Phi\otimes\id^{\otimes (k+l-1)}\right)\\
&(a\otimes b\otimes a_1\otimes\ldots\otimes a_k\otimes b_1\otimes\ldots\otimes b_l),\\
\sum a''_1\otimes\ldots a''_{k+1}\otimes b''_1\otimes\ldots\otimes b''_{l+1}&=\left(\id^{\otimes(k-1)}\otimes\Phi\otimes\id^{\otimes (l+1)}\right)\circ\ldots\circ\left(\Phi\otimes\id^{\otimes k+l}\right)\\
&\circ\left(\id^{\otimes(k+l)}\otimes\Phi\right)\circ\ldots\circ\left(\id^{\otimes (k+1)}\otimes\Phi\otimes\id^{\otimes (l-1)}\right)\\
&\circ\left(\id^{\otimes(k-1)}\otimes\tau\otimes\id^{\otimes (l+1)}\right)\circ\ldots\circ\left(\id\otimes\tau\otimes\id^{\otimes (k+l-1)}\right)\\
&\circ\left(\Phi\otimes\id^{\otimes k+l}\right)(a\otimes b\otimes a_1\otimes\ldots\otimes a_k\otimes b_1\otimes\ldots\otimes b_l),
\end{align*}
where $\tau:A\otimes A\longrightarrow A\otimes A$ is the usual flip. In order to prove that these two elements of $A^{\otimes (k+l+2)}$ are equal, we shall use braid diagrams-like notations. 
We shall represent $\Phi$ and $\tau$ by diagrams (to be read from bottom to top)
\begin{align*}
\Phi&=\begin{tikzpicture}[line cap=round,line join=round,>=triangle 45,x=0.15cm,y=0.15cm]
\clip(0.,0.) rectangle (3.,3.);
\draw [line width=1.2pt] (0.,0.)-- (1.,1.);
\draw [line width=1.2pt] (3.,0.)-- (0.,3.);
\draw [line width=1.2pt] (2.,2.)-- (3.,3.);
\end{tikzpicture},&\tau&=\begin{tikzpicture}[line cap=round,line join=round,>=triangle 45,x=0.15cm,y=0.15cm]
\clip(0.,0.) rectangle (3.,3.);
\draw [line width=1.2pt] (0.,0.)-- (3.,3.);
\draw [line width=1.2pt] (3.,0.)-- (0.,3.);
\end{tikzpicture},
\end{align*}
and the composition will be represented by a vertical concatenation. For example, (\ref{EQ14}) is rewritten as
\begin{align}
\tag{\ref{EQ14}''}\substack{\hspace{2mm}\\ \begin{tikzpicture}[line cap=round,line join=round,>=triangle 45,x=0.15cm,y=0.15cm]
\clip(-0.5,-0.1) rectangle (6.5,9.1);
\draw [line width=1.2pt] (-0.5,0.)-- (6.5,0.);
\draw [line width=1.2pt] (-0.5,9.)-- (6.5,9.);
\draw [line width=1.2pt] (0.,0.)-- (1.,1.);
\draw [line width=1.2pt] (3.,0.)-- (0.,3.);
\draw [line width=1.2pt] (2.,2.)-- (3.,3.);
\draw [line width=1.2pt] (3.,3.)-- (6.,6.);
\draw [line width=1.2pt] (6.,0.)-- (6.,3.);
\draw [line width=1.2pt] (6.,3.)-- (3.,6.);
\draw [line width=1.2pt] (0.,3.)-- (0.,6.);
\draw [line width=1.2pt] (0.,6.)-- (1.,7.);
\draw [line width=1.2pt] (3.,6.)-- (0.,9.);
\draw [line width=1.2pt] (2.,8.)-- (3.,9.);
\draw [line width=1.2pt] (6.,6.)-- (6.,9.);
\end{tikzpicture}}&=
\substack{\hspace{2mm}\\ \begin{tikzpicture}[line cap=round,line join=round,>=triangle 45,x=0.15cm,y=0.15cm]
\clip(8.5,-0.1) rectangle (15.5,9.1);
\draw [line width=1.2pt] (8.5,0.)-- (15.5,0.);
\draw [line width=1.2pt] (8.5,9.)-- (15.5,9.);
\draw [line width=1.2pt] (9.,0.)-- (9.,3.);
\draw [line width=1.2pt] (12.,0.)-- (13.,1.);
\draw [line width=1.2pt] (15.,0.)-- (9.,6.);
\draw [line width=1.2pt] (14.,2.)-- (15.,3.);
\draw [line width=1.2pt] (9.,3.)-- (10.,4.);
\draw [line width=1.2pt] (11.,5.)-- (12.,6.);
\draw [line width=1.2pt] (15.,3.)-- (15.,6.);
\draw [line width=1.2pt] (15.,6.)-- (12.,9.);
\draw [line width=1.2pt] (9.,6.)-- (9.,9.);
\draw [line width=1.2pt] (15.,9.)-- (14.,8.);
\draw [line width=1.2pt] (12.,6.)-- (13.,7.);
\end{tikzpicture}}.\end{align}
Then
\begin{align*}
\sum a'_1\otimes\ldots a'_{k+1}\otimes b'_1\otimes\ldots\otimes b'_{l+1}&=\substack{\hspace{2mm}\\ \begin{tikzpicture}[line cap=round,line join=round,>=triangle 45,x=0.15cm,y=0.15cm]
\clip(-0.5,-0.1) rectangle (15.5,12.1);
\draw [line width=1.2pt] (-0.5,0.)-- (15.5,0.);
\draw [line width=1.2pt] (-0.5,12.)-- (15.5,12.);
\draw [line width=1.2pt] (0.,0.)-- (0.,6.);
\draw [line width=1.2pt] (3.,6.)-- (6.,0.);
\draw [line width=1.2pt] (9.,0.)-- (6.,6.);
\draw [line width=1.2pt] (3.,0.)-- (4.,1.);
\draw [line width=1.2pt] (9.,6.)-- (8.,5.);
\draw [line width=1.2pt] (12.,0.)-- (12.,6.);
\draw [line width=1.2pt] (15.,0.)-- (15.,6.);
\draw [line width=1.2pt] (15.,6.)-- (12.,12.);
\draw [line width=1.2pt] (12.,6.)-- (9.,12.);
\draw [line width=1.2pt] (9.,6.)-- (6.,12.);
\draw [line width=1.2pt] (6.,6.)-- (3.,12.);
\draw [line width=1.2pt] (3.,6.)-- (0.,12.);
\draw [line width=1.2pt] (0.,6.)-- (1.6,6.6);
\draw [line width=1.2pt] (5.8,8.3)-- (6.6,8.6);
\draw [line width=1.2pt] (8.4,9.4)-- (9.3,9.7);
\draw [line width=1.2pt] (13.5,11.5)-- (15.,12.);
\draw (3.2,6.) node[anchor=north west] {\tiny $\ldots$};
\draw (11.7,6.) node[anchor=north west] {\tiny $\ldots$};
\end{tikzpicture}}\:
(a\otimes b\otimes a_1\otimes\ldots\otimes a_k\otimes b_1\otimes\ldots\otimes b_l),\\[2mm]
\sum a'_1\otimes\ldots a'_{k+1}\otimes b'_1\otimes\ldots\otimes b'_{l+1}&=\substack{\hspace{2mm}\\ \begin{tikzpicture}[line cap=round,line join=round,>=triangle 45,x=0.15cm,y=0.15cm]
\clip(-0.5,-0.1) rectangle (15.5,15.1);
\draw [line width=1.2pt] (-0.5,0.)-- (15.5,0.);
\draw [line width=1.2pt] (-0.5,15.)-- (15.5,15.);
\draw [line width=1.2pt] (3.,0.)-- (0.,3.);
\draw [line width=1.2pt] (0.,0.)-- (1.,1.);
\draw [line width=1.2pt] (2.,2.)-- (3.,3.);
\draw [line width=1.2pt] (6.,0.)-- (6.,3.);
\draw [line width=1.2pt] (9.,0.)-- (9.,3.);
\draw [line width=1.2pt] (6.,3.)-- (3.,9.);
\draw [line width=1.2pt] (9.,3.)-- (6.,9.);
\draw [line width=1.2pt] (3.,3.)-- (9.,9.);
\draw [line width=1.2pt] (0.,3.)-- (0.,9.);
\draw [line width=1.2pt] (12.,0.)-- (12.,9.);
\draw [line width=1.2pt] (15.,0.)-- (15.,9.);
\draw [line width=1.2pt] (12.,9.)-- (9.,15.);
\draw [line width=1.2pt] (15.,9.)-- (12.,15.);
\draw [line width=1.2pt] (9.,9.)-- (10.,10.);
\draw [line width=1.2pt] (15.,15.)-- (14.,14.);
\draw [line width=1.2pt] (3.,9.)-- (0.,15.);
\draw [line width=1.2pt] (6.,9.)-- (3.,15.);
\draw [line width=1.2pt] (0.,9.)-- (1.,10.);
\draw [line width=1.2pt] (6.,15.)-- (5.,14.);
\draw (3.5,8.4) node[anchor=north west] {\tiny $\ldots$};
\draw (11.7,8.4) node[anchor=north west] {\tiny $\ldots$};
\end{tikzpicture}}\:(a\otimes b\otimes a_1\otimes\ldots\otimes a_k\otimes b_1\otimes\ldots\otimes b_l).
\end{align*}
The iterated application of (\ref{EQ14}) gives that the two diagrams above are equal, which finally gives (\ref{EQ20}) for $C_A(V)$. 
Moreover,
\begin{align*}
X\diamond_a(Y\diamond_b Z)&=\sum a'_1 x_1\ldots a'_k x_k a'_{k+1}b'_1 y_1\ldots b'_l y_l b'_{l+1}y c_1 z_1\ldots c_m z_m\mid z,\\
Y\diamond_b(X\diamond_a Z)&=\sum a''_1 x_1\ldots a''_k x_k a''_{k+1}b''_1 y_1\ldots b''_l y_l b''_{l+1}y c_1 z_1\ldots c_m z_m\mid z,
\end{align*}
with
\begin{align*}
\sum a'_1\otimes\ldots a'_{k+1}\otimes b'_1\otimes\ldots\otimes b'_{l+1}&=\substack{\hspace{2mm}\\ \begin{tikzpicture}[line cap=round,line join=round,>=triangle 45,x=0.15cm,y=0.15cm]
\clip(-0.5,-0.1) rectangle (15.5,18.1);
\draw [line width=1.2pt] (-0.5,0.)-- (15.5,0.);
\draw [line width=1.2pt] (-0.5,18.)-- (15.5,18.);
\draw [line width=1.2pt] (0.,0.)-- (0.,6.);
\draw [line width=1.2pt] (6.,0.)-- (3.,6.);
\draw [line width=1.2pt] (9.,0.)-- (6.,6.);
\draw [line width=1.2pt] (3.,0.)-- (9.,6.);
\draw [line width=1.2pt] (12.,0.)-- (12.,6.);
\draw [line width=1.2pt] (15.,0.)-- (15.,6.);
\draw [line width=1.2pt] (12.,6.)-- (9.,12.);
\draw [line width=1.2pt] (15.,6.)-- (12.,12.);
\draw [line width=1.2pt] (9.,6.)-- (10.,7.);
\draw [line width=1.2pt] (15.,12.)-- (14.,11.);
\draw [line width=1.2pt] (6.,6.)-- (6.,12.);
\draw [line width=1.2pt] (3.,6.)-- (3.,12.);
\draw [line width=1.2pt] (0.,6.)-- (0.,12.);
\draw [line width=1.2pt] (3.,12.)-- (0.,18.);
\draw [line width=1.2pt] (6.,12.)-- (3.,18.);
\draw [line width=1.2pt] (9.,12.)-- (9.,18.);
\draw [line width=1.2pt] (12.,12.)-- (12.,18.);
\draw [line width=1.2pt] (15.,12.)-- (15.,18.);
\draw [line width=1.2pt] (0.,12.)-- (1.,13.);
\draw [line width=1.2pt] (6.,18.)-- (5.,17.);
\draw (2.7,9.) node[anchor=north west] {\tiny $\ldots$};
\draw (10.7,9.) node[anchor=north west] {\tiny $\ldots$};
\end{tikzpicture}}\:(a\otimes b\otimes a_1\otimes\ldots\otimes a_k\otimes b_1\otimes\ldots\otimes b_l),\\[2mm]
\sum a'_1\otimes\ldots a'_{k+1}\otimes b'_1\otimes\ldots\otimes b'_{l+1}&=\substack{\hspace{2mm}\\ \begin{tikzpicture}[line cap=round,line join=round,>=triangle 45,x=0.15cm,y=0.15cm]
\clip(-0.5,-0.1) rectangle (15.5,18.1);
\draw [line width=1.2pt] (-0.5,0.)-- (15.5,0.);
\draw [line width=1.2pt] (-0.5,18.)-- (15.5,18.);
\draw [line width=1.2pt] (0.,0.)-- (0.,6.);
\draw [line width=1.2pt] (6.,0.)-- (3.,6.);
\draw [line width=1.2pt] (9.,0.)-- (6.,6.);
\draw [line width=1.2pt] (3.,0.)-- (9.,6.);
\draw [line width=1.2pt] (12.,0.)-- (12.,6.);
\draw [line width=1.2pt] (15.,0.)-- (15.,6.);
\draw [line width=1.2pt] (9.,6.)-- (9.,12.);
\draw [line width=1.2pt] (3.,6.)-- (0.,12.);
\draw [line width=1.2pt] (6.,6.)-- (3.,12.);
\draw [line width=1.2pt] (9.,12.)-- (10.,13.);
\draw [line width=1.2pt] (12.,6.)-- (12.,12.);
\draw [line width=1.2pt] (15.,6.)-- (15.,12.);
\draw [line width=1.2pt] (0.,6.)-- (1.,7.);
\draw [line width=1.2pt] (6.,12.)-- (5.02,11.14);
\draw [line width=1.2pt] (0.,12.)-- (0.,18.);
\draw [line width=1.2pt] (3.,12.)-- (3.,18.);
\draw [line width=1.2pt] (6.,12.)-- (6.,18.);
\draw [line width=1.2pt] (12.,12.)-- (9.,18.);
\draw [line width=1.2pt] (15.,12.)-- (12.,18.);
\draw [line width=1.2pt] (15.,18.)-- (14.,17.);
\draw (1.7,9.) node[anchor=north west] {\tiny $\ldots$};
\draw (11.7,9.) node[anchor=north west] {\tiny $\ldots$};
\end{tikzpicture}}\:(a\otimes b\otimes a_1\otimes\ldots\otimes a_k\otimes b_1\otimes\ldots\otimes b_l).
\end{align*}
The two appearing diagrams are equal to
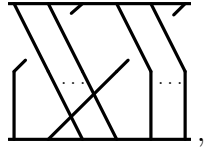
\[\begin{tikzpicture}[line cap=round,line join=round,>=triangle 45,x=0.15cm,y=0.15cm]
\clip(-0.5,-0.1) rectangle (15.5,12.1);
\draw [line width=1.2pt] (-0.5,0.)-- (15.5,0.);
\draw [line width=1.2pt] (-0.5,12.)-- (15.5,12.);
\draw [line width=1.2pt] (0.,0.)-- (0.,6.);
\draw [line width=1.2pt] (6.,0.)-- (3.,6.);
\draw [line width=1.2pt] (9.,0.)-- (6.,6.);
\draw [line width=1.2pt] (3.,0.)-- (9.,6.);
\draw [line width=1.2pt] (12.,0.)-- (12.,6.);
\draw [line width=1.2pt] (15.,0.)-- (15.,6.);
\draw [line width=1.2pt] (3.,6.)-- (0.,12.);
\draw [line width=1.2pt] (6.,6.)-- (3.,12.);
\draw [line width=1.2pt] (9.,6.)-- (10.,7.);
\draw [line width=1.2pt] (0.,6.)-- (1.,7.);
\draw [line width=1.2pt] (6.,12.)-- (5.02,11.14);
\draw [line width=1.2pt] (12.,6.)-- (9.,12.);
\draw [line width=1.2pt] (15.,6.)-- (12.,12.);
\draw [line width=1.2pt] (15.,12.)-- (14.,11.);
\draw (3.2,6.) node[anchor=north west] {\tiny $\ldots$};
\draw (11.7,6.) node[anchor=north west] {\tiny $\ldots$};
\end{tikzpicture}\:,\]
which implies (\ref{EQ21}) for $C_A(V)$.\\

Let $B$ be an $(A,\Phi)$-permutative algebra and let $\theta:V\longrightarrow B$ be a linear span. Let us prove that there exists a unique permutative algebra morphism $\Theta:C_A(V)\longrightarrow B$, extending $\theta$.\\

\textit{Uniqueness.} Let us remark that for any $x_1,x_2,\ldots,x_k,y\in V$, $a_1,\ldots,a_k\in A$, by definition of the product of $C_A(V)$,
\[x_1\diamond_{a_1} a_2 x_2\ldots a_k x_k\mid x=a_1 x_1\ldots a_k x_k\mid x.\]
A direct consequence is that $C_A(V)$ is generated by $V$, which implies that such a $\Theta$ is unique.\\

\textit{Existence.} Let us define $\Theta$ by
\[\Theta(a_1 x_1\ldots a_n kx_l\mid x)=\theta(x_1)\diamond_{a_1}(\ldots (\theta(x_k)\diamond_{a_k}\theta(x))\ldots).\]
By (\ref{EQ21}), this does not change if one permutes $a_1 x_1,\ldots, a_k x_k$, so this is well-defined. Let us consider
\[C_A'(V)=\{x\in C_A(V),\:\forall y\in C_A(V),\: f\forall a\in A,\:\Theta(x\diamond_a y)=\Theta(x)\diamond_a\Theta(y)\}.\]
Let $x_1$, $x_2\in C_A'(V)$ and $b\in A$. For any $y\in C_A(V)$ and $a\in A$,
\begin{align*}
\Theta(x_1\diamond_b x_2)\diamond_a\Theta(x_1)&=(\Theta(x_1)\diamond_b\Theta(x_2))\diamond_a\Theta(x_1)\\
&=\sum\Theta(x_1)\diamond_{a'\curvearrowright b'} (\Theta(x_2)\diamond_{a''\blacktriangleright b''}\Theta(x_1))\\
&=\sum\Theta(x_1)\diamond_{a'\curvearrowright b'}\Theta(x_2\diamond_{a''\blacktriangleright b''}x_1)\\
&=\Theta\left(\sum x_1\diamond_{a'\curvearrowright b'}( x_2\diamond_{a''\blacktriangleright b''}x_1)\right)\\
&=\Theta\left( x_1\diamond_b x_2)\diamond_a x_1\right).
\end{align*}
Therefore, 
$x_1\diamond_b x_2\in C'_A(V)$: we proved that $C'_A(V)$ is an $(A,\Phi)$-permutative subalgebra of $C_A(V)$. Moreover, for any $x\in V$, for any $y_1,\ldots,y_k,y\in V$, for any $a,b_1,\ldots,b_k\in A$,
\begin{align*}
\Theta(x\diamond_a b_1 y_1\ldots b_k y_k\mid y)&=\Theta(ax b_1 y_1\ldots b_k y_k\mid y)\\
&=\theta(x)\diamond_a (y_1\diamond_{b_1}(\ldots (y_k\diamond_{b_k} y)\ldots))\\
&=\Theta(x)\diamond_a\Theta(b_1 y_1\ldots b_k y_k\mid y),
\end{align*}
so $x\in C'_A(V)$. As $C_A(V)$ is generated by $V$, $C_A(V)=C'_A(V)$, which implies that $\Theta$ is an $(A,\Phi)$-permutative algebra morphism. Therefore, $C_A(V)$ is the free $(A,\Phi)$-permutative algebra generated by $V$.
\end{proof}

\begin{cor}
Let $(A,\Phi)$ be a finite-dimensional dual $\ell$CEDS. For any $n\geqslant 1$,
\[\dim(\perm_\Phi(n))=n\dim(A)^{n-1}.\]
\end{cor}

\subsection{From permutative to generalized permutative algebras}

\begin{notation}
1. We denote by $\perm$ the operad of permutative algebras, that is to say binary algebras $(A,\diamond)$ such that for any $a,b,c\in A$,
\begin{align*}
(x\diamond y)\diamond z&=x\diamond (y\diamond z),&x\diamond(y\diamond z)&=y\diamond(x\diamond z).
\end{align*}
The operad $\perm$ is the Koszul dual of the operad $\prelie$.\\

2. We denote by $\perm^{(0)}$ the operad of $0$-permutative algebras, that is to say binary algebras $(A,\diamond)$ such that for any $a,b,c\in A$,
\begin{align*}
(x\diamond y)\diamond z&=0,&x\diamond(y\diamond z)&=y\diamond(x\diamond z).
\end{align*}
The operad $\perm^{(0)}$ is the Koszul dual of the operad of NAP algebras.
\end{notation}

We now consider operad morphisms from $\perm$ or $\perm^{(0)}$ to $\perm_\Phi$, or equivalently permutative or $0$-permutative products in $\perm_\Phi$. We shall need the following notion:

\begin{defi}
Let $(A,\Phi)$ be an $\ell$EAS. A weak special vector of $(A,\Phi)$ is an element $a\in A$ such that
\[\Phi(a\otimes a)+\tau\circ\Phi(a\otimes a)=a\otimes a.\]
\end{defi}

\begin{prop}\label{prop3.6}
Let $(A,\Phi)$ be a dual $\ell$CEDS. The permutative products in $\perm_\Phi$ are of one of the following form:
\begin{itemize}
\item $\diamond_a$, where $a$ is a special vector of $\Phi$ of eigenvalue $0$.
\item $\diamond_a+\diamond_a^{op}$, where $a$ is a weak special vector.
\end{itemize}
The $0$-permutative products in $\perm_\Phi$ are the products $\diamond_a$, where $a$ is a special vector of eigenvalue $0$.
\end{prop}

\begin{proof}
Let $\diamond=\diamond_a+\diamond_b^{op}$ be any element of $\perm_\Phi(2)$, with $a,b\in A$. Let us apply it in the free $\perm_\Phi$-algebra on three generators $x,y,z$. 
\begin{align*}
(x\diamond y)\diamond z&=(a'\curvearrowright a' x a''\blacktriangleright a'' y+a'\curvearrowright b'y a''\blacktriangleright b''x)\mid z+axbz\mid y+bybz\mid x,\\
x\diamond (y\diamond z)&=axay\mid z+axbz\mid y+(b'\curvearrowright a' y b''\blacktriangleright a'' zb'\curvearrowright b' z b''\blacktriangleright b'' y\mid x,\\
y\diamond (x\diamond z)&=axay\mid z+aybz\mid x+(b'\curvearrowright a' x b''\blacktriangleright a'' zb'\curvearrowright b' z b''\blacktriangleright b'' x\mid y.
\end{align*}
Identifying terms $cxdy\mid z$, $cxdz\mid y$ and $cycz\mid x$, if $\diamond$ is a permutative product,
\begin{align}
\label{EQ23}\Phi(a\otimes a)+\tau\circ\Phi(a\otimes b)&=a\otimes a,\\
\label{EQ24}\Phi(b\otimes a)+\tau\circ\Phi(b\otimes b)&=b\otimes b,\\
\label{EQ25}\Phi(b\otimes a)+\tau\circ\Phi(b\otimes b)&=a\otimes b.
\end{align}
By (\ref{EQ25})-(\ref{EQ24}), $(a-b)\otimes b=0$, so $a=b$ or $b=0$. If $b=0$, then by (\ref{EQ23}), $a$ is a special vector of eigenvalue $1$. If $a=b$, by (\ref{EQ23}), $a$ is a weak special vector. The converse implication is immediate.\\

If $\diamond$ is a $0$-permutative product,
\begin{align}
\label{EQ26}\Phi(a\otimes a)+\tau\circ\Phi(a\otimes b)&=0,\\
\label{EQ27}a\otimes b&=0,\\
\label{EQ28}b\otimes b&=0,\\
\label{EQ29}\Phi(b\otimes a)+\tau\circ\Phi(b\otimes b)&=a\otimes b.
\end{align}
By (\ref{EQ28}), $b=0$ and, by (\ref{EQ26}), $a$ is a special vector of eigenvalue $0$. The converse implication is immediate.\end{proof}

We do not know any $\ell$EAS with a nonzero weak special vector, and we only have negative results on their existence:

\begin{prop}\label{prop3.7}\begin{enumerate}
\item Let $(A,\Phi)$ be a two-dimensional $\ell$EAS. Then its only weak special vector is $0$.
\item Let $\Omega$ be a semigroup. If $\Omega$ is a group or if $\Omega$ is finite, then the only weak special vector of $\leas(\K\Omega)$ is $0$.
\item Let $(\Omega,\star)$ be a group and let $(A,\Phi)$ be the linearization of $\eas'(\Omega,\star)$. Then the only weak special vector of $(A,\Phi)$ is $0$.
\end{enumerate}\end{prop}

\begin{proof}
1. Let us assume that $(A,\Phi)$ has a nonzero weak special vector $x$, which we complete in a basis $(x,y)$ of $A$.nThe basis of $\Phi$ in the basis $(x\otimes x,x\otimes y,y\otimes x,y\otimes y)$ is denoted by $M=(a_{ij})_{1\leqslant i,j\leqslant 4}$. Then
\begin{align*}
\Phi(x\otimes x)+\tau\circ\Phi(x\otimes x)&=2a_{11} x\otimes x+(a_{21}+a_{31})(x\otimes y+y\otimes x)+2a_{41}y\otimes y=x\otimes x,
\end{align*}
so $a_{11}=\dfrac{1}{2}$, $a_{31}=-a_{21}$ and $a_{41}=0$. Considering the matrix of the map 
$(\id\otimes\Phi)\circ (\Phi\otimes\id)\circ (\id\otimes\Phi)-(\Phi\otimes\id)\circ (\id\otimes\tau)\circ (\Phi\otimes\id)$
in the basis of tensors products of $x$ and $y$, we obtain a $8\times 8$ matrix, which all coefficients are zero. Up to the replacement of $y$ by $a_{42}y$ if $a_{42}$ is nonzero, we can assume that $a_{42}=0$ or $1$.
\begin{itemize}
\item If $a_{42}=0$, coefficient (6,1) is $\dfrac{a_{43}}{2}=0$, whereas coefficient (7,5) is $-\dfrac{a_{43}^2}{2}=0$: this is a contradiction.
\item If $a_{42}=1$, coefficient (7,1) is $-a_{43}-1=0$, so $a_{43}=-1$. Coefficient (6,1) is $-a_{33}-\dfrac{1}{2}=0$,
whereas coefficient (7,5) is $a_{33}-\dfrac{1}{2}=0$: this is a contradiction. 
\end{itemize}
So there is no nonzero weak special vector in $A$.\\

2. Let $\displaystyle a=\sum_{\alpha\in\Omega} a_\alpha a\in\K\Omega$.
\begin{align}
\nonumber&\mbox{$a$ is a weak special vector of }\leas(\K\Omega)\\
\nonumber&\Longleftrightarrow\sum_{\alpha,\beta\in\Omega} a_\alpha a_\beta\alpha\otimes\beta=\sum_{\alpha,\beta\in\Omega} a_\alpha a_\beta(\alpha\beta\otimes\alpha+\alpha\otimes\alpha\beta)\\
\label{EQ30} &\Longleftrightarrow\forall\alpha,\beta\in\Omega,\: a_\alpha a_\beta=\left(\sum_{\gamma\in\Omega,\:\beta\gamma=\alpha} a_\gamma\right)a_{\beta}+\left(\sum_{\gamma\in\Omega,\:\alpha\gamma=\beta} a_\gamma\right)a_{\alpha}.
\end{align}
Let $a$ be a nonzero weak special vector of $\leas(\K\Omega)$. We put $G=\{\alpha\in\Omega,\:a_\alpha\neq 0\}$. Then $G$ is nonempty.\\

We first assume that $\Omega$ is a group. Then (\ref{EQ30}) becomes
\begin{align*}
&\forall\alpha,\beta\in\Omega,&a_\alpha a_\beta&=a_\beta a_{\beta^{-1}\alpha}+a_{\alpha}a_{\alpha^{-1}\beta}.
\end{align*}
Taking $\alpha=\beta=e_G$, we obtain $a_{e_g}^2=2a_{e_G}^2$, so $a_{e_G}=0$.Taking $\alpha=\beta\in\Omega$, we obtain $a_{\alpha}^2=2a_{e_G}a_\alpha=0$, so $a_\alpha=0$. There is no nonzero weak special vector in this case.\\

We now assume that $\Omega$ is finite. We define a relation $\preceq$ on $G$ by
\[\alpha\preceq\beta\Longleftrightarrow\alpha\beta=\alpha\mbox{ or }\alpha=\beta.\]
This relation is obviously reflexive. Let us assume that $\alpha\preceq\beta$ and $\beta\preceq\gamma$. If $\alpha=\beta$ or $\beta=\gamma$, then obviously $\alpha\preceq\gamma$. Otherwise,
\[\alpha\gamma=(\alpha\beta)\gamma=\alpha(\beta\gamma)=\alpha\beta=\alpha.\]
So $\alpha\preceq\gamma$. We proved that $\preceq$ is transitive, so $\preceq$ is a quasi-order. Consequently:
\begin{itemize}
\item The relation $\sim$ defined on $G$ by $\alpha\sim\beta$ if, $\alpha\preceq\beta$ and $\beta\preceq\alpha$ is an equivalence.
\item The relation $\overline{\preceq}$ defined on $G/\sim$ by $\overline{\alpha}\overline{\preceq}\overline{\beta}$ if $\alpha\preceq\beta$ is an order.
\end{itemize}
The poset $(G/\sim,\overline{\preceq})$ is finite, as $\Omega$ is finite; so it has a maximal element, which is a class of $\sim$ denoted by $H$. By construction of $H$:
\begin{itemize}
\item $H$ is nonempty and, for any $\alpha\in H$, $a_\alpha\neq 0$.
\item If $\alpha\in H$ and $\beta\in G$ such that $\alpha\beta=\alpha$, then $\beta\in H$.
\item For any $\alpha,\beta\in H$, $\alpha=\beta$ or $\alpha\beta=\alpha$. 
\end{itemize}
We put 
\begin{align*}
H_1&=\{\alpha\in H,\alpha\alpha=\alpha\},&H_{-1}&=\{\alpha\in H,\alpha\alpha\neq\alpha\}.
\end{align*}
Then (\ref{EQ30}) gives
\begin{align*}
&\forall\alpha\in H_1,&a_\alpha&=2\sum_{\gamma\in H}a_\gamma,\\
&\forall\alpha\in H_{-1},&a_\alpha&=2\sum_{\gamma\in H,\:\gamma\neq\alpha }a_\gamma,
\end{align*}
which can be reformulate as 
\begin{align*}
&\forall\alpha\in H_1,&a_\alpha+2\sum_{\gamma\in H,\gamma\neq\alpha }a_\gamma&=0,\\
&\forall\alpha\in H_{-1},&-a_\alpha+2\sum_{\gamma\in H,\:\gamma\neq\alpha }a_\gamma&=0.
\end{align*}
Hence, after a convenient indexation, the vector $(a_\alpha)_{\alpha\in H}$ is a nonzero vector of the kernel of the matrix
\[\begin{pmatrix}
1&2&\ldots&2&
2&\ldots&\ldots&2\\
2&\ddots&\ddots&\vdots&
\vdots&&&\vdots\\
\vdots&&\ddots&2&
\vdots&&&\vdots\\
2&\ldots&2&1&
2&\ldots&\ldots&2\\
2&\ldots&\ldots&2&-1&2&\ldots&2\\
\vdots&&&\vdots&2&\ddots&\ddots&\vdots\\
\vdots&&&\vdots&\vdots&&\ddots&2\\
2&\ldots&\ldots&2&2&\ldots&2&-1
\end{pmatrix}\]
We leave to the reader the proof that this matrix is invertible: this is a contradiction, so $\leas(\K\Omega)$ has no nonzero weak special vector.\\

3. Let $\displaystyle a=\sum_{\alpha\in\Omega} a_\alpha a\in\K\Omega$.
\begin{align}
\nonumber&\mbox{$a$ is a weak special vector of the linearization of }\eas'(\Omega,\star)\\
\nonumber&\Longleftrightarrow\sum_{\alpha,\beta\in\Omega} a_\alpha a_\beta\beta\otimes\beta=\sum_{\alpha,\beta\in\Omega} a_\alpha a_\beta(\beta\otimes\alpha\star\beta^{\star -1}+\alpha\star\beta^{\star-1}\otimes\beta)\\
\label{EQ31} &\Longleftrightarrow\forall\alpha,\beta\in\Omega,\: a_\alpha a_\beta=a_{\beta\alpha}a_\alpha+a_{\alpha\beta}a_\beta.
\end{align}
Let $a$ be a nonzero weak special vector and let $\alpha\in\Omega$, such that $a_\alpha\neq 0$. For $\beta=\alpha^n$ in (\ref{EQ31}):
\[a_\alpha a_{\alpha^n}=a_{\alpha^{n+1}}a_\alpha+a_{\alpha^{n+1}}a_{\alpha^n}.\]
A direct induction proves that for any $n\geqslant 1$,
\[a_{\alpha^{n+1}}=\frac{1}{n}a_\alpha.\]
Consequently, if $\beta\in\Omega$, such that $a_\beta\neq 0$, then $a_\beta=3 a_{\beta^4}=a_{\beta^2}$. Moreover, applying this for $\alpha=\beta^2$, $a_{\beta^4}=a_{\beta^2}$. 
Finally, $a_\beta=3 a_{\beta^4}=a_{\beta^4}$, so $a_{\beta^4}=0$ and $a_\beta=0$, this is a contradiction. There is no nonzero weak special vector in this case.\end{proof}

\section{Hopf algebras and bialgebras from generalized pre-Lie algebras}

\subsection{Dual bialgebras}

As shown in\cite{Foissy55}, the symmetric algebra generated by the dual of the coinvariants space of the operad $\prelie_\Phi$ inherits a bialgebra structure, induced by the operadic composition.
In order to avoid unnecessary technical difficulties, we restrict ourselves to the case of a finite-dimensional dual $\ell$CEDS $(A,\Phi)$. 
We identify the graded dual $\g_A^*$ of $\g_A$ with $\g_{A^*}$, through the pairing defined in the following way: if $T$ is an $A^*$-typed and $\calD$-decorated tree and $T'$ is an $A$-typed tree and $\calD$-decorated tree, then
\[\langle T,T'\rangle=\sum_{\sigma\in\mathrm{Iso}(T,T')}\prod_{e\in E(T)} f_e(a_{\sigma(e)})\prod_{v\in V(T)}\delta_{d_v,d'_{\sigma(v)}},\]
with the following notations:
\begin{itemize}
\item $\mathrm{Iso}(T,T')$ is the set of isomorphisms of rooted trees from $T$ to $T'$. If $\sigma\in\mathrm{Iso}(T,T')$ and $e=(x,y)$ is an edge of $T$, we denote by $\sigma(e)=(\sigma(x),\sigma(y))$ the corresponding edge of $T'$.
\item For any $e\in E(T)$, $f_e$ is the type of $e$ in $T$ (and belongs to $A^*$); for any $e'\in E(T')$, $a_{e'}$ is the type of $e'$ in $T'$ (and belongs to $A$).
\end{itemize}

We shall use the following definitions:

\begin{defi}\label{defi4.1} 
We assume that $(\calD,\times)$ is a commutative semigroup. Let $F$ be an $A$-typed and $\calD$-decorated forest and $I$ be a subset of the set $E(F)$ of edges of $F$. 
\begin{itemize}
\item We denote by $F_{\mid I}$ the forest obtained by deleting all the edges of $F$which does not belong to $I$. As its set of vertices is the set of vertices of $F$ and its set of edges is $I$, it is naturally an $A$-typed and $\calD$-decorated forest.
\item We denote by $F_{/I}$ the forest obtained by contracting all the edges of $F$ which belong to $I$.
As its set of edges is $E(F)\setminus I$, it is naturally an $A$-typed forest. Any of its vertices $v$ can be seen as the contraction of a subtree $T_v$ of $F$: we decorate $v$ by the product in $\calD$ of the decorations of the vertices of $T_v$. 
With these decorations, $F_{/I}$ is an $A$-typed and $\calD$-decorated forest.
\end{itemize}
Note that $E(F)=E(F_{\mid I})\sqcup E(F_{/I})$. We make $F_{\mid I}\otimes F_{/I}$ a sum of tensors of $A$-typed and $\calD$-decorated forests in the following process:\\

For each edge $e\notin I$: let us denote by $e_1,\ldots,e_k$ the edges of $F$ in the path from the root of the tree containing $e$ to the first extremity of $e$ which belongs to $I$; denote by $a$ the type of $e$ and by $a_i$ the type of $e_i$. Compute
\begin{align*}
&\sum_{i=1}^p a'_i\otimes a'_{1,i}\otimes\ldots\otimes a'_{k,i}\\
&=\left(\Phi\otimes\id^{\otimes (k-1)}\right)\circ\left(\id\otimes\Phi\otimes\id^{\otimes (k-2)}\right)\circ\ldots\circ\left(\id^{\otimes (k-1)}\otimes\ldots\otimes\Phi\right)(a_1\otimes\ldots\otimes a_k\otimes a).
\end{align*}
and take the sum for $1\leq i\leq p$ of tensors of forests obtained by typing $e$ by $a'_i$ and $e_j$ by $a'_{j,i}$ for any $j$ in $F_{\mid I}\otimes F_{/I}$.\\

This sum of tensor products is denoted by $F_{\mid I}\otimes_\Phi F_{/I}$.
\end{defi}

Dualizing the combinatorial description of the operadic composition, we obtain a description of the dual bialgebra induced by the operad $\prelie_{\Phi^*}$:

\begin{theo}
Let $(\calD,\times)$ be an associative and commutative semigroup and $(A,\Phi)$ be a finite-dimensional dual $\ell$CEDS. Let us consider the symmetric algebra generated by $\g_{A,\calD}$, with its usual product $m$.
For any $A$-typed and $\calD$-decorated forest $F$, we put
\[\delta(F)=\sum_{I\subseteq E(F)} F_{\mid I}\otimes_\Phi F_{/I}.\]
Then $(S(\g_{A,\calD}),m,\delta)$ is a bialgebra, denoted by $B_{\Phi,(\calD,\times)}$.
\end{theo}

\begin{example}
Let $a,b\in A$ and $d,d',d''\in\calD$. 
\begin{align*}
\delta\left(\xymatrix{\rond{d}}\right)&=\xymatrix{\rond{d}}\otimes\xymatrix{\rond{d}},\\[3mm]
\delta\left(\substack{\hspace{3mm}\\ \xymatrix{\rond{d'}\ar@{-}[d]_a\\ \rond{d}}}\:\right)&=\substack{\hspace{3mm}\\ \xymatrix{\rond{d'}\ar@{-}[d]_a\\ \rond{d}}}\otimes\xymatrix{\rond{dd'}}+\xymatrix{\rond{d}}\:\xymatrix{\rond{d'}}\otimes\substack{\hspace{3mm}\\ 
\xymatrix{\rond{d'}\ar@{-}[d]_a\\ \rond{d}}},\\[3mm]
\delta\left(\substack{\hspace{5mm}\\ \xymatrix{\rond{d''}\ar@{-}[d]_b\\ \rond{d'}\ar@{-}[d]_a\\ \rond{d}}}\:\right)&=\xymatrix{\rond{d}}\:\xymatrix{\rond{d'}}\:\xymatrix{\rond{d''}}\otimes 
\substack{\hspace{5mm}\\ \xymatrix{\rond{d''}\ar@{-}[d]_b\\ \rond{d'}\ar@{-}[d]_a\\ \rond{d}}}
+\sum\substack{\hspace{3mm}\\ \xymatrix{\rond{d'}\ar@{-}[d]_{a''\triangleright b''}\\ \rond{d}}}\:\xymatrix{\rond{d''}}
\otimes\substack{\hspace{3mm}\\ \xymatrix{\rond{d''}\ar@{-}[d]_{a'\rightarrow b'}\\ \rond{dd'}}}\\
&+\substack{\hspace{3mm}\\ \xymatrix{\rond{d''}\ar@{-}[d]_b\\ \rond{d'}}}\:\xymatrix{\rond{d}}\otimes 
\substack{\hspace{3mm}\\ \xymatrix{\rond{d'd''}\ar@{-}[d]_a\\ \rond{d}} }
+\substack{\hspace{5mm}\\ \xymatrix{\rond{d''}\ar@{-}[d]_b\\ \rond{d'}\ar@{-}[d]_a\\ \rond{d}}}\otimes 
\xymatrix{\rond{dd'd''}},\\[3mm]
\delta\left(\substack{\hspace{3mm}\\ \xymatrix{\rond{d'}\ar@{-}[dr]_a&\rond{d''}\ar@{-}[d]_b\\ &\rond{d}}}\:\right)
&=\xymatrix{\rond{d}}\:\xymatrix{\rond{d'}}\:\xymatrix{\rond{d''}}\otimes 
\substack{\hspace{3mm}\\ \xymatrix{\rond{d'}\ar@{-}[dr]_a&\rond{d''}\ar@{-}[d]_b\\ &\rond{d}}}
+\substack{\hspace{3mm}\\ \xymatrix{\rond{d'}\ar@{-}[d]_a\\ \rond{d}}}\:\xymatrix{\rond{d''}}\otimes 
\substack{\hspace{3mm}\\ \xymatrix{\rond{d''}\ar@{-}[d]_b\\ \rond{dd'}}}\\
&+\substack{\hspace{3mm}\\ \xymatrix{\rond{d''}\ar@{-}[d]_b\\ \rond{d}}}\:\xymatrix{\rond{d'}}\otimes 
\substack{\hspace{3mm}\\ \xymatrix{\rond{d'}\ar@{-}[d]_a\\ \rond{dd''}} }
+\substack{\hspace{3mm}\\ \xymatrix{\rond{d'}\ar@{-}[dr]_a&\rond{d''}\ar@{-}[d]_b\\ &\rond{d}}}
\otimes\xymatrix{\rond{dd'd''}}.
\end{align*}\end{example}

\begin{prop}
Let $(\calD,\times)$ and $(\calD',\times')$ be two associative and commutative semigroups and $(A,\Phi)$, $(A',\Phi')$ be two finite-dimensional dual $\ell$CEDS such that $B_{\Phi,(\calD,\times)}$ and $B_{\Phi',(\calD',\times')}$ are isomorphic bialgebras. 
Then $(\calD,\times)$ and $(\calD',\times')$ are isomorphic and $A$ and $A'$ have the same dimension.
\end{prop}

\begin{proof}
We denote by $M$ the monoid of group-like elements of $B=B_{\Phi,(\calD,\times)}$. Its elements are the forests with no edge and vertices decorated by $\calD$ and its product is the disjoint union.
For any forest $F$ in $B$, $F_{\mid\emptyset}$ is an element of $M$ which we denote by $\pi_L(F)$, and $F_{/E(F)}$ is another element of $M$ which we denote by $\pi_R(F)$.
We denote by $B^+$ the subspace of $B$ generated by forests with at least one edge. Then, if $F$ is a forest of $B^+$,
\[\delta(F)-\pi_L(F)\otimes F-F\otimes\pi_R(F)\in B^+\otimes B^+.\]
For any $g,g'\in M$, we denote by $B_{g,g'}$ the subspace of $B$ generated by forests $F$ such that $\pi_L(F)=g$ and $\pi_R(F)=g'$ and we put
\[P_{g,g'}=\{x\in B_{\Phi,(\calD,\times)},\:\delta(x)=g\otimes x+x\otimes g'.\}.\]
We adopt similar notations for $B'=B_{\Phi',(\calD',\times')}$.\\

\textit{First step.} Let $g,g'\in M$. Let us prove that
\[P_{g,g'}=P_{g,g'}\cap B_{g,g'}\oplus\vect(g-g').\]
As $\delta(g-g')=g\otimes g-g'\otimes g'=g\otimes (g-g')+(g-g')\otimes g'$, $\supseteq$ is obvious. Let $x\in P_{g,g'}$. We write it under the form of a span of forests $x=\sum a_FF$. Then
\begin{align*}
\delta(x)&=g\otimes x+x\otimes g+\sum_{F\in B^+} a_F\left(\pi_L(F)\otimes F+F\otimes\pi_R(F)\right)+\sum_{F\in M} a_F F\otimes F+B^+\otimes B^+\\
&=\sum a_F (g\otimes F+F\otimes g').
\end{align*}
Identifying, if $F\in B^+$ such that $\pi_L(F)\neq g$ or $\pi_R(F)\neq g'$, then $a_F=0$. If $F\in M\setminus\{g,g'\}$, then $a_F=0$. Therefore, $P_{g,g'}\subseteq B_{g,g'}\oplus\vect(g,g')$.
If $g\neq g'$, considering the coefficient of $g\otimes g'$, we obtain $0=a_g+a_{g'}$, so $P_{g,g'}\subseteq B_{g,g'}\oplus\vect(g-g')$. If $g=g'$, considering the coefficient of $g\otimes g$, we obtain $2a_g=a_g$, so $a_g=0$. 
Hence, $P_{g,g}\subseteq B_{g,g}=B_{g,g}\oplus\vect(g-g)$. As $g-g'\in P_{g,g'}$, we obtain the announced result.\\
 
\textit{Second step}. Let $\Psi:B\longrightarrow B'$ be a bialgebra isomorphism. Then $\Psi$ sends any group-like element of $B$ on a group-like element of $B'$, so $\Psi_{\mid M}:M\longrightarrow M'$ is a bijection.
Looking at the indecomposable elements of the monoids $M$ and $M'$, we obtain that $\Psi$ sends any tree with only one $\calD$-decorated vertex to a tree with only one $\calD'$-decorated vertex, so $\Psi$ induces a bijection $\psi:\calD\longrightarrow\calD'$.
 Up to an isomorphism, we assume now that $\calD'=\calD'$ and that $\psi$ is the identity of $\calD$.
 
Let us consider three elements $d,d',d''$ of $\calD$. Let us consider the group-like elements $g_L=\xymatrix{\rond{d}}\xymatrix{\rond{d'}}$ and $g_R=\xymatrix{\rond{d''}}$.
Then, by the first step, noticing that $B_{gL,gR}$ is generated by trees with two vertices:
\begin{align*}
P_{g_L,g_R}&=\vect(g_L-g_R)\oplus\vect P_{g_L,g_R}\cap B_{g_L,g_R}\\
&=\vect(g_L-g_R)\oplus\begin{cases}
\vect\left(\substack{\hspace{2mm}\\ \xymatrix{\rond{d}\ar@{-}[d]_a\\ \rond{d'}}},
\substack{\hspace{2mm}\\ \xymatrix{\rond{d'}\ar@{-}[d]_a\\ \rond{d}}}, a\in A\right)\mbox{ if }d''=d\times d',\\
0\mbox{ otherwise}.
\end{cases}\end{align*} 
As $\Psi$ induces an isomorphism between $P_{g_L,g_R}$ and $P_{g_L,g_R}'$, we obtain that for any $(d,d',d'')\in\calD^3$,
\[d''=d\times d'\Longleftrightarrow d''=d'\times'd',\]
so $\Psi$ induces an isomorphism between $(\calD,\times)$ and $(\calD',\times')$.\\
 
Let us fix now an element $d\in\calD$, and let us consider the group-like elements $g_L=\xymatrix{\rond{d}}\xymatrix{\rond{d}}$ and $g_R=\xymatrix{\rond{d^2}}$. Then
\[P_{g_L,g_R}=P_{g_L,g_R}\cap B_{g_L,g_R}
 =\vect\left(\substack{\hspace{2mm}\\ \xymatrix{\rond{d}\ar@{-}[d]_a\\ \rond{d}}},a\in A\right).\]
 As $\Psi$ induces a bijection form $P_{g_L,g_R}$ to $P_{g_L,g_R}'$, it induced a bijection from $A$ to $A'$, so $A$ et $A'$ have the same dimension.\end{proof}

\begin{remark}
We conjecture that there exist two non isomorphic dual $\ell$EAS $(A,\Phi)$ and $(A',\Phi')$, such that $B_{\Phi,(\calD,\times)}$ and $B_{\Phi',(\calD,\times)}$ are isomorphic.
\end{remark}

\subsection{Cointeractions}

From\cite{Foissy55}, if $(A,\Phi)$ is a finite-dimensional dual $\ell$CEDS, then for any operad morphism $\prelie\longrightarrow\prelie_{\Phi^*}$, we obtain a pair of cointeracting bialgebras. 
We shall use the following definition, using the definition of admissible cuts and the notations of\cite{Connes1998}:

\begin{defi}
Let $(A,\Phi)$ be a dual $\ell$CEDS and $\calD$ be a set. Let us fix $f\in A^*$. Let $F$ be an $A$-typed and $\calD$-decorated forest, and let $c$ be an admissible cut of $F$. 
This admissible cuts give a tensor $P^c(F)\otimes R^c(F)$ of $\calD$-decorated forests. We make it a sum of tensors of $A$-typed and $\calD$-decorated forests by the following process:\\

For each edge $e\in c$, let us denote by $e_1,\ldots,e_k$ the edges of $F$ in the path from the root of the tree containing $e$ to the first extremity of $e$; denote by $a$ the type of $e$ and by $a_i$ the type of $e_i$. Compute
\begin{align*}
&\sum_{i=1}^p a'_{1,i}\otimes\ldots\otimes a'_{k,i}=(f\otimes\id^{\otimes k})\circ\\
&\left(\Phi\otimes\id^{\otimes (k-1)}\right)\circ\left(\id\otimes\Phi\otimes\id^{\otimes (k-2)}\right)\circ\ldots\circ\left(\id^{\otimes (k-1)}\otimes\ldots\otimes\Phi\right)(a_1\otimes\ldots\otimes a_k\otimes a).
\end{align*}
and take the sum for $1\leq i\leq p$ of tensors of forests obtained by typing $e_j$ by $a'_{j,i}$ for any $j$ in $P^c(F)\otimes R^c(F)$.\\

This sum of tensor products is denoted by $P^c(F)\otimes_{\Phi,f} R^c(F)$.
\end{defi}

From Proposition\ref{prop2.9}:

\begin{theo}
Let $(A,\Phi)$ be a finite-dimensional dual $\ell$CEDS and let $\calD$ be a set.
Let us consider the symmetric algebra generated by $\g_{A,\calD}$, with its usual product $m$, and let $f\in A^*$. For any $A$-typed and $\calD$-decorated forest $F$, we put
\[\Delta_f(F)=\sum_{c\in\adm(F)} P^c(F)\otimes_{\Phi,f} R^c(F).\]
If $(f\otimes f)\circ\Phi=f\otimes f$, then $(S(\g_{A,\calD}),m,\Delta_f)$ is a Hopf algebra, denoted by $H_{\Phi,\calD,f}$.
Moreover, if $(\calD,\times)$ is a commutative associative semigroup, then $H_{\Phi,\calD,f}$ is a Hopf algebra in the category of left $B_{\Phi,(\calD,\times)}$-comodules, with the coaction $\delta$.
\end{theo}

\begin{example}
Let $a,b,c\in A$ and $d,d',d'',d'''\in\calD$.
\begin{align*}
\Delta_f\left(\substack{\hspace{5mm}\\ \xymatrix{\rond{d''}\ar@{-}[rd]_b&\rond{d'''}\ar@{-}[d]_c\\
&\rond{d'}\ar@{-}[d]_a\\
&\rond{d}}}\:\right)
&=1\otimes\hspace{-5mm}\substack{\hspace{5mm}\\ \xymatrix{\rond{d''}\ar@{-}[rd]_b&\rond{d'''}\ar@{-}[d]_c\\
&\rond{d''}\ar@{-}[d]_a\\
&\rond{d'}}}+\substack{\hspace{5mm}\\ \xymatrix{\rond{d''}\ar@{-}[rd]_b&\rond{d'''}\ar@{-}[d]_c\\
&\rond{d'}\ar@{-}[d]_a\\
&\rond{d}}}\otimes 1+f(a)\substack{\hspace{3mm}\\ \xymatrix{\rond{d''}\ar@{-}[rd]_b&\rond{d'''}\ar@{-}[d]_c\\&\rond{d'}}}\otimes\xymatrix{\rond{d}}\\
&+\sum f(a'\rightarrow b')\xymatrix{\rond{d''}}\otimes\substack{\hspace{3mm}\\ \xymatrix{\rond{d'''}\ar@{-}[d]_c\\ 
\rond{d'}\ar@{-}[d]_{a''\triangleright b''}\\ \rond{d}}}+\sum f(a'\rightarrow c')\xymatrix{\rond{d'''}}\otimes\substack{\hspace{5mm}\\ 
\xymatrix{\rond{d''}\ar@{-}[d]_b\\ \rond{d'}\ar@{-}[d]_{a''\triangleright c''}\\ \rond{d}}}\\
&+\sum\sum f((a''\blacktriangleright b'')'\rightarrow c') f(a'\rightarrow b')\xymatrix{\rond{d''}}\:\xymatrix{\rond{d'''}}\otimes
\substack{\hspace{3mm}\\ \xymatrix{\rond{d'}\ar@{-}[d]^{(a''\triangleright b'')''\triangleright c''}\\ \rond{d}}}.
\end{align*}\end{example}

Dualizing Proposition\ref{prop2.10}:

\begin{prop}
Let $(A,\Phi)$ be a finite-dimensional dual $\ell$CEDS, $\calD$ be a set and let $f\in A^*$ such that $(f\otimes f)\circ\Phi=f\otimes f$.
 We denote by $\g_{A,\calD}$ be subspace of $\g_{A,\calD}$ generated by $\calD$-decorated and $A$-typed trees $T=B_d(a_1T_1,\ldots,a_kT_k)$ such that for any $i$, $f(a_i)=0$. 
Then the Hopf algebra $H_{\Phi,\calD,f}$ is isomorphic to the Connes-Kreimer Hopf algebra of non-typed, $\g_{A,\calD}'$-decorated rooted trees with its usual product of admissible cuts. 
\end{prop}

\begin{remark}
The coproduct $\Delta_0$ is coassociative and sends any tree $T$ to $T\otimes 1+1\otimes T$. 
\end{remark}

\subsection{Examples}

\begin{example}
Let $\Omega$ be a finite set and $\Omega'\subseteq\Omega$. We consider $\K\Omega$ and $\Phi$ defined by
\begin{align*}
&\forall\alpha,\beta\in\Omega,&\Phi(\alpha\otimes\beta)=\begin{cases}
\beta\otimes\alpha\mbox{ if }\beta\in\Omega',\\
0\mbox{ otherwise.}
\end{cases}\end{align*}
A direct computation shows that $(A,\Phi)$ is a dual $\ell$CEDS. In particular, if $\Omega'=\Omega$, we recover the linearization of $\eas(\Omega)$. 
Hence, for any $\alpha,\alpha_1,\ldots,\alpha_k\in\Omega$, with the notations of Definition\ref{defi4.1}, we obtain
\begin{align*}
&\sum_{i=1}^p\alpha'_i\otimes\alpha'_{1,i}\otimes\ldots\otimes\alpha'_{k,i}\\
&=\left(\Phi\otimes\id^{\otimes (k-1)}\right)\circ\left(\id\otimes\Phi\otimes\id^{\otimes (k-2)}\right)\circ\ldots\circ\left(\id^{\otimes (k-1)}\otimes\ldots\otimes\Phi\right)(\alpha_1\otimes\ldots\otimes\alpha_k\otimes\alpha)\\
&=\begin{cases}
\alpha\otimes\alpha_1\otimes\ldots\otimes\alpha_k\mbox{ if }\alpha\in\Omega'\mbox{or $k=0$},\\
0\mbox{ otherwise}.
\end{cases}\end{align*}
Consequently:

1. In $B_{\Phi,(\Omega,\times)}$, we obtain
\[\delta(F)=\sum_{\substack{I\subseteq E(F),\\t((E(F)\setminus I)')\subseteq\Omega'}} F_{\mid I}\otimes F_{/I},\]
where $(E(F)\setminus I)'$ is the set of edges $e$ in $E(F)\setminus I$ such that at least one edge in the path between a root of $F$ and the first extremity of $e$ is in $I$. For example,
\[\delta\left(\substack{\hspace{5mm}\\ \xymatrix{\rond{d''}\ar@{-}[d]_\beta\\ \rond{d'}\ar@{-}[d]_\alpha\\ \rond{d}}}
\:\right)=\xymatrix{\rond{d}}\:\xymatrix{\rond{d'}}\:\xymatrix{\rond{d''}}\otimes 
\substack{\hspace{5mm}\\ \xymatrix{\rond{d''}\ar@{-}[d]_\beta\\ \rond{d'}\ar@{-}[d]_\alpha\\ \rond{d}}}
+\indic_{\Omega'}(\beta)
\substack{\hspace{3mm}\\ \xymatrix{\rond{d'}\ar@{-}[d]_\alpha\\ \rond{d}}}\:\xymatrix{\rond{d''}}\otimes 
\substack{\hspace{3mm}\\ \xymatrix{\rond{d''}\ar@{-}[d]_\beta\\ \rond{dd'}} }
+\substack{\hspace{3mm}\\ \xymatrix{\rond{d''}\ar@{-}[d]_\beta\\ \rond{d'}}}\:
\xymatrix{\rond{d}}\otimes 
\substack{\hspace{3mm}\\ \xymatrix{\rond{d'd''}\ar@{-}[d]_\alpha\\ \rond{d}} }
+\substack{\hspace{5mm}\\ \xymatrix{\rond{d''}\ar@{-}[d]_\beta\\ \rond{d'}\ar@{-}[d]_\alpha\\ \rond{d}}}\otimes 
\xymatrix{\rond{dd'd''}},\]
where $\indic_{\Omega'}$ is the characteristic function of $\Omega'$.\\

2. Let $(\lambda_\alpha)_{\alpha\in\Omega}$ be any family of scalars such that $\lambda_\alpha=0$ if $\alpha\notin\Omega'$. If $f:\K\Omega\longrightarrow\K$ is defined by $f(\alpha)=\lambda_\alpha$ for any $\alpha\in\Omega$, then $(f\otimes f)\circ\Phi=f\otimes f$. 
For any forest $F$ and for any admissible cut $c$ of $F$, $P^c(F)$ and $R^c(F)$ are obtained by deleting the edges of $c$, the types of the remaining edges being untouched. Then
\[\Delta_f(F)=\sum_{c\in\mathrm{Adm}(F)}\prod_{e\in c}\lambda_{\mathrm{type}(e)} P^c(F)\otimes R^c(F),\]
For example,
\begin{align*}
\Delta_f\left(\substack{\hspace{5mm}\\ \xymatrix{\rond{d''}\ar@{-}[rd]_\beta&\rond{d'''}\ar@{-}[d]_\gamma\\
&\rond{d'}\ar@{-}[d]_\alpha\\
&\rond{d}}}\:\right)
&=1\otimes\hspace{-5mm}\substack{\hspace{5mm}\\ \xymatrix{\rond{d''}\ar@{-}[rd]_\beta&\rond{d'''}\ar@{-}[d]_\gamma\\
&\rond{d'}\ar@{-}[d]_\alpha\\
&\rond{d}}}+\substack{\hspace{5mm}\\ \xymatrix{\rond{d''}\ar@{-}[rd]_\beta&\rond{d'''}\ar@{-}[d]_\gamma\\
&\rond{d'}\ar@{-}[d]_\alpha\\
&\rond{d}}}\otimes 1+\lambda_\alpha
\substack{\hspace{3mm}\\ \xymatrix{\rond{d''}\ar@{-}[rd]_\beta&\rond{d'''}\ar@{-}[d]_\gamma\\&\rond{d'}}}
\otimes\xymatrix{\rond{d}}\\
&+\lambda_\beta
\xymatrix{\rond{d''}}\otimes\substack{\hspace{5mm}\\ \xymatrix{\rond{d'''}\ar@{-}[d]_\gamma\\ 
\rond{d'}\ar@{-}[d]_\alpha\\ \rond{d}}}+\lambda_\gamma
\xymatrix{\rond{d'''}}\otimes\substack{\hspace{5mm}\\
\xymatrix{\rond{d''}\ar@{-}[d]_\gamma\\ \rond{d'}\ar@{-}[d]_\alpha\\ \rond{d}}}
+\lambda_\beta\lambda_\gamma\xymatrix{\rond{d''}}\:\xymatrix{\rond{d'''}}\otimes
\substack{\hspace{3mm}\\ \xymatrix{\rond{d'}\ar@{-}[d]^\alpha\\ \rond{d}}}.
\end{align*}\end{example}

\begin{example} Let $(\Omega,\star)$ be a finite associative semigroup such that for any $\alpha,\beta,\gamma\in\Omega$,
\[(\alpha\star\beta)\star\gamma=(\beta\star\alpha)\star\gamma.\]
We shall take $A=\eas(\Omega,\star)^*$, which we identify as a vector space $A=\K\Omega$. For any $\alpha,\beta\in\Omega$, $\Phi^*(\alpha\otimes\beta)=\alpha\star\beta\otimes\alpha$, so
\[\Phi(\alpha\otimes\beta)=\sum_{\gamma\in\Omega,\:\beta\star\gamma=\alpha}\beta\otimes\gamma.\]
Hence, for any $\alpha,\alpha_1,\ldots,\alpha_k\in\Omega$, with the notations of Definition\ref{defi4.1}, we obtain
\begin{align*}
&\sum_{i=1}^p\alpha'_i\otimes\alpha'_{1,i}\otimes\ldots\otimes\alpha'_{k,i}\\
&=\left(\Phi\otimes\id^{\otimes (k-1)}\right)\circ\left(\id\otimes\Phi\otimes\id^{\otimes (k-2)}\right)\circ\ldots\circ\left(\id^{\otimes (k-1)}\otimes\ldots\otimes\Phi\right)(\alpha_1\otimes\ldots\otimes\alpha_k\otimes\alpha)\\
&=\sum_{i=1}^k\sum_{\gamma_i\in\Omega,\:\alpha_i\star\gamma_i=\alpha}\alpha\otimes\gamma_1\otimes\ldots\otimes\gamma_k.
\end{align*}
Consequently:

1. In $B_{\Phi,(\calD,\times)}$, if $F$ is an $\Omega$-typed and $\calD$-decorated forest and $I\subseteq E(F)$, 
$F_{\mid I,\star}$ is obtained by deleting all the edges which does not belong to $I$, the type of the remaining edges being modified in the following way:
if $e\in I$ and $e_1,\ldots,e_k$ begin the edges between a root of $F$ and $e$ which do not belong to $I$, $\alpha$ the type of $e$ and $\alpha_i$ the type of $e_i$ for any $i$, then the type of $e$ is replaced by
\[\sum_{\substack{\gamma\in\Omega,\\ \alpha_1\star\ldots\star\alpha_k\star\gamma=\alpha}}\gamma.\]
Moreover, $F_{/I}$ is obtained by contracting all the edges which belong to $I$, with the types of the remaining edges being untouched. We obtain
\[\delta(F)=\sum_{I\subseteq E(F)} F_{\mid I,\star}\otimes F_{/I}.\]
For example,
\begin{align*}
\delta\left(\substack{\hspace{5mm}\\ \xymatrix{\rond{d''}\ar@{-}[d]_\beta\\ \rond{d'}\ar@{-}[d]_\alpha\\ \rond{d}}}\:\right)
&=\xymatrix{\rond{d}}\:\xymatrix{\rond{d'}}\:\xymatrix{\rond{d''}}\otimes\substack{\hspace{5mm}\\ \xymatrix{\rond{d''}\ar@{-}[d]_\beta\\ \rond{d'}\ar@{-}[d]_\alpha\\ \rond{d}}}+\sum_{\alpha'\in\Omega,\:\beta\star\alpha'=\alpha} 
\substack{\hspace{3mm}\\ \xymatrix{\rond{d'}\ar@{-}[d]_{\alpha'}\\ \rond{d}}}\:
\xymatrix{\rond{d''}}\otimes 
\substack{\hspace{3mm}\\ \xymatrix{\rond{d''}\ar@{-}[d]_\beta\\ \rond{dd'}}}\\
&+\substack{\hspace{3mm}\\ \xymatrix{\rond{d''}\ar@{-}[d]_\beta\\ \rond{d'}}}\:\xymatrix{\rond{d}}\otimes 
\substack{\hspace{3mm}\\ \xymatrix{\rond{d'd''}\ar@{-}[d]_\alpha\\ \rond{d}} }
+\substack{\hspace{5mm}\\ \xymatrix{\rond{d''}\ar@{-}[d]_\beta\\ \rond{d'}\ar@{-}[d]_\alpha\\ \rond{d}}}\otimes 
\xymatrix{\rond{dd'd''}}.
\end{align*}
In the particular case where $\Omega$ is a commutative group,
\[\delta\left(\substack{\hspace{5mm}\\ \xymatrix{\rond{d''}\ar@{-}[d]_\beta\\ \rond{d'}\ar@{-}[d]_\alpha\\ \rond{d}}}\:\right)=\xymatrix{\rond{d}}\:\xymatrix{\rond{d'}}\:\xymatrix{\rond{d''}}\otimes 
\substack{\hspace{5mm}\\ \xymatrix{\rond{d''}\ar@{-}[d]_\beta\\ \rond{d'}\ar@{-}[d]_\alpha\\ \rond{d}}}
+\substack{\hspace{3mm}\\ \xymatrix{\rond{d'}\ar@{-}[d]_{\beta^{\star -1}\star\alpha}\\ \rond{d}}}\:\xymatrix{\rond{d''}}\otimes 
\substack{\hspace{3mm}\\ \xymatrix{\rond{d''}\ar@{-}[d]_\beta\\ \rond{dd'}} }+\substack{\hspace{3mm}\\ \xymatrix{\rond{d''}\ar@{-}[d]_\beta\\ \rond{d'}}}\:\xymatrix{\rond{d}}\otimes 
\substack{\hspace{3mm}\\ \xymatrix{\rond{d'd''}\ar@{-}[d]_\alpha\\ \rond{d}} }+\substack{\hspace{5mm}\\ \xymatrix{\rond{d''}\ar@{-}[d]_\beta\\ \rond{d'}\ar@{-}[d]_\alpha\\ \rond{d}}}\otimes 
\xymatrix{\rond{dd'd''}}.\]

2. Let $a$ be a special vector of eigenvalue 1 of $\eas(\Omega,\star)$ (see Example\ref{ex2.5}), written under the form 
\[a=\sum_{\alpha\in\Omega}\lambda_\alpha\alpha\in\K\Omega.\]
For any forest $F$ and for any admissible cut $c$ of $F$, $P^c(F)$ and $R_\star^c(F)$ are obtained by deleting the edges of $c$, the types of the remaining edges of $P^c(F)$ being untouched, 
whereas the types of the remaining edges of $R^c(F)$ are modified in the following way:
if $e$ is such an edge, $\alpha$ its type, $e_1,\ldots,e_k$ the edges $f$ of $F$ such that there exists a path from the last extremity of $e$ to the first extremity of $f$, $\alpha_1,\ldots,\alpha_k$ their types, then the type of $e$ is replaced by
\[\sum_{\substack{\gamma\in\Omega,\:\alpha_1\star\ldots\star\alpha_k\star\gamma=\alpha}}\gamma.\] 
Then
\[\Delta_a(F)=\sum_{c\in\mathrm{Adm}(F)}\prod_{e\in c}\lambda_{\mathrm{type}(e)} P^c(F)\otimes R_\star^c(F).\]
For example,
\begin{align*}
\Delta_a\left(\substack{\hspace{5mm}\\ \xymatrix{\rond{d''}\ar@{-}[rd]_\beta&\rond{d'''}\ar@{-}[d]_\gamma\\
&\rond{d''}\ar@{-}[d]_\alpha\\
&\rond{d}}}\:\right)
&=1\otimes\hspace{-5mm}\substack{\hspace{5mm}\\ \xymatrix{\rond{d''}\ar@{-}[rd]_\beta&\rond{d'''}\ar@{-}[d]_\gamma\\
&\rond{d''}\ar@{-}[d]_\alpha\\
&\rond{d}}}+\substack{\hspace{5mm}\\ \xymatrix{\rond{d''}\ar@{-}[rd]_\beta&\rond{d'''}\ar@{-}[d]_\gamma\\
&\rond{d'}\ar@{-}[d]\\
&\rond{d}}}\otimes 1+\lambda_\alpha
\substack{\hspace{3mm}\\ \xymatrix{\rond{d''}\ar@{-}[rd]_\beta&\rond{d'''}\ar@{-}[d]_\gamma\\&\rond{d'}}}\otimes\xymatrix{\rond{d}}\\
&+\sum_{\alpha'\in\Omega,\:\beta\star\alpha'=\alpha}\lambda_\beta\:\xymatrix{\rond{d''}}\otimes\substack{\hspace{5mm}\\
\xymatrix{\rond{d'''}\ar@{-}[d]_\gamma\\ \rond{d'}\ar@{-}[d]_{\alpha'}\\ \rond{d}}}
+\sum_{\alpha'\in\Omega,\:\gamma\star\alpha'=\alpha}\lambda_\gamma\:\xymatrix{\rond{d'''}}\otimes 
\substack{\hspace{5mm}\\ \xymatrix{\rond{d''}\ar@{-}[d]_\gamma\\ \rond{d'}\ar@{-}[d]_{\alpha'}\\ \rond{d}}}\\
&+\sum_{\alpha'\in\Omega,\:\gamma\star\beta\star\alpha'=\alpha}\lambda_\beta\lambda_\gamma\:
\xymatrix{\rond{d''}}\:\xymatrix{\rond{d'''}}\otimes\substack{\hspace{3mm}\\ \xymatrix{\rond{d'}\ar@{-}[d]^{\alpha'}\\ \rond{d}}}.
\end{align*}
In the special case where $\Omega$ is a commutative group, let us choose a subgroup $H$ of $\Omega$. We denote by $\indic_H$ the characteristic function of $H$. For any $\alpha\in\Omega$, we put $\lambda_\alpha=\indic_H(\alpha)$.
From\cite[Proposition 4.12]{Foissy46}, this defines a special vector of eigenvalue 1. Then
\begin{align*}
\Delta_a\left(\substack{\hspace{5mm}\\ \xymatrix{\rond{d''}\ar@{-}[rd]_\beta&\rond{d'''}\ar@{-}[d]_\gamma\\
&\rond{d'}\ar@{-}[d]_\alpha\\
&\rond{d}}}\:\right)
&=1\otimes\hspace{-5mm}\substack{\hspace{5mm}\\ \xymatrix{\rond{d''}\ar@{-}[rd]_\beta&\rond{d'''}\ar@{-}[d]_\gamma\\
&\rond{d'}\ar@{-}[d]_\alpha\\
&\rond{d}}}+\substack{\hspace{5mm}\\ \xymatrix{\rond{d''}\ar@{-}[rd]_\beta&\rond{d'''}\ar@{-}[d]_\gamma\\
&\rond{d'}\ar@{-}[d]\\
&\rond{d}}}\otimes 1+\indic_H(\alpha)\substack{\hspace{3mm}\\ \xymatrix{\rond{d''}\ar@{-}[rd]_\beta&\rond{d'''}\ar@{-}[d]_\gamma\\&\rond{d'}}}\otimes\xymatrix{\rond{d}}\\
&+\indic_H(\beta)\:\xymatrix{\rond{d''}}\otimes\hspace{-5mm}\substack{\hspace{5mm}\\ \xymatrix{\rond{d'''}\ar@{-}[d]_\gamma\\ 
\rond{d'}\ar@{-}[d]_{\beta^{\star-1}\star\alpha}\\ \rond{d}}}
+\indic_H(\gamma)\:\xymatrix{\rond{d'''}}\otimes\hspace{-5mm}\substack{\hspace{5mm}\\
\xymatrix{\rond{d''}\ar@{-}[d]_\gamma\\ \rond{d'}\ar@{-}[d]_{\gamma^{\star-1}\star\alpha}\\ \rond{d}}}\\
&+\indic_H(\beta)\indic_H(\gamma)\:\xymatrix{\rond{d''}}\:\xymatrix{\rond{d'''}}\otimes
\substack{\hspace{3mm}\\ \xymatrix{\rond{d'}\ar@{-}[d]^{\beta^{\star-1}\star\gamma^{\star-1}\star\alpha}\\ \rond{d}}}.
\end{align*}\end{example}

\begin{example} Let $(\Omega,\star)$ be a finite associative semigroup with the right inverse condition, and let us consider $(A,\Phi)=\K\eas'(\Omega,\star)^*$, identified with $\K\Omega$ as a vector space. For any $\alpha,\beta\in\Omega$,
\begin{align*}
\Phi(\alpha\otimes\beta)&=\beta\star\alpha\otimes\alpha.
\end{align*}
Hence, for any $\alpha,\alpha_1,\ldots,\alpha_k\in\Omega$, with the notations of Definition\ref{defi4.1},
we obtain
\begin{align*}
&\sum_{i=1}^p\alpha'_i\otimes\alpha'_{1,i}\otimes\ldots\otimes\alpha'_{k,i}\\
&=(\Phi\otimes\id^{\otimes (k-1)})\circ (\id\otimes\Phi\otimes\id^{\otimes (k-2)})\circ\ldots\circ (\id^{\otimes (k-1)}\otimes\ldots\otimes\Phi)(\alpha_1\otimes\ldots\otimes\alpha_k\otimes\alpha)\\
&=\alpha\star\alpha_k\star\ldots\star\alpha_1\otimes\alpha_1\otimes\ldots\otimes\alpha_k.
\end{align*}
Consequently:

1. In $B_{\Phi,(\calD,\times)}$, if $F$ is an $\Omega$-typed and $\calD$-decorated forest and $I\subseteq E(F)$, $F_{\mid I}$ is obtained by deleting all the edges which does not belong to $I$, the type of the remaining edges being untouched.
Moreover, $F_{/I,\star}$ is obtained by contracting all the edges which belong to $I$, with the types of the remaining edges being modified in the following way:
if $e\in I$ and $e_1,\ldots,e_k$ being the edges between the root of the tree of $F$ containing $e$ and $e$, $\alpha$ the type of $e$ and $\alpha_i$ the type of $e_i$ for any $i$, then the type of $e$ is replaced by
\[\alpha\star\alpha_k\star\ldots\star\alpha_1.\]
We obtain
\[\delta(F)=\sum_{I\subseteq E(F)} F_{\mid I}\otimes F_{/I,\star}.\]
For example,
\[\delta\left(\substack{\hspace{5mm}\\ \xymatrix{\rond{d''}\ar@{-}[d]_\beta\\ \rond{d'}\ar@{-}[d]_\alpha\\ \rond{d}}}\:\right)=\xymatrix{\rond{d}}\:\xymatrix{\rond{d'}}\:\xymatrix{\rond{d''}}\otimes 
\substack{\hspace{5mm}\\ \xymatrix{\rond{d''}\ar@{-}[d]_\beta\\ \rond{d'}\ar@{-}[d]_\alpha\\ \rond{d}}}+\substack{\hspace{3mm}\\ \xymatrix{\rond{d'}\ar@{-}[d]_\alpha\\ \rond{d}}}\:\xymatrix{\rond{d''}}\otimes 
\substack{\hspace{3mm}\\ \xymatrix{\rond{d''}\ar@{-}[d]_{\beta\star\alpha}\\ \rond{dd'}} }+\substack{\hspace{3mm}\\ \xymatrix{\rond{d''}\ar@{-}[d]_\beta\\ \rond{d'}}}\:\xymatrix{\rond{d}}\otimes 
\substack{\hspace{3mm}\\ \xymatrix{\rond{d'd''}\ar@{-}[d]_\alpha\\ \rond{d}} }+\substack{\hspace{5mm}\\ \xymatrix{\rond{d''}\ar@{-}[d]_\beta\\ \rond{d'}\ar@{-}[d]_\alpha\\ \rond{d}}}\otimes\xymatrix{\rond{dd'd''}}.\]

2. Let $a$ be a special vector of eigenvalue 1 of $A$ (see Example\ref{ex2.5}), written under the form
\[a=\sum_{\alpha\in\Omega}\lambda_\alpha\alpha\in\K\Omega.\]
For any forest $F$ and for any admissible cut $c$ of $F$, $P^c(F)$ and $R^c(F)$ are obtained by deleting the edges of $c$, the types of the remaining edges of $P^c(F)$ and $R^c(F)$ being untouched.
For any cut edge $e\in c$, let us denote by $\alpha_e$ its type, by $e_1,\ldots,e_k$ in the path from the root of the tree of $F$ containing $e$ to the first extremity of $e$ and $\alpha_1,\ldots,\alpha_k$ their types. We put
\[\alpha'_e=\alpha_e\star\alpha_k\star\ldots\star\alpha_1.\]
Then
\[\Delta_a(F)=\sum_{c\in\mathrm{Adm}(F)}\prod_{e\in c}\lambda_{\alpha'_e} P^c(F)\otimes R^c(F).\]
For example,
\begin{align*}
\Delta_a\left(\substack{\hspace{5mm}\\ \xymatrix{\rond{d''}\ar@{-}[rd]_\beta&\rond{d'''}\ar@{-}[d]_\gamma\\
&\rond{d'}\ar@{-}[d]_\alpha\\
&\rond{d}}}\:\right)
&=1\otimes\hspace{-5mm}\substack{\hspace{5mm}\\ \xymatrix{\rond{d''}\ar@{-}[rd]_\beta&\rond{d'''}\ar@{-}[d]_\gamma\\
&\rond{d'}\ar@{-}[d]_\alpha\\
&\rond{d}}}+\substack{\hspace{5mm}\\ \xymatrix{\rond{d''}\ar@{-}[rd]_\beta&\rond{d'''}\ar@{-}[d]_\gamma\\
&\rond{d'}\ar@{-}[d]_\alpha\\
&\rond{d}}}\otimes 1\\
&+\lambda_\alpha\substack{\hspace{3mm}\\ \xymatrix{\rond{d''}\ar@{-}[rd]_\beta&\rond{d'''}\ar@{-}[d]_\gamma\\&\rond{d'}}}\otimes\xymatrix{\rond{d}}
+\lambda_{\beta\star\alpha}\:\xymatrix{\rond{d''}}\otimes\substack{\hspace{5mm}\\ \xymatrix{\rond{d'''}\ar@{-}[d]_\gamma\\ \rond{d'}\ar@{-}[d]_\alpha\\ \rond{d}}}\\
&+\lambda_{\gamma\star\alpha}\:\xymatrix{\rond{d'''}}\otimes\substack{\hspace{5mm}\\ \xymatrix{\rond{d''}\ar@{-}[d]_\gamma\\ \rond{d'}\ar@{-}[d]_\alpha\\ \rond{d}}}
+\lambda_{\beta\star\alpha}\lambda_{\gamma\star\alpha}\:\xymatrix{\rond{d''}}\:\xymatrix{\rond{d'''}}\otimes\substack{\hspace{3mm}\\ \xymatrix{\rond{d'}\ar@{-}[d]^\alpha\\ \rond{d}}}.
\end{align*}
In the special case where $\Omega$ is a commutative group, let us choose a subgroup $H$ of $\Omega$. We denote by $\indic_H$ the characteristic function of $H$.
For any $\alpha\in\Omega$, we put $\lambda_\alpha=\indic_H(\alpha)$.From\cite[Proposition 4.12]{Foissy46}, this defines a special vector of eigenvalue 1, and therefore a pair of bialgebras in cointeraction.
\end{example}

\section*{Declarations}

\emph{Ethical Approval}. Not applicable.\\

\noindent\emph{Funding}. 
The author acknowledges support from the grant ANR-20-CE40-0007
\emph{Combinatoire Algébrique, Résurgence, Probabilités Libres et Opérades}.\\
 
\noindent\emph{Availability of data and materials}. Not applicable.

\bibliographystyle{amsplain}
\addcontentsline{toc}{section}{References}
\bibliography{biblio}

\providecommand{\bysame}{\leavevmode\hbox to3em{\hrulefill}\thinspace}
\providecommand{\MR}{\relax\ifhmode\unskip\space\fi MR }
\providecommand{\MRhref}[2]{%
  \href{http://www.ams.org/mathscinet-getitem?mr=#1}{#2}
}
\providecommand{\href}[2]{#2}
\begin{thebibliography}{10}

\bibitem{Aguiar2020}
Marcelo Aguiar, \emph{Dendriform algebras relative to a semigroup}, SIGMA
  Symmetry Integrability Geom. Methods Appl. \textbf{16} (2020), Paper No. 066,
  15.

\bibitem{Bruned2015}
Yvain Bruned, \emph{Singular {KPZ} {T}ype {E}quations}, PhD thesis, avalaible
  at \url{https://tel.archives-ouvertes.fr/tel-01306427v2/document}, 2015.

\bibitem{Bruned2019}
Yvain Bruned, Martin Hairer, and Lorenzo Zambotti, \emph{Algebraic
  renormalisation of regularity structures}, Invent. Math. \textbf{215} (2019),
  no.~3, 1039--1156.

\bibitem{Calaque2011}
Damien Calaque, Kurusch Ebrahimi-Fard, and Dominique Manchon, \emph{Two
  interacting {H}opf algebras of trees: a {H}opf-algebraic approach to
  composition and substitution of {B}-series}, Adv. in Appl. Math. \textbf{47}
  (2011), no.~2, 282--308.

\bibitem{Chapoton2001}
Fr\'{e}d\'{e}ric Chapoton and Muriel Livernet, \emph{Pre-{L}ie algebras and the
  rooted trees operad}, Internat. Math. Res. Notices (2001), no.~8, 395--408.

\bibitem{Connes1998}
Alain Connes and Dirk Kreimer, \emph{Hopf algebras, renormalization and
  noncommutative geometry}, Comm. Math. Phys. \textbf{199} (1998), no.~1,
  203--242.

\bibitem{Foissy55}
Lo{\"{\i}}c Foissy, \emph{Algebraic structures associated to operads},
  arXiv:1702.05344, 2017.

\bibitem{Foissy37}
\bysame, \emph{Commutative and non-commutative bialgebras of quasi-posets and
  applications to {E}hrhart polynomials}, Adv. Pure Appl. Math. \textbf{10}
  (2019), no.~1, 27--63.

\bibitem{Foissy58}
\bysame, \emph{Generalized dendrifom algebras and typed binary trees},
  arXiv:2002.12120, 2020.

\bibitem{Foissy47}
\bysame, \emph{Algebraic structures on typed decorated rooted trees}, SIGMA,
  Symmetry Integrability Geom. Methods Appl. \textbf{17} (2021), paper 086, 28.

\bibitem{Foissy36}
\bysame, \emph{{Chromatic polynomials and bialgebras of graphs}}, {Int.
  Electron. J. Algebra} \textbf{30} (2021), 116--167.

\bibitem{Foissy56}
\bysame, \emph{Generalized associative algebras}, arXiv:2104.00908, 2021.

\bibitem{Foissy46}
\bysame, \emph{On extended associative semigroups}, arXiv:2105.01326, 2021.

\bibitem{Foissy50}
Lo{\"\i}c Foissy, Dominique Manchon, and Yuanyuan Zhang, \emph{Families of
  algebraic structures}, arXiv:2005.05116, 2020.

\bibitem{Foissy57}
Lo{\"{\i}}c Foissy and Xiao-Song Peng, \emph{Typed angularly decorated planar
  rooted trees and generalized rota-baxter algebras}, arXiv:2112.02859, 2021.

\bibitem{Li2021}
Yu~Li, Qiuhui Mo, and Leonid~A. Bokut, \emph{Gr\"{o}bner--{S}hirshov bases for
  symmetric brace algebras}, Comm. Algebra \textbf{49} (2021), no.~2, 892--904.

\bibitem{Livernet2006}
Muriel Livernet, \emph{A rigidity theorem for pre-{L}ie algebras}, J. Pure
  Appl. Algebra \textbf{207} (2006), no.~1, 1--18.

\bibitem{Manchon2012}
Dominique Manchon, \emph{On bialgebras and {H}opf algebras or oriented graphs},
  Confluentes Math. \textbf{4} (2012), no.~1, 1240003, 10.

\bibitem{Manchon2020}
Dominique Manchon and Yuanyuan Zhang, \emph{Free pre-{L}ie family algebras},
  arXiv:2003.00917, 2020.

\bibitem{Sloane}
Neil J.~A. Sloane, \emph{The on-line encyclopedia of integer sequences},
  \url{https://oeis.org/}.

\bibitem{Zhang2020}
Yi~Zhang, Xing Gao, and Li~Guo, \emph{Matching {R}ota-{B}axter algebras,
  matching dendriform algebras and matching pre-{L}ie algebras}, J. Algebra
  \textbf{552} (2020), 134--170.

\bibitem{Zhang2019}
Yuanyuan Zhang and Xing Gao, \emph{Free {R}ota-{B}axter family algebras and
  (tri)dendriform family algebras}, Pacific J. Math. \textbf{301} (2019),
  no.~2, 741--766.

\bibitem{Zhang2020-3}
Yuanyuan Zhang, Xing Gao, and Dominique Manchon, \emph{Free (tri)dendriform
  family algebras}, J. Algebra \textbf{547} (2020), 456--493.

\end{thebibliography}

\end{document}